\documentclass[11pt]{amsart}

\pdfoutput=1

\usepackage{etex,amsfonts, amsmath, amsthm, amssymb, stmaryrd, epsfig, graphics,
  psfrag, latexsym, mathtools, mathrsfs,enumitem, subfig,
  longtable, booktabs, yfonts, centernot,ifthen,eso-pic} 
\usepackage[bbgreekl]{mathbbol}
\DeclareSymbolFontAlphabet{\mathbb}{AMSb}
\DeclareSymbolFontAlphabet{\mathbbl}{bbold}
\usepackage{xcolor}
\usepackage[all,cmtip]{xy}
\usepackage{tikz} \usetikzlibrary{matrix,arrows,shapes,calc}
\definecolor{darkblue}{rgb}{0,0,0.4} 
 \usepackage[colorlinks=true, citecolor=darkblue, filecolor=darkblue, linkcolor=darkblue,urlcolor=darkblue]{hyperref} 
\usepackage[all]{hypcap}
\usepackage{xr-hyper}

\usepackage[color=blue!20!white,textsize=tiny]{todonotes}

\graphicspath{{draws/}{}}

\numberwithin{equation}{section}

\newtheorem{thm}{Theorem}
\newtheorem{theorem}[thm]{Theorem}

\newtheorem{lem}{Lemma}[section]               
\newtheorem{lemma}[lem]{Lemma}

\newtheorem{corollary}[lem]{Corollary}               

\newtheorem{proposition}[lem]{Proposition}

\theoremstyle{definition}

\newtheorem{definition}[lem]{Definition}

\theoremstyle{remark}

\newtheorem{remark}[lem]{Remark}

\newtheorem{convention}[lem]{Convention}

\newcommand{\R}{\mathbb{R}}

\newcommand{\C}{\mathbb{C}}

\newcommand{\N}{\mathbb{N}}

\newcommand{\mc}{\mathcal}

\newcommand{\wt}{\widetilde}
\newcommand{\ol}{\overline}

\renewcommand{\emptyset}{\varnothing}

\newcommand{\Wmirror}[1]{\hat{#1}}

\newcommand{\smas}{\wedge}

\newcommand{\from}{\colon}
\newcommand{\into}{\hookrightarrow}

\renewcommand{\th}{^{\text{th}}}

\newcommand{\SLalg}{\mathfrak{sl}}
\DeclareMathOperator{\Inv}{Inv}

\renewcommand{\hat}{\widehat}

\DeclareMathOperator{\Diff}{Diff}

\DeclareMathOperator{\Ob}{Ob}
\DeclareMathOperator{\Id}{Id}
\DeclareMathOperator{\Sq}{Sq}

\DeclareMathOperator{\Hom}{Hom}

\newcommand{\Kh}{\mathit{Kh}}

\newcommand{\KhCx}{\mc{C}_{\mathit{Kh}}}

\newcommand{\AKh}{\mathit{AKh}}
\newcommand{\AKhCx}{\mc{C}_\mathit{AKh}}

\newcommand{\Cat}{\mathscr{C}}

\newcommand{\Dat}{\mathscr{D}}

\newcommand{\Realize}[2][{}]{|#2|_{#1}}
\newcommand{\CRealize}[2][{}]{\|#2\|_{#1}}

\newcommand{\Forget}{\mathcal{F}_{\!\!\scriptscriptstyle\mathit{orget}}}

\newcommand{\op}{\mathrm{op}}

\newcommand{\gr}{\mathrm{gr}}

\newcommand{\CubeCat}[1]{\underline{2}^{#1}}

\newcommand{\Filt}{\mathcal{F}}

\newcommand{\co}{\colon}
\newcommand{\bdy}{\partial}
\newcommand{\RR}{\R}
\newcommand{\CC}{\C}
\newcommand{\DD}{\mathbb{D}}

\newcommand{\pt}{\mathrm{pt}}

\newcommand{\NN}{\N}
\newcommand{\ZZ}{\mathbb{Z}}

\newcommand{\BurnsideCat}{\mathscr{B}}
\newcommand{\CCat}[1]{\CubeCat{#1}}

\newcommand{\GrAbelianGroups}{\ZZ^{\mathsf{Ab}}}

\newcommand{\Spectra}{\mathscr{S}}
\newcommand{\GrSpectra}{\ZZ^{\Spectra}}
\newcommand{\mSpectra}{\Spectra}
\newcommand{\GrmSpectra}{\GrSpectra}

\newcommand{\Complexes}{\mathsf{Kom}}
\newcommand{\GrComplexes}{\ZZ^{\mathsf{Kom}}}

\newcommand{\Total}[1]{\mathsf{Tot}(#1)}

\newcommand{\Sets}{\mathsf{Sets}}

\newcommand{\SphereS}{\mathbb{S}} 

\newcommand{\CobD}{\mathsf{Cob}_d}

\newcommand{\DeltaInj}{\Delta_{\mathit{inj}}}

\newcommand{\HKKa}{\mathit{HKK}}

\DeclareMathOperator{\hocolim}{hocolim}

\newcommand{\KTfunc}[1]{\mc{C}_{#1}}

\newcommand{\KTalg}[1]{\mathcal{A}_{#1}}
\newcommand{\KTalgOld}[1]{\mathcal{H}^{#1}}
\newcommand{\CKTalg}[3]{\mathcal{A}_{#1}^{#2,#3}}
\newcommand{\CKTalgT}[3]{\overline{\mathcal{A}}_{#1}^{#2,#3}}
\newcommand{\CKTalgB}[2]{\mathcal{A}_{#1}^{#2}}
\newcommand{\CKTalgTB}[2]{\overline{\mathcal{A}}_{#1}^{#2}}
\newcommand{\CKTalgBig}[1]{\mathcal{A}_{#1}^{\mathscr{P}}}
\newcommand{\CKTalgOld}[2]{\mathcal{A}^{#1,#2}}
\newcommand{\CKTalgBigOld}[1]{\mathcal{A}^{#1}}

\newcommand{\CKTideal}[3]{\mathcal{I}_{#1}^{#2,#3}}
\newcommand{\CKTidealB}[2]{\mathcal{I}_{#1}^{#2}}
\newcommand{\CKTfunc}[5]{\mc{C}_{#1}^{#2,#3;#4,#5}}
\newcommand{\CKTfuncB}[3]{\mc{C}_{#1}^{#2;#3}}
\newcommand{\CKTfuncBig}[1]{\mc{C}_{#1}^{\mathscr{P}}} 
\newcommand{\CKTfuncIdeal}[5]{\mathcal{J}_{#1}^{#2,#3;#4,#5}}
\newcommand{\CKTfuncIdealB}[3]{\mathcal{J}_{#1}^{#2;#3}}
\newcommand{\Crossingless}[1]{{\mathsf{B}}_{#1}}
\newcommand{\rCrossingless}[3]{{\mathsf{B}}_{#1}^{#2,#3}} 
\newcommand{\rCrossinglessB}[2]{{\mathsf{B}}_{#1}^{#2}} 

\newcommand{\bh}{\mathbf{h}}
\newcommand{\bk}{\mathbf{k}}
\newcommand{\bl}{\mathbf{l}}
\newcommand{\bbk}{\mathbbl{k}}
\newcommand{\mHshape}[1]{\wt{\mathcal{S}}_{#1}}
\newcommand{\mHshapeS}[1]{\mathcal{S}_{#1}}
\newcommand{\mCKAshape}[3]{\wt{\mathcal{S}}_{#1}^{#2,#3}}
\newcommand{\mCKAshapeS}[3]{\mathcal{S}_{#1}^{#2,#3}} 
\newcommand{\mCKAshapeB}[2]{\wt{\mathcal{S}}_{#1}^{#2}}
\newcommand{\mCKAshapeSB}[2]{\mathcal{S}_{#1}^{#2}} 
\newcommand{\mTshape}[2]{\wt{\mathcal{T}}_{#1;#2}}
\newcommand{\SmTshape}[2]{\mathcal{T}_{#1;#2}} 
\newcommand{\mCKTshape}[6]{\wt{\mathcal{T}}_{#1;#2}^{#3,#4;#5,#6}}
\newcommand{\SmCKTshape}[6]{\mathcal{T}_{#1;#2}^{#3,#4;#5,#6}} 
\newcommand{\mCKTshapeB}[4]{\wt{\mathcal{T}}_{#1;#2}^{#3;#4}} 
\newcommand{\SmCKTshapeB}[4]{\mathcal{T}_{#1;#2}^{#3;#4}} 
\newcommand{\mHinv}[1]{\Phi_{#1}}
\newcommand{\mHinvRes}[3]{\overline{\Phi}_{#1}^{#2,#3}}
\newcommand{\mHinvResB}[2]{\overline{\Phi}_{#1}^{#2}}
\newcommand{\mAinv}[3]{\Phi_{#1}^{#2,#3}}
\newcommand{\mAinvB}[2]{\Phi_{#1}^{#2}}

\newcommand{\mTinv}[1]{\Psi_{#1}}

\newcommand{\mTinvResB}[3]{\overline{\Psi}_{#1}^{#2;#3}}
\newcommand{\mCKTinv}[5]{\Psi_{#1}^{#2,#3;#4,#5}} 
\newcommand{\mCKTinvB}[3]{\Psi_{#1}^{#2;#3}} 
\newcommand{\mCKTinvBtemp}[3]{\tilde{\Psi}_{#1}^{#2;#3}}

\newcommand{\KTSpecCat}[1]{\mathscr{A}_{#1}}
\newcommand{\KTSpecRing}[1]{\mathscr{A}_{#1}^{\scriptscriptstyle\mathit{ring}}}
\newcommand{\KTSpecBim}[1]{\mathscr{X}_{#1}}
\newcommand{\KTSpecBimRing}[1]{\mathscr{X}_{#1}^{\scriptscriptstyle\mathit{\!mod}}}
\newcommand{\CKTSpecCat}[3]{\mathscr{A}_{#1}^{#2,#3}}
\newcommand{\CKTSpecCatB}[2]{\mathscr{A}_{#1}^{#2}}
\newcommand{\CKTSpecBim}[5]{\mathscr{X}_{#1}^{#2,#3;#4,#5}}
\newcommand{\CKTSpecBimB}[3]{\mathscr{X}_{#1}^{#2;#3}} 
\newcommand{\oCKTSpecBimB}[3]{\overline{\mathscr{X}}_{#1}^{#2;#3}} 
\newcommand{\CKTSpecBimBig}[1]{\mathscr{X}_{#1}^{\mathscr{P}}}
\newcommand{\CKTSpecCatBig}[1]{\mathscr{A}_{#1}^{\mathscr{P}}}

\newcommand{\KhSpace}[1]{\mathscr{X}_{#1}}
\newcommand{\AKhSpace}[1]{\mathscr{A\!X}_{\!#1}}

\newcommand{\mGlue}[3]{\wt{\mathcal{U}}_{#1;#2;#3}}
\newcommand{\mGlueS}[3]{\mathcal{U}_{#1;#2;#3}}
\newcommand{\mCKGlue}[3]{\wt{\mathcal{U}}_{#1;#2;#3}^{\mathbf{k}}}
\newcommand{\mCKGlueS}[3]{\mathcal{U}_{#1;#2;#3}^{\mathbf{k}}}
\newcommand{\mBurnside}{\BurnsideCat}
\newcommand{\ttimes}{\tilde{\times}}

\newcommand{\ulmHinv}[1]{\underline{\mathsf{MB}}_{#1}}
\newcommand{\ulmTinvNF}[1]{{\underline{\mathsf{MB}}}_{#1}}

\newcommand{\DTP}{\otimes^{\mathbb{L}}}

\newcommand{\Tan}{T}

\newcommand{\anclose}[1]{[#1]}
\newcommand{\flatclose}[1]{\langle #1\rangle}

\newcommand{\mI}{\mathsf{I}}
\newcommand{\mJ}{\mathsf{J}}
\newcommand{\mK}{\mathsf{K}}

\newcommand{\ou}[1]{\overline{\underline{#1}}}

\DeclareMathOperator{\THH}{THH}
\DeclareMathOperator{\HH}{HH}
\DeclareMathOperator{\HC}{HC}

\newcommand{\annulus}{\mkern4mu\circ\mkern-13.5mu\bigcirc}

\begin{document}

 
\title{Chen-Khovanov spectra for tangles}

\author{Tyler Lawson}
\thanks{\texttt{TL was supported by NSF FRG Grant DMS-1560699.}}
\email{\href{mailto:tlawson@math.umn.edu}{tlawson@math.umn.edu}}
\address{Department of Mathematics, University of Minnesota, Minneapolis, MN 55455}

\author{Robert Lipshitz}
\thanks{\texttt{RL was supported by NSF FRG Grant DMS-1560783.}}
\email{\href{mailto:lipshitz@uoregon.edu}{lipshitz@uoregon.edu}}
\address{Department of Mathematics, University of Oregon, Eugene, OR 97403}

\author{Sucharit Sarkar}
\thanks{\texttt{SS was supported by NSF FRG Grant DMS-1563615.}}
\email{\href{mailto:sucharit@math.ucla.edu}{sucharit@math.ucla.edu}}
\address{Department of Mathematics, University of California, Los Angeles, CA 90095}



\date{}

\begin{abstract}
  We note that our stable homotopy refinements of Khovanov's arc algebras and tangle invariants induce refinements of Chen-Khovanov and Stroppel's platform algebras and tangle invariants, and discuss the topological Hochschild homology of these refinements.
\end{abstract}
\maketitle
\vspace{-1cm}


\tableofcontents



\section{Introduction}\label{sec:intro}
In this paper, we continue our homotopical journey through Khovanov
homology by giving spectral refinements of Chen-Khovanov and
Stroppel's platform algebras.

Khovanov homology~\cite{Kho-kh-categorification} associates a bigraded abelian group $\Kh^{i,j}(L)$ to an oriented link $L\subset \RR^3$, so that the graded Euler characteristic of $\Kh^{i,j}(L)$ is the Jones polynomial. Khovanov~\cite{Kho-kh-tangles} extended his invariant to $(2m,2n)$-tangles by associating a graded algebra $\KTalgOld{n}$ to each non-negative integer $n$ and a chain complex of graded $(\KTalgOld{m},\KTalgOld{n})$-bimodules $\KTfunc{\Tan}$ to an oriented $(2m,2n)$-tangle diagram $\Tan$, such that:
\begin{enumerate}[label=(K-\arabic*),leftmargin=*]
\item\label{item:KT-prop-1} the chain homotopy type of $\KTfunc{\Tan}$ is an isotopy invariant of the tangle represented by $\Tan$,
\item\label{item:KT-prop-2} composition of tangles corresponds to the tensor product of bimodules, i.e., $\KTfunc{\Tan_1\Tan_2}\simeq \KTfunc{\Tan_1}\otimes_{\KTalgOld{n}}\KTfunc{\Tan_2}$, and
\item\label{item:KT-prop-3} $\KTalgOld{0}\cong\ZZ$ and for a $(0,0)$-tangle $L$, $H_*(\KTfunc{L})=\Kh(L)$.
\end{enumerate}
(See Table~\ref{tab:notation} for a dictionary of notation between this paper and the references. Note that our tangles are drawn left-to-right, not top-to-bottom as is more usual.)

Let $V$ be the fundamental (2-dimensional) representation of $U_q(\SLalg(2))$. 
Khovanov further showed that the Grothendieck group of the category of finitely-generated, graded $\KTalgOld{n}$-modules (and of the category of finitely-generated complexes of $\KTalgOld{n}$-modules) is canonically isomorphic to $\Inv(V^{\otimes 2n})$, the subspace of $U_q(\SLalg(2))$-invariants in $V^{\otimes 2n}$.

From a representation-theoretic standpoint, it is more interesting to
categorify $V^{\otimes n}$ itself. This was accomplished by
Chen-Khovanov and Stroppel~\cite{CK-kh-tangle,St-kh-tangle}, who
defined an algebra $\CKTalgBigOld{n}=\bigoplus_{k=0}^n\CKTalgOld{n-k}{k}$
for each $n\in\NN$ and associated to each $(m,n)$-tangle an
$(\CKTalgBigOld{m},\CKTalgBigOld{n})$-bimodule satisfying
Properties~\ref{item:KT-prop-1}--\ref{item:KT-prop-3} and so that the
Grothendieck group of graded, projective $\CKTalgBigOld{n}$-modules is
canonically isomorphic to $V^{\otimes n}$. These \emph{platform
  algebras} were further studied and generalized by Brundan-Stroppel~\cite{BS-kh-tangle}. A related, though more geometric, tangle invariant was also introduced by Bar-Natan~\cite{Bar-kh-tangle-cob} and, recently, another extension of Khovanov homology to tangles has been given by Roberts~\cite{Roberts-kh-tangle,Roberts-kh-A-tangle}.

In a previous paper~\cite{LLS-kh-tangles}, we gave stable homotopy refinements of Khovanov's algebras and modules. That is, for each integer $n$ we constructed a ring spectrum $\KTSpecRing{2n}$ and for each $(2m,2n)$-tangle diagram $\Tan$ an $(\KTSpecRing{2m},\KTSpecRing{2n})$-bimodule spectrum $\KTSpecBimRing{\Tan}$ such that:
\begin{enumerate}[label=(K-\arabic*),leftmargin=*]
\item\label{item:old-T-prop-1} the weak equivalence class of $\KTSpecBimRing{\Tan}$ is an isotopy invariant of the tangle represented by $\Tan$~\cite[Theorem 4]{LLS-kh-tangles},
\item\label{item:old-T-prop-2} composition of tangles corresponds to the tensor product of bimodule spectra, i.e., $\KTSpecBimRing{\Tan_1\Tan_2}\simeq \KTSpecBimRing{\Tan_1}\otimes_{\KTSpecRing{2n}}^L\KTSpecBimRing{\Tan_2}$~\cite[Theorem 5]{LLS-kh-tangles},
\item\label{item:old-T-prop-3} $\KTSpecRing{0}\simeq\SphereS$, the sphere spectrum, and for a $(0,0)$-tangle $L$, $\KTSpecBimRing{\Tan}$ is weakly equivalent to the previously-constructed~\cite{RS-khovanov,HKK-Kh-htpy,LLS-khovanov-product} Khovanov spectrum, and
\item\label{item:old-T-prop-4} the singular chain complex of $\KTSpecRing{2n}$ (respectively $\KTSpecBimRing{\Tan}$) is quasi-isomorphic to Khovanov's algebra $\KTalgOld{n}$ (respectively complex of bimodules $\KTfunc{\Tan}$)~\cite[Proposition 4.2]{LLS-kh-tangles}.
\end{enumerate}

In this paper, we modify $\KTSpecRing{2n}$ and $\KTSpecBimRing{\Tan}$ to give stable homotopy refinements of the platform algebras and modules.

This paper is organized as follows. Section~\ref{sec:CK-background}
reviews the arc algebras and bimodules and platform algebras and
bimodules. Section~\ref{sec:review-spectral-arc-algebra} reviews the
construction of the spectral refinements of the arc algebras and bimodules
(from~\cite{LLS-kh-tangles}).  Section~\ref{sec:spectral-CK}
constructs spectral refinements of the platform algebras and bimodules
and proves their basic properties. In Section~\ref{sec:THH} we show
that the topological Hochschild homology of the spectral platform
bimodules is homotopy equivalent to the naive spectral refinement of
annular Khovanov homology.

\subsection*{Acknowledgments} We thank Jesse Cohen, Aaron Lauda, Tony Licata, Andy
Manion, and Matthew Stoffregen for helpful conversations. We also
thank the referee for their comments.
\section{The platform algebras and modules}\label{sec:CK-background}
The platform algebras are subquotients of Khovanov's arc algebras. In
this section, we review both collections of algebras. For the platform algebras, we will
expand on some details in Chen-Khovanov's proofs, so that it is easier
to see how they adapt to the spectral case.

\subsection{Some notation}

In order to keep track of the quantum gradings throughout, it is convenient to work in graded versions of various well-known categories. We list them below.

\begin{itemize}[leftmargin=*]
\item Let $\ZZ^\Sets$ denote the category of finite, graded sets, whose
  objects are finite sets $X$ together with set maps $\gr\co X\to\ZZ$,
  and morphisms $f\co (X,\gr_X)\to (Y,\gr_Y)$ are set maps $f\co X\to
  Y$ so that $\gr_X=\gr_Y\circ f$.
\item Let $\Complexes$ be the category of freely generated chain
  complexes over $\ZZ$. (The Khovanov complex is usually presented as
  a cochain complex; we will view it as a chain complex by negating
  the homological grading; see \cite[\S
  2.1]{LLS-kh-tangles}.) Let $\GrComplexes$ denote the
  full subcategory of $\prod_\ZZ\Complexes$ where all but finitely
  many of the chain complexes are zero. So, $\GrComplexes$ is isomorphic to the
  category of bigraded chain complexes---the first grading being the
  homological grading and the second grading being an additional
  grading---that are bounded in the second grading.
\item Let $\GrAbelianGroups$ denote the category of freely and
  finitely generated graded abelian groups. We can (and will) identify
  $\GrAbelianGroups$ with the full subcategory of $\GrComplexes$
  with objects the finitely generated chain complexes supported in
  homological grading $0$.
\item Let $\Spectra$ be the category of symmetric spectra
  \cite{HSS-top-symmetric}. Let $\GrSpectra$ denote the full
  subcategory of $\prod_\ZZ\Spectra$ where all but finitely many of
  the spectra are trivial (that is, just the basepoint). Taking
  reduced singular chain complexes gives a functor $C_*\from
  \GrSpectra\to\GrComplexes$, cf.~\cite[\S2.7]{LLS-kh-tangles}.
\end{itemize}

Our notation for the arc and platform algebras and modules differs
slightly from Khovanov's and Chen-Khovanov's; see
Table~\ref{tab:notation} for a summary.

\begin{table}
  \centering
  \begin{tabular}[tab:notation]{llp{3.25in}}
    \toprule
    Our notation & \cite{Kho-kh-tangles,CK-kh-tangle} notation & Meaning\\
\midrule
    $V$ & $\mathcal{A}$ & Frobenius algebra $H^*(S^2)$.\\
    $\Crossingless{2n}$ & $B^n$ & Isotopy classes of crossingless matchings of $2n$ points.\\
    $\Wmirror{b}$ & $W(b)$ & Reflection of $b$.\\
    $\KTalg{2n}=\KTalgOld{n}$ & $H^n$ & Algebra associated to $n$ points.\\
    $\KTfunc{\Tan}$ & $V_D=\displaystyle\bigoplus_{v\in 2^N}\mathcal{F}(D(v))\{-v\}$ & Khovanov tangle invariant, a complex of $(\KTalg{2m},\KTalg{2n})$-bimodules.\\
    \midrule
    $\rCrossingless{n}{k_1}{k_2}=\rCrossinglessB{n}{\bk}$ & $B^{n-k,k}$ & Crossingless matchings of $n+k_1+k_2$ points with no matching among first $k_1$ or last $k_2$ points. Chen-Khovanov require $k_1+k_2=n$.\\
    $\CKTalgT{n}{k_1}{k_2}=\CKTalgTB{n}{\bk}$ & $\wt{A}^{n-k,k}$ & Subring of $\KTalg{2n}$ induced by $\rCrossingless{n}{k_1}{k_2}$.\\
    $\CKTideal{n}{k_1}{k_2}=\CKTidealB{n}{\bk}$ & $I^{n-k,k}$ & A particular ideal in $\CKTalgT{n}{k_1}{k_2}$.\\
    $\CKTalg{n}{k_1}{k_2}=\CKTalgB{n}{\bk}$ & $A^{n-k,k}$ & Platform algebra, $\CKTalgT{n}{k_1}{k_2}/\CKTideal{n}{k_1}{k_2}$.\\
    $\CKTalgBig{n}$ & $A^{n}$ & Total platform algebra, $\bigoplus_{k=0}^n\CKTalg{n}{n-k}{k}$.\\
    $\CKTfuncIdeal{\Tan}{h_1}{h_2}{k_1}{k_2}=\CKTfuncIdealB{\Tan}{\bh}{\bk}$ & $I(T)$ & A particular submodule of $\KTfunc{\ou{\Tan}}$.\\ 
    $\CKTfunc{\Tan}{h_1}{h_2}{k_1}{k_2}=\CKTfuncB{\Tan}{\bh}{\bk}$ & $\mathcal{F}(T)$ & Platform tangle invariant, $\KTfunc{\ou{\Tan}}/\CKTfuncIdeal{\Tan}{h_1}{h_2}{k_1}{k_2}$.\\
    $\CKTfuncBig{\Tan}$ & $\mathcal{F}(T)$ & Total platform bimodule, $\displaystyle\bigoplus_{\substack{h,k\\m-n=2(h-k)}}\!\!\!\CKTfunc{\Tan}{m-h}{h}{n-k}{k}$.\\
    \bottomrule
  \end{tabular}
  \caption{\textbf{Comparison of notation with Khovanov and Chen-Khovanov.}} 
  \label{tab:notation}
\end{table}

\subsection{Arc algebras and modules}
Let $V=H^*(S^2)=\ZZ[X]/(X^2)$ denote Khovanov's Frobenius algebra.
Explicitly, the comultiplication is given by $1\mapsto 1\otimes
X+X\otimes 1$ and $X\mapsto X\otimes X$, and the counit map is given
by $1\mapsto 0$ and $X\mapsto 1$. We view $V$ as a graded abelian
group with $1$ in grading $(-1)$ and $X$ in grading $1$; this grading
is called the quantum grading. (See \cite[Remark~2.55]{LLS-kh-tangles}
for a brief discussion of gradings.) Given a compact $1$-manifold $Z$,
let $V(Z)=\bigotimes_{\pi_0(Z)}V$, the tensor product over the
connected components of $Z$. Via the equivalence between Frobenius
algebras and 2-dimensional topological field theories, we can view $V$
as a topological field theory which assigns $V$ to the circle.

We prefer to view the arc and platform algebras as linear
categories, rather than algebras. Given a category $\Cat$, we will
write $\Cat(a,b)$ for $\Hom_{\Cat}(a,b)$. Then Khovanov's \emph{arc
  algebra} $\KTalgOld{n}$ is the category with objects the crossingless
matchings of $2n$ points, $\Crossingless{2n}$, and
\[
  \KTalgOld{n}(a,b)\coloneqq V(a\Wmirror{b})\{n\},
\]
where $\Wmirror{b}$ denotes the horizontal reflection of $b$ and
$\{n\}$ denotes an upward quantum grading shift by $n$.  Composition
is induced by the \emph{canonical cobordisms} $\Wmirror{b}\amalg b\to
\Id$, which gives cobordisms $a\Wmirror{b}\amalg b\Wmirror{c}\to
a\Wmirror{c}$, and the topological field theory associated to $V$.

To keep notation consistent later in this paper, let
\[
  \KTalg{2n}=\KTalgOld{n}.
\]

Given a flat $(2m,2n)$-tangle $\Tan$
there is an induced $(\KTalg{2m},\KTalg{2n})$-bimodule, i.e., functor of
linear categories
\[
  \KTfunc{\Tan}\co (\KTalg{2m})^{\op}\times\KTalg{2n}\to \GrAbelianGroups,
\]
which on objects is given by
$\KTfunc{\Tan}(a,b)=V(a\Tan\Wmirror{b})\{n\}$ and on morphisms is
induced by the canonical cobordisms. More generally, given a non-flat
oriented tangle diagram $\Tan$, there is a functor
\[
  \CCat{N}\times (\KTalg{2m})^{\op}\times\KTalg{2n}\to\GrAbelianGroups
\]
given on objects by
\[
  (v,a,b)\mapsto \KTfunc{\Tan_v}(a,b)\{-|v|-N_++2N_-\}=V(a\Tan_v\Wmirror{b})\{n-|v|-N_++2N_-\}.
\]
Here, $\CCat{N}=\{0\to 1\}^N$, and $N_+,N_-,N$ are the number of
positive, negative, and total crossings in $\Tan$, respectively. This induces a functor
\[
  \KTfunc{\Tan}\co (\KTalg{2m})^{\op}\times\KTalg{2n}\to \GrComplexes,
\]
by taking iterated mapping cones along $\CCat{N}$, and then shifting
the homological grading down by $N_+$.

\subsection{Platform algebras}\label{sec:platform-algebras}
Let $\rCrossingless{n}{k_1}{k_2}\subset\Crossingless{n+k_1+k_2}$ be
the set of crossingless matchings with no matchings among the first
$k_1$ points and no matchings among the last $k_2$ points. So,
$\Crossingless{2n}=\rCrossingless{2n}{0}{0}$. The cases used for the
platform algebras and modules are $\rCrossingless{n}{n-k}{k}$ (for
$0\leq k\leq n$), but the definitions generalize to
$\rCrossingless{n}{k_1}{k_2}$ for any $k_1+k_2\equiv n\pmod{2}$, and
it is convenient to work in the more general setting.

Given $a,b\in\Crossingless{n+k_1+k_2}$, for each $1\leq i\leq
n+k_1+k_2$ there is a corresponding circle $Z_i$ of $a\Wmirror{b}$
containing the point $i$. This induces an equivalence relation
$\sim_{a,b}$ on $1,\dots,n+k_1+k_2$ by $i\sim_{a,b}j$ if and only if
$Z_i=Z_j$. (Equivalently, this equivalence relation is generated by
$i\sim j$ if $i$ is matched to $j$ in either $a$ or $b$.) So,
$a\in\rCrossingless{n}{k_1}{k_2}$ if and only if $i\not\sim_{a,a}j$
whenever $i<j\leq k_1$ or $i>j>n+k_1$.

For $a,b\in\rCrossingless{n}{k_1}{k_2}$ define
$\CKTideal{n}{k_1}{k_2}(a,b)\subset \KTalg{n+k_1+k_2}(a,b)$
by
\begin{enumerate}[label=(I-\arabic*),leftmargin=*]
\item\label{item:alg-III} $\CKTideal{n}{k_1}{k_2}(a,b)=\KTalg{n+k_1+k_2}(a,b)$ if there is some pair
  $i<j\leq k_1$ or $i>j>n+k_1$ so that $i\sim_{a,b}j$.  
\item\label{item:alg-II} the span of the set of generators in $\KTalg{n+k_1+k_2}(a,b)$ which label some circle $Z_i$ with $i\leq k_1$ or $i>n+k_1$ by $X$. 
\end{enumerate}
Then $\CKTalg{n}{k_1}{k_2}$ is the category with objects $\rCrossingless{n}{k_1}{k_2}$ and 
\[
  \CKTalg{n}{k_1}{k_2}(a,b)=\KTalg{n+k_1+k_2}(a,b)/\CKTideal{n}{k_1}{k_2}(a,b).
\]
In other words, if we let $\CKTalgT{n}{k_1}{k_2}$ be the full
subcategory of $\KTalg{n+k_1+k_2}$ spanned by objects in
$\rCrossingless{n}{k_1}{k_2}$ then 
\[
  \CKTalg{n}{k_1}{k_2}=\CKTalgT{n}{k_1}{k_2}/\CKTideal{n}{k_1}{k_2}.
\]

Chen-Khovanov encode the points $1,\dots,k_1$ and $n+k_1+1,\dots,n+k_1+k_2$ by drawing two vertical line segments, \emph{platforms}, in $a\Wmirror{b}$ where $a$ and $\Wmirror{b}$ meet, one containing the first $k_1$ points and the other containing the last $k_2$ points. Then, the points in Case~\ref{item:alg-III} lie on a circle that meets one of the platforms more than once; following Chen-Khovanov, we will sometimes call such a circle a \emph{type III circle}. The points in Case~\ref{item:alg-II} lie on a circle that meets at least one platform, which we will sometimes follow Chen-Khovanov in calling a \emph{type II circle} (if it is not a type III circle).

There is an inclusion
\[
  \imath\co \Crossingless{n+k_1+k_2}\into \Crossingless{n+k_1+k_2+2}
\]
where $\imath(a)$ is obtained by matching the first and last points,
and then matching the remaining points by $a$ (shifted up by
$1$). (That is, if we think of $a$ as an involution of $\{1,\dots,n+k_1+k_2\}$ then
$\imath(a)(1)=n+k_1+k_2+2$ and $\imath(a)(i)=a(i-1)+1$ if $2\leq i\leq n+k_1+k_2+1.$) The map
$\imath$ sends $\rCrossingless{n}{k_1}{k_2}$ to
$\rCrossingless{n}{k_1+1}{k_2+1}$. If $k_1+k_2\geq n$ then $\imath\co
\rCrossingless{n}{k_1}{k_2}\to \rCrossingless{n}{k_1+1}{k_2+1}$ is a bijection.

There is an induced map
\[
  \imath\co \KTalg{n+k_1+k_2}\to \KTalg{n+k_1+k_2+2}
\]
which labels the new circle in $\imath(a)\Wmirror{\imath(b)}$
(containing the points $1$ and $n+k_1+k_2+2$) by $1$.

\begin{remark}\label{rem:cat-inclusion}
  As a map of rings, $\imath$ does not send the unit to the unit, but viewing
  $\KTalg{n+k_1+k_2}$ and $\KTalg{n+k_1+k_2+2}$ as linear categories, $\imath$
  corresponds to the inclusion of a full subcategory.
\end{remark}

\begin{lemma}\label{lem:CK-ideal}
  Given $n,k_1,k_2$ with $k_1+k_2\equiv n\pmod{2}$,
  $\CKTideal{n}{k_1}{k_2}\subset\CKTalgT{n}{k_1}{k_2}$ is a
  2-sided ideal, so $\CKTalg{n}{k_1}{k_2}$ is a linear
  category.
\end{lemma}
\begin{proof}
  The proof is spelled out by
  Chen-Khovanov~\cite[Proof of Lemma 1]{CK-kh-tangle}, but we repeat
  it here for completeness. Suppose $x$ is a labeling of
  $a\Wmirror{b}$ so that some circle passing through a platform, at a
  point $P$, is labeled $X$. Let $y$ be any labeling of
  $b\Wmirror{c}$. In the sequence of merges and splits relating
  $a\Wmirror{b}\amalg b\Wmirror{c}$ to $a\Wmirror{c}$, the circle
  containing $P$ is always labeled $X$, so, in particular, in the
  product $xy$ a circle passing through a platform is labeled $X$. If
  there is a circle in $a\Wmirror{b}$ which passes through two points,
  $P$ and $Q$, on the same platform then in the cobordism from
  $a\Wmirror{b}\amalg b\Wmirror{c}$ to $a\Wmirror{c}$, either $P$ and
  $Q$ stay on the same circle throughout, in which case $a\Wmirror{c}$
  has a type III circle, or some split occurs to the circle containing
  $P$ and $Q$, in which case either $P$ or $Q$ is labeled $X$ after
  the split. In the latter case, the circle containing that point
  continues to be labeled $X$ throughout the cobordism, giving a type
  II or III circle labeled $X$.
\end{proof}

\begin{lemma}\label{lem:iota-func-iso}
  Given $n,k_1,k_2$ with $k_1+k_2\equiv n\pmod{2}$ and
  $a,b\in\rCrossingless{n}{k_1}{k_2}=\Ob(\CKTalg{n}{k_1}{k_2})$,
  \[
    \imath^{-1}(\CKTideal{n}{k_1+1}{k_2+1}(\imath(a),\imath(b)))=\CKTideal{n}{k_1}{k_2}(a,b).
  \]
  So, $\imath$ induces a homomorphism (functor of linear categories)
  \[
    \imath\co \CKTalg{n}{k_1}{k_2}\to \CKTalg{n}{k_1+1}{k_2+1}
  \]
  which is always full and faithful and injective on objects, and is
  an isomorphism if $k_1+k_2\geq n$.
\end{lemma}
\begin{proof}
  This is immediate from the definitions.
\end{proof}

\begin{figure}
  \centering
  \includegraphics{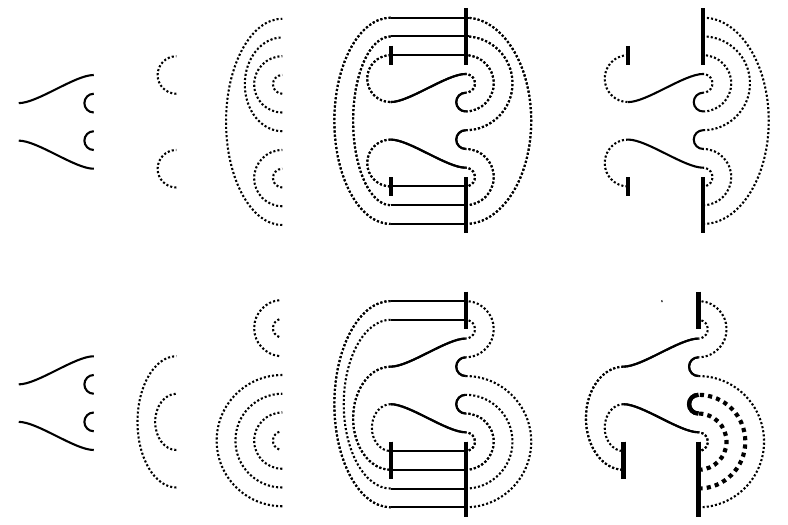}
  \caption{\textbf{Capping off a flat tangle.} Each row, left to
    right: a tangle $\protect\Tan$, crossingless matchings $a,b$, the
    closure $\flatclose{a\Tan\protect\Wmirror{b}}$,
    and the partial closure $a\protect\Tan\protect\Wmirror{b}$. On the
    first row, $m=2$, $n=6$, $h_1=1$, $h_2=1$, $k_1=3$, and $k_2=3$. On the second row, $m=2$,
    $n=6$, $h_1=2$, $h_2=0$, $k_1=4$, and $k_2=2$. On the first row, there are no arcs
    with both endpoints on the same platform (type III arcs); on the
    second row, there is one, drawn in bold.}
  \label{fig:cap-flat}
\end{figure}

\subsection{Platform modules}

Given an $(m,n)$-flat tangle $\Tan$, Chen-Khovanov define an
$(\CKTalg{m}{m-h}{h},\CKTalg{n}{n-k}{k})$-bimodule for all
$(h,k)$ with $0\leq h\leq m$, $0\leq k\leq n$, and
$m-n=2(h-k)$. Their construction extends immediately to give
$(\CKTalg{m}{h_1}{h_2},\CKTalg{n}{k_1}{k_2})$-bimodules
for any $h_1,h_2,k_1,k_2$, with $h_1-h_2=k_1-k_2$, and
some details are easier in the more general setting, so we sketch
their construction there.

Fix $a\in\rCrossingless{m}{h_1}{h_2}$ and
$b\in\rCrossingless{n}{k_1}{k_2}$ with $h_1-h_2=k_1-k_2$,
and an $(m,n)$-flat tangle $\Tan$. Assume that $k_1-h_1\geq 0$; the
other case is symmetric. Form a closed $1$-manifold
$\flatclose{a\Tan\Wmirror{b}}$ by:
\begin{enumerate}[leftmargin=*]
\item Adding $k_1$ horizontal strands below $\Tan$ and $k_2$
  horizontal strands above $T$, to obtain $\ou{\Tan}$.
\item Gluing $\imath^{k_1-h_1}(a)$ to $\ou{\Tan}$ to $\Wmirror{b}$, to obtain
  \[
    \flatclose{a\Tan\Wmirror{b}}=\imath^{k_1-h_1}(a)\thinspace\ou{\Tan}\thinspace\Wmirror{b}.
  \]
\end{enumerate}
See Figure~\ref{fig:cap-flat}. There is a subset
$a\Tan\Wmirror{b}\subset \flatclose{a\Tan\Wmirror{b}}$ which we call the
\emph{partial closure}. The endpoints of the arcs in
$a\Tan\Wmirror{b}$ are on four \emph{platforms}.

Define a submodule $\CKTfuncIdeal{\Tan}{h_1}{h_2}{k_1}{k_2}(a,b)\subset \KTfunc{\ou{\Tan}}(\imath^{k_1-h_1}(a),b)$ by declaring that 
\begin{enumerate}[label=(J-\arabic*),leftmargin=*]
\item\label{item:tan-III} If $a\Tan\Wmirror{b}$ has an arc with both ends on the same platform then $\CKTfuncIdeal{\Tan}{h_1}{h_2}{k_1}{k_2}(a,b)=\KTfunc{\ou{\Tan}}(\imath^{k_1-h_1}(a),b)$. 
\item\label{item:tan-II} Otherwise, $\CKTfuncIdeal{\Tan}{h_1}{h_2}{k_1}{k_2}(a,b)$ is spanned by the generators of $\KTfunc{\ou{\Tan}}(\imath^{k_1-h_1}(a),b)$ which label at least one arc component of the partial closure $a\Tan\Wmirror{b}$ by $X$.
\end{enumerate}
Define 
\[
\CKTfunc{\Tan}{h_1}{h_2}{k_1}{k_2}(a,b)=\KTfunc{\ou{\Tan}}(\imath^{k_1-h_1}(a),b)/\CKTfuncIdeal{\Tan}{h_1}{h_2}{k_1}{k_2}(a,b).
\]

Following Chen-Khovanov, we will call the arcs in
Case~\ref{item:tan-III} \emph{type III arcs} and the arcs in
Case~\ref{item:tan-II} \emph{type II arcs} (if they are not
type III arcs). We call a circle in $\flatclose{a\Tan\Wmirror{b}}$ containing a
type III arc a \emph{type III circle} and a circle in
$\flatclose{a\Tan\Wmirror{b}}$ containing a type II arc a \emph{type II circle}
(if it is not a type III circle). The following reformulation
of these conditions will be useful:

\begin{lemma}\label{lem:type-III}
  A circle $Z$ in $\flatclose{a\Tan\Wmirror{b}}$ is a type III circle if and
  only if $Z$ intersects some platform more than once. A circle $Z$ in
  $\flatclose{a\Tan\Wmirror{b}}$ is a type II circle if and only if $Z$
  intersects each platform at most once and some platform at least
  once.
\end{lemma}
\begin{proof}
  We start with the statement about type III circles. That a type III
  circle intersects some platform more than once is immediate from the
  definition. For the other direction, observe first that $Z$ contains
  a type III arc if and only if there is an embedded bigon in $\RR^2$
  with one edge on $Z$, one edge on a platform, and interior disjoint
  from $Z$ and the platforms. Let $P$ be the union of the top-left and
  bottom-left platforms.  Since the two top platforms are connected by
  horizontal lines in $\ou{\Tan}$, as are the two bottom platforms,
  existence of a bigon as above is equivalent to existence of a bigon
  with one edge on $Z$, one edge on $P$, and interior disjoint from
  $Z\cup P$.

  \begin{figure}
    \centering
    \includegraphics{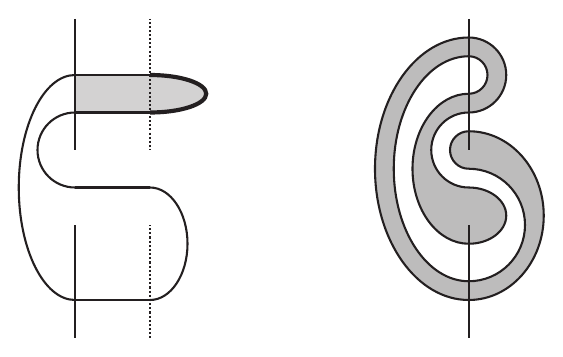}
    \caption{\textbf{Type III circles and arcs.} Left: a bigon
      (shaded) and the corresponding  type III arc (bold). The
      platforms not in $P$ are dotted. Right: a
      case where both bigons in $D$ (shaded) contain tips.}
    \label{fig:type-3}
  \end{figure}
  
  Let $D$ be the disk inside $Z$, i.e., the bounded region of
  $\RR^2\setminus Z$. The platforms $P$ cut $D$ into a collection of
  polygons, two of which might contain the \emph{tips of $P$}---the
  segments of $P\setminus Z$ adjacent to the endpoints of $P$. (See
  Figure~\ref{fig:type-3}.) Let $P'$ be the complement of the tips in
  $P$.  Since $D$ intersects some platform more than once,
  $D\setminus P'$ has at least two connected components. Recall that the Euler measure of a $2n$-gon is $1-n/2$. Every polygon
  in $D\setminus P'$ has an even number of edges. Consequently, the
  only polygons in $D\setminus P'$ with positive Euler measure are
  bigons. Since the Euler measure of $P'$ is $1$ and Euler measure is
  additive, at least two of the components of $D\setminus P'$ must be
  bigons, and if $D\setminus P'$ has exactly two bigons then all other
  components of $D\setminus P'$ are rectangles.

  If some component of $D\setminus P'$ is a bigon not containing a tip
  of $P$ then we are done. In the remaining case, $D\setminus P'$
  consists of two bigons, both of which contain tips of $P$, and some
  number of rectangles.

  In this last case, consider the complement $D'=\RR^2\setminus
  D$. View the platforms as extending to infinity in $\RR^2$. Since
  $D$ contains both tips, the platforms $P$ cut the punctured disk
  $D'$ into two non-compact regions and some polygons. At least one of
  the non-compact regions has more than $2$ corners, for otherwise,
  $Z$ will be a circle passing through each platform in $P$ once,
  contradicting the hypothesis. It follows by considering the Euler measure that
  at least one of the regions in $D'$ is a bigon, which implies that
  there is a type III arc.

  The statement about type II circles is immediate from the statement
  about type III circles.
\end{proof}

\begin{remark}
  In the special case that $h_1=h_2=k_1=k_2=0$,
  $\CKTfuncIdeal{\Tan}{0}{0}{0}{0}(a,b)=\{0\}$, so
  $\CKTfunc{\Tan}{0}{0}{0}{0}=\KTfunc{\Tan}$, the arc algebra
  bimodule. The case considered by Chen-Khovanov is $h_2=m-h_1$ and
  $k_2=n-k_1$.
\end{remark}

\begin{convention}\label{conv:bk}
To shorten notation, we will often write $\bh=(h_1,h_2)$ and $\bk=(k_1,k_2)$, so
\begin{align*}
  \CKTalgB{m}{\bh}&=\CKTalg{m}{h_1}{h_2} &
  \CKTfuncIdealB{\Tan}{\bh}{\bk}&=\CKTfuncIdeal{\Tan}{h_1}{h_2}{k_1}{k_2} &
  \CKTfuncB{\Tan}{\bh}{\bk}&=\CKTfunc{\Tan}{h_1}{h_2}{k_1}{k_2}.
\end{align*}
\end{convention}

Composing the functor
$\imath^{k_1-h_1}\co \CKTalgTB{m}{\bh}\into \CKTalgTB{m}{\bk}$ with
the inclusion $\jmath\co \CKTalgTB{m}{\bk}\into \KTalg{m+k_1+k_2}$
gives a functor $\jmath\circ \imath^{k_1-h_1}$ of linear categories. Given an
$(\KTalg{m+k_1+k_2},\KTalg{n+k_1+k_2})$-bimodule $M$ (regarded as a functor), we can restrict
along the functor $\jmath\circ \imath^{k_1-h_1}\otimes \jmath$ to obtain an
$(\CKTalgTB{m}{\bh},\CKTalgTB{n}{\bk})$-bimodule
$(\jmath\circ \imath^{k_1-h_1}\otimes \jmath)^*M$. (Compare Remark~\ref{rem:cat-inclusion}.)

\begin{proposition}\cite{CK-kh-tangle}\label{prop:CK-Ideal}
  The subsets
  \[
    \CKTfuncIdealB{\Tan}{\bh}{\bk}(a,b)\subset \KTfunc{\ou{\Tan}}(\imath^{k_1-h_1}(a),b)
  \]
  form a submodule of the
  $(\CKTalgTB{m}{\bh},\CKTalgTB{n}{\bk})$-bimodule
  $(\jmath\circ \imath^{k_1-h_1}\otimes \jmath)^*\KTfunc{\ou{\Tan}}$, and
  \begin{align*}
  \CKTidealB{m}{\bh}\cdot \bigl((\jmath\circ \imath^{k_1-h_1}\otimes \jmath)^*\KTfunc{\ou{\Tan}}\bigr)&\subset \CKTfuncIdealB{\Tan}{\bh}{\bk}\\
  \bigl((\jmath\circ \imath^{k_1-h_1}\otimes \jmath)^*\KTfunc{\ou{\Tan}}\bigr)\cdot \CKTidealB{n}{\bk}&\subset \CKTfuncIdealB{\Tan}{\bh}{\bk},
  \end{align*}
  so $\CKTfuncB{\Tan}{\bh}{\bk}$ inherits the structure of an
  $(\CKTalgB{m}{\bh},\CKTalgB{n}{\bk})$-bimodule.
\end{proposition}
\begin{proof}
  The same argument as in the proof of Lemma~\ref{lem:CK-ideal} shows
  that the subsets
  \[
    \CKTfuncIdealB{\Tan}{\bk}{\bk}(a,b)\subset \KTfunc{\ou{\Tan}}(a,b),
  \]
  for $a\in\rCrossinglessB{m}{\bk}$ and
  $b\in\rCrossinglessB{n}{\bk}$, form a submodule of
  $\KTfunc{\ou{\Tan}}$ and that 
  \begin{align*}
  \CKTidealB{m}{\bk}\cdot \bigl((\jmath\otimes \jmath)^*\KTfunc{\ou{\Tan}}\bigr)&\subset \CKTfuncIdealB{\Tan}{\bk}{\bk}\\
  \bigl((\jmath\otimes \jmath)^*\KTfunc{\ou{\Tan}}\bigr)\cdot \CKTidealB{n}{\bk}&\subset \CKTfuncIdealB{\Tan}{\bk}{\bk}.
  \end{align*}
  By Lemma~\ref{lem:iota-func-iso},
  $(\imath^{k_1-h_1})^{-1}(\CKTidealB{m}{\bk}) =
  \CKTidealB{m}{\bh}$. The result follows.
\end{proof}

\begin{proposition}\cite{CK-kh-tangle}\label{prop:CK-cube-maps-descend}
  Given a non-flat tangle $\Tan$, the maps in the cube of resolutions
  \[\KTfunc{\ou{\Tan_v}}(\imath^{k_1-h_1}(a),b)\to
  \KTfunc{\ou{\Tan_w}}(\imath^{k_1-h_1}(a),b)\]
  send
  $\CKTfuncIdealB{\Tan_v}{\bh}{\bk}(a,b)$ to
  $\CKTfuncIdealB{\Tan_w}{\bh}{\bk}(a,b)$ and hence descend to homomorphisms (natural transformations)
  \[\CKTfuncB{\Tan_v}{\bh}{\bk}\{-|v|\}\to
    \CKTfuncB{\Tan_w}{\bh}{\bk}\{-|w|\}.\]
\end{proposition}
\begin{proof}
  Suppose $y$ is a labeling of
  $\imath^{k_1-h_1}(a)\ou{\Tan_w}\Wmirror{b}$ and $x$ is a labeling
  of $\imath^{k_1-h_1}(a)\ou{\Tan_v}\Wmirror{b}$ so that $(w,y)$
  occurs as a term in $\bdy(v,x)$. If there is a circle in
  $\imath^{k_1-h_1}(a)\ou{\Tan_v}\Wmirror{b}$ passing through a
  point $P$ on a platform which is labeled $X$ by $x$ (a type II
  circle labeled $X$) then the circle in
  $\imath^{k_1-h_1}(a)\ou{\Tan_w}\Wmirror{b}$ containing $P$ is also
  labeled $X$.

  The more interesting case is that there is an arc in the partial
  closure $a\Tan_v\Wmirror{b}$ with both endpoints on the same
  platform (a type III arc). Let $Z\subset \flatclose{a\Tan_v\Wmirror{b}}$ be
  the corresponding circle. By Lemma~\ref{lem:type-III}, if
  $\flatclose{a\Tan_w\Wmirror{b}}$ does not have a type III circle then the
  circle $Z$ must split into two circles, each of which intersects a
  platform. Since the split map sends $X$ to $X\otimes X$ and $1$ to
  $1\otimes X+X\otimes 1$, one of these circles must be labeled $X$ by
  $y$, as desired.
\end{proof}

\begin{definition}
  By Propositions~\ref{prop:CK-Ideal}
  and~\ref{prop:CK-cube-maps-descend}, associated to a non-flat tangle
  $\Tan$ is a cube of $(\CKTalgB{m}{\bh},\CKTalgB{n}{\bk})$-bimodules:
  $v\mapsto \CKTfuncB{\Tan_v}{\bh}{\bk}\{-|v|-N_++2N_-\}$. Let
  $\CKTfuncB{\Tan}{\bh}{\bk}$ be the total complex of this cube, with
  homological grading shifted down by $N_+$.
\end{definition}

\begin{lemma}\label{lem:restrict-invt}
  There are isomorphisms
  \[
    \CKTfunc{\Tan}{h_1}{h_2}{k_1}{k_2}\cong 
    (\imath\otimes\Id)^*\CKTfunc{\Tan}{h_1+1}{h_2+1}{k_1}{k_2}\cong 
    (\Id\otimes\imath)^*\CKTfunc{\Tan}{h_1}{h_2}{k_1+1}{k_2+1}.
  \]
\end{lemma}
\begin{proof}
  This is immediate from the definitions.
\end{proof}

\begin{theorem}\cite{CK-kh-tangle}\label{thm:comb-invariance}
  Up to homotopy equivalence of chain complexes of bimodules,
  $\CKTfuncB{\Tan}{\bh}{\bk}$ is invariant under Reidemeister
  moves.
\end{theorem}
\begin{proof}
  By Lemma~\ref{lem:restrict-invt}, it suffices to prove the result
  when $\bh=\bk$. We will spell out the proof for a Reidemeister II
  move; the cases of Reidemeister I moves and braid-like Reidemeister
  III moves (see, e.g.,~\cite[\S7.3]{Baldwin-hf-s-seq}
  or~\cite[\S6]{RS-khovanov}) are similar, and we comment on
  them briefly at the end of the proof.

  \begin{figure}
    \centering
    \begin{tikzpicture}
      \tikzset{knot/.style={solid,thick}}
      \tikzset{resols/.style={inner sep=0pt,outer sep=2pt}}
      
      \foreach \a [count=\i from 0] in {a,b,c,d,e}{
        
        \begin{scope}[xshift=\i*2.8cm,xscale=1.3, yscale=2]
          
          \node at (0.5,-0.6) {(\a)};
          
          \ifthenelse{\i=0\OR\i=2\OR\i=3}{
            \node[resols] (a00) at (0,0) {\begin{tikzpicture}[scale=0.4]
                \draw[knot] (-0.75,-1.75) -- (-0.5,-1.5) to[out=45,in=-45] (-0.5,-0.5) to[out=135,in=-135] (-0.5,0.5) to[out=45,in=135] (0.5,0.5) to[out=-45,in=45] (0.5,-0.5) to[out=-135,in=135] (0.5,-1.5) -- (0.75,-1.75);
                \draw[knot] (-0.75,1.75) -- (-0.5,1.5) to[out=-45,in=-135] (0.5,1.5) -- (0.75,1.75);
              \end{tikzpicture}};}{}
          
          \ifthenelse{\i=0\OR\i=1\OR\i=2\OR\i=3}{
            \node[resols] (a10) at (1,0) {\begin{tikzpicture}[scale=0.4]
                \draw[knot] (-0.75,-1.75) -- (-0.5,-1.5) to[out=45,in=135] (0.5,-1.5) -- (0.75,-1.75);
                \draw[knot] (-0.5,-0.5) to[out=135,in=-135] (-0.5,0.5) to[out=45,in=135] (0.5,0.5) to[out=-45,in=45] (0.5,-0.5) to[out=-135,in=-45] (-0.5,-0.5);
                \draw[knot] (-0.75,1.75) -- (-0.5,1.5) to[out=-45,in=-135] (0.5,1.5) -- (0.75,1.75);
                
                \ifthenelse{\i=1}{
                  \node at (0,0) {$1$};}{}
                \ifthenelse{\i=2\OR\i=3}{
                  \node at (0,0) {$X$};}{}
              \end{tikzpicture}};}{}
          
          \ifthenelse{\i=0\OR\i=2\OR\i=4}{
            \node[resols] (a01) at (0,1) {\begin{tikzpicture}[scale=0.4]
                \draw[knot] (-0.75,-1.75) -- (-0.5,-1.5) to[out=45,in=-45] (-0.5,-0.5) to[out=135,in=-135] (-0.5,0.5) to[out=45,in=-45] (-0.5,1.5) --(-0.75,1.75);
                \draw[knot] (0.75,-1.75) -- (0.5,-1.5) to[out=135,in=-135] (0.5,-0.5) to[out=45,in=-45] (0.5,0.5) to[out=135,in=-135] (0.5,1.5) --(0.75,1.75);
              \end{tikzpicture}};}{}
          
          \ifthenelse{\i=0\OR\i=1}{
            \node[resols] (a11) at (1,1) {\begin{tikzpicture}[scale=0.4]
                \draw[knot] (-0.75,1.75) -- (-0.5,1.5) to[out=-45,in=45] (-0.5,0.5) to[out=-135,in=135] (-0.5,-0.5) to[out=-45,in=-135] (0.5,-0.5) to[out=45,in=-45] (0.5,0.5) to[out=135,in=-135] (0.5,1.5) -- (0.75,1.75);
                \draw[knot] (-0.75,-1.75) -- (-0.5,-1.5) to[out=45,in=135] (0.5,-1.5) -- (0.75,-1.75);
                
              \end{tikzpicture}};}{}
          
          \ifthenelse{\i=0\OR\i=2}{
            \draw[->] (a00) -- (a01);}{}
          \ifthenelse{\i=0\OR\i=2\OR\i=3}{
            \draw[->] (a00) -- (a10);}{}
          \ifthenelse{\i=0}{
            \draw[->] (a01) -- (a11);}{}
          \ifthenelse{\i=0\OR\i=1}{
            \draw[->] (a10) -- (a11);}{}
          
        \end{scope}}
      
    \end{tikzpicture}
    \caption{\textbf{Invariance under Reidemeister II moves.} (a)
      $\KTfunc{\ou{\Tan}}(a,b)$, the square of resolutions at a
      Reidemeister II move. (b) The acyclic subcomplex $C_1(a,b)$. (c)
      The quotient complex $C_2(a,b)$.  (d) The acyclic quotient
      complex $C_3(a,b)$ of $C_2(a,b)$. (e) The corresponding
      subcomplex $C_4(a,b)\cong\KTfunc{\ou{\Tan'}}(a,b)$.}
    \label{fig:Reidemeister-II}
  \end{figure}
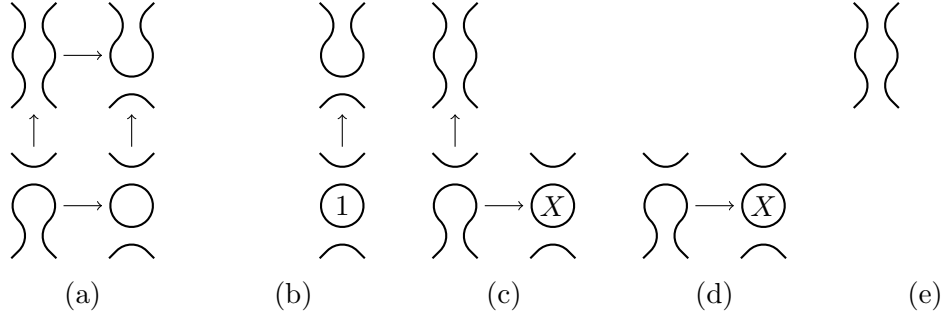
  
  Suppose $\Tan$ and $\Tan'$ are related by a Reidemeister II move,
  and $\Tan$ has two more crossings than $\Tan'$. Given
  $a\in\rCrossinglessB{m}{\bk}$, $b\in\rCrossinglessB{n}{\bk}$,
  Figure~\ref{fig:Reidemeister-II}~(a) shows the complex
  $\KTfunc{\ou{\Tan}}(a,b)$, where the two crossings involved in the
  Reidemeister II move are resolved in the four possible
  ways. Figure~\ref{fig:Reidemeister-II}~(b) shows an acyclic
  subcomplex $C_1(a,b)$ consisting of two of the resolutions, and where
  one closed circle is labeled $1$. The quotient complex
  $C_2(a,b)=\KTfunc{\ou{\Tan}}(a,b)/C_1(a,b)$ (shown in
  Figure~\ref{fig:Reidemeister-II}~(c)) has a subcomplex $C_4(a,b)$
  (shown in Figure~\ref{fig:Reidemeister-II}~(e)), so that
  $C_3(a,b)=C_2(a,b)/C_4(a,b)$ (shown in
  Figure~\ref{fig:Reidemeister-II}~(d)) is acyclic, and $C_4(a,b)\cong
  \KTfunc{\ou{\Tan'}}(a,b)$.

  Since each $C_i(a,b)$ is defined by restricting to certain vertices
  of the cube and certain labels of the closed circles in the
  resolutions of $\Tan$ (not $a\ou{\Tan}\Wmirror{b}$), each $C_i$
  restricts to an
  $(\CKTalgTB{m}{\bk},\CKTalgTB{n}{\bk})$-bimodule. The inclusion maps
  $C_1(a,b)\into \KTfunc{\ou{\Tan}}(a,b)$ and $C_4(a,b)\into
  C_2(a,b)$, and the identification $C_4(a,b)\cong
  \KTfunc{\ou{\Tan'}}(a,b)$ respect the labels of type II and III
  circles. Hence, if we let
  $C'_i(a,b)=C_i(a,b)/\CKTfuncIdealB{\Tan}{\bk}{\bk}$ then there are
  induced maps $C'_1\into \CKTfuncB{\Tan}{\bk}{\bk}$, $C'_4\into
  C'_2$, and $C'_4\cong \CKTfuncB{\Tan'}{\bk}{\bk}$.

  It remains to verify that $C'_1$ and $C'_3$ are acyclic. Let
  $\Tan_v$ and $\Tan_w$ be resolutions of $\Tan$ which agree at all of
  the crossings not involved in the Reidemeister II move and have the
  two forms allowed in $C_1$ near the Reidemeister II move. Notice
  that $a\ou{\Tan_v}\Wmirror{b}$ has a type III arc if and only if
  $a\ou{\Tan_w}\Wmirror{b}$ does, and the type II arcs of
  $a\ou{\Tan_v}\Wmirror{b}$ and $a\ou{\Tan_w}\Wmirror{b}$
  correspond. It follows that $C'_1$ is acyclic. Similarly, if
  $\Tan_v$ and $\Tan_w$ are resolutions as in $C_3$ which agree away
  from the Reidemeister II move, the type II (respectively III) arcs
  of $\Tan_v$ correspond to the type II (respectively III) arcs of
  $\Tan_w$, so it follows that $C'_3$ is acyclic. This completes the
  proof for Reidemeister II moves.

  As Chen-Khovanov note, there are two key points to this
  argument. First, the sub-complexes involved are defined locally near
  the Reidemeister move by placing restrictions on the labels of
  closed circles only, hence descend to the Chen-Khovanov quotient.
  Second, the acyclic subcomplexes stay acyclic after one quotients by
  the Chen-Khovanov submodule, essentially because type II and III
  arcs at the different relevant vertices correspond. Inspecting the
  proofs of Reidemeister I and III invariance (see,
  e.g.,~\cite[\S6]{RS-khovanov}), both properties hold for them as well.
  Indeed, in all cases
  there are two kinds of cancellations that occur: merging on a circle
  labeled $1$ or splitting off a circle labeled $X$. Neither operation
  changes the labels of the other components nor which points in the
  boundary of the tangle lie on the same strand, so each of the sub-
  and quotient complexes remains acyclic after quotienting by the
  Chen-Khovanov submodule. This completes the proof.
\end{proof}

The gluing theorem for the tangle invariants only holds in
Chen-Khovanov's generality:
\begin{theorem}\cite{CK-kh-tangle}\label{thm:CK-pairing}
  Let $\Tan_1$ be an $(m,n)$-tangle and $\Tan_2$ an
  $(n,p)$-tangle. Fix integers $h_1,h_2,k_1,k_2,\ell_1,\ell_2$
  satisfying
  \[
    h_1-h_2=k_1-k_2=\ell_1-\ell_2
  \]
  and
  \[
    h_1+h_2\geq m,\ k_1+k_2\geq n,\ \ell_1+\ell_2\geq p.
  \]
  Then,
  \[
    \CKTfuncB{\Tan_1\Tan_2}{\bh}{\bl}\simeq \CKTfuncB{\Tan_1}{\bh}{\bk}\DTP_{\CKTalgB{n}{\bk}}\CKTfuncB{\Tan_2}{\bk}{\bl}.
  \]
\end{theorem}

As Chen-Khovanov note, the proof of Theorem~\ref{thm:CK-pairing} is essentially
the same as the arc algebra case~\cite[Proposition 13]{Kho-kh-tangles}.  For the
spectral refinements, we need an explicit description of the homotopy
equivalence in Theorem~\ref{thm:CK-pairing}, so as usual we give a few more
details.

Let $\Tan_1$ be an $(m,n)$-tangle and $\Tan_2$ an $(n,p)$-tangle. By
Lemma~\ref{lem:iota-func-iso}, Theorem~\ref{thm:CK-pairing} is equivalent to the
same statement with $(h_1,h_2)$ replaced by $(h_1+1,h_2+1)$ (or $\bk$, $\bl$
increased similarly). Thus, for the rest of the section, we will assume that
\begin{align*}
  h_1&=k_1=\ell_1\\
   \shortintertext{so also}
  h_2&=k_2=\ell_2.
\end{align*}

Given
$a\in \rCrossinglessB{m}{\bk}$,
$b\in \rCrossinglessB{n}{\bk}$, and
$c\in \rCrossinglessB{p}{\bk}$,
the canonical cobordism from $\Wmirror{b}\amalg b$ to the
identity braid induces a map
\begin{equation}
  \wt{\Psi}(a,b,c)\co \KTfunc{\ou{\Tan_1}}(a,b)\otimes_\ZZ\KTfunc{\ou{\Tan_2}}(b,c)
  \to \KTfunc{\ou{\Tan_1\Tan_2}}(a,c).
\end{equation}

\begin{lemma}\label{lem:glue-is-chain-map}
  The gluing map $\wt{\Psi}$ is a chain map of bimodules. Further, for
  $a \in \rCrossinglessB{m}{\bk}$,
  $b\in \rCrossinglessB{n}{\bk}$, and
  $c\in \rCrossinglessB{p}{\bk}$, $\Phi$ takes the
  submodule
  \[
    \CKTfuncIdealB{\Tan_1}{\bk}{\bk}(a,b)\otimes\KTfunc{\ou{\Tan_2}}(b,c)+
    \KTfunc{\ou{\Tan_1}}(a,b)\otimes \CKTfuncIdealB{\Tan_2}{\bk}{\bk}(b,c)
    \subset \KTfunc{\ou{\Tan_1}}(a,b)\otimes\KTfunc{\ou{\Tan_2}}(b,c)
  \]
  to the submodule
  \[
    \CKTfuncIdealB{\Tan_1\Tan_2}{\bk}{\bk}(a,c)\subset \KTfunc{\ou{\Tan_1\Tan_2}}(a,c)
  \]
  and hence induces a homomorphism
  \begin{equation}
    \label{eq:glueMap}
    \overline{\Psi}\co
    \CKTfuncB{\Tan_1}{\bk}{\bk}\otimes_\ZZ\CKTfuncB{\Tan_2}{\bk}{\bk}\to
    \CKTfuncB{\Tan_1\Tan_2}{\bk}{\bk}.
  \end{equation}
\end{lemma}
\begin{proof}
  It is immediate from far-commutativity that the gluing map is a
  chain map of bimodules (see also~\cite{Kho-kh-tangles}). It remains
  to check that its restriction preserves the submodules. This is
  clear for elements coming from type III arcs on the left of
  $T_1$ or the right of
  $T_2$. Next, suppose we are considering a generator of
  $\KTfunc{\ou{\Tan_1}}(a,b)$ in which some type II circle is
  labeled $X$. Let
  $P$ be any intersection of that circle with a platform for
  $a$. Then the cobordism map takes this generator to one where the
  circle containing $P$ is labeled
  $X$, so there is a type II circle labeled $X$ as desired.

  Next, suppose there is a type III arc in $a\Tan_1\Wmirror{b}$ to the
  right of $T_1$, and let $P,Q$ be its endpoints. If $P$ and
  $Q$ are on the same circle of $a\ou{\Tan_1\Tan_2}\Wmirror{c}$ then
  by Lemma~\ref{lem:type-III}, $a\Tan_1\Tan_2\Wmirror{c}$ has a type
  III arc and we are done. Otherwise, at some point in the canonical
  cobordism, the circle containing $P$ and $Q$ splits. After the
  split, one of the components is a circle labeled $X$ passing through
  a platform, implying that the corresponding generator for
  $a\ou{\Tan_1\Tan_2}\Wmirror{c}$ has a type II (or perhaps III) circle
  labeled $X$.
\end{proof}

\begin{lemma}
  The map $\wt{\Psi}$ descends to the tensor product
  \[
    \KTfunc{\ou{\Tan_1}}\otimes_{\CKTalgTB{n}{\bk}}\KTfunc{\ou{\Tan_2}}
    \to\KTfunc{\ou{\Tan_1\Tan_2}}
  \]
  and hence $\overline{\Psi}$ descends to a map 
  \[
    \Psi\co
    \CKTfuncB{\Tan_1}{\bk}{\bk}\otimes_{\CKTalgB{n}{\bk}}\CKTfuncB{\Tan_2}{\bk}{\bk}\to
    \CKTfuncB{\Tan_1\Tan_2}{\bk}{\bk}.
  \]
\end{lemma}
\begin{proof}
  This follows from far-commutativity of the cobordism maps and the
  fact that the action of $\CKTalgTB{n}{\bk}$ is itself induced by the
  canonical cobordism $\Wmirror{b}\amalg b\to \Id~$\cite[Proof of
  Theorem 1]{Kho-kh-tangles}.
\end{proof}

\begin{lemma}\label{lem:crT}
  Suppose $\Tan_1$ is a flat tangle and $a\in\rCrossinglessB{m}{\bk}$
  is a crossingless matching so that $a\Tan_1$ has no type III arcs. Then
  $a\ou{\Tan_1}$ is (isotopic rel endpoints to) the union of an element
  $a'\in \rCrossinglessB{n}{\bk}$ and a collection of unknots. 
\end{lemma}
\begin{proof}
  Since any flat $(0,n+k_1+k_2)$-tangle is the disjoint union of a crossingless
  matching and some unknots, all that remains is to verify that none of the
  points on the same platform are connected by the matching; but if two points
  on the same platform were connected then the corresponding arc of $a\Tan_1$
  would be a type III arc.
\end{proof}

\begin{lemma}\label{lem:sweet}
  For $\Tan_1$ an $(m,n)$-tangle,
  $\CKTfuncB{\Tan_1}{\bk}{\bk}$ is a projective left module over
  $\CKTalgB{m}{\bk}$ and is a projective right module over
  $\CKTalgB{n}{\bk}$.
\end{lemma}
That is, in Khovanov's language~\cite{Kho-kh-tangles},
$\CKTfuncB{\Tan_1}{\bk}{\bk}$ is a sweet bimodule. This is weaker than
$\CKTfuncB{\Tan_1}{\bk}{\bk}$ being bi-projective (projective as a
bimodule).
\begin{proof}[Proof of Lemma~\ref{lem:sweet}]
  It suffices to prove the result when $\Tan_1$ is a flat tangle.  
  We prove $\CKTfuncB{\Tan_1}{\bk}{\bk}$ is right projective. As a right
  module,
  \[
    \CKTfuncB{\Tan_1}{\bk}{\bk}=\bigoplus_{a\in\rCrossinglessB{m}{\bk}}\CKTfuncB{\Tan_1}{\bk}{\bk}(a,\cdot),
  \]
  so it suffices to prove each summand is right projective. Let $a'$ be the
  crossingless matching from Lemma~\ref{lem:crT}, and suppose that
  $a\ou{\Tan_1}$ has $q$ circle components.  Then, it is immediate from the
  definitions that
  \[
    \CKTfuncB{\Tan_1}{\bk}{\bk}(a,\cdot)=V^{\otimes q}\otimes (a'\cdot \CKTalgB{n}{\bk}).
  \]
  This implies the result.
\end{proof}

\begin{proof}[Proof of Theorem~\ref{thm:CK-pairing}]
  As noted above, by Lemma~\ref{lem:iota-func-iso}, we may assume $\bh=\bk=\bl$. 
  By Lemma~\ref{lem:sweet}, we can replace the derived tensor product
  with the ordinary tensor product. So, it suffices to prove that the
  gluing map $\Psi$ is a chain isomorphism.  By
  Lemma~\ref{lem:glue-is-chain-map}, we know $\Psi$ is a well-defined
  chain map, so it suffices to prove that for single vertices $v$, $w$
  for the cube of resolutions of $\Tan_1$, $\Tan_2$, and crossingless
  matchings $a$, $c$, the gluing map
  \begin{equation}\label{eq:Psi-is-iso-pf}
    \Psi\co \CKTfuncB{\Tan_{1,v}}{\bk}{\bk}(a,\cdot)\otimes_{\CKTalgB{n}{\bk}}\CKTfuncB{\Tan_{2,w}}{\bk}{\bk}(\cdot,c)\to
    \CKTfuncB{\Tan_{1,v}\Tan_{2,w}}{\bk}{\bk}(a,c)
  \end{equation}
  is an isomorphism of free abelian groups.

  If $a\ou{\Tan_{1,v}}$ or $\ou{\Tan_{2,w}}\Wmirror{c}$ has a type III arc then both
  sides vanish, so there is nothing to prove.
  
  In the remaining case, by Lemma~\ref{lem:sweet} and its proof, we have
  \[
    \CKTfuncB{\Tan_{1,v}}{\bk}{\bk}(a,\cdot)\otimes_{\CKTalgB{n}{\bk}}\CKTfuncB{\Tan_{2,w}}{\bk}{\bk}(\cdot,c)
    \cong
    V^{\otimes (q+q')}\otimes (a'\cdot \CKTalgB{n}{\bk})\otimes_{\CKTalgB{n}{\bk}}(\CKTalgB{n}{\bk}\cdot c'),
  \]
  where $a'$ is as in Lemma~\ref{lem:crT} and $c'$ is defined similarly. This is
  isomorphic to
  \[
    V^{\otimes (q+q')}\otimes (a'\cdot \CKTalgB{n}{\bk}\cdot c')=
    V^{\otimes (q+q')}\otimes \CKTalgB{n}{\bk}(a',c').
  \]
  This already proves that the two sides of Equation~\eqref{eq:Psi-is-iso-pf}
  are isomorphic.

  So, to see that $\Psi$ is an isomorphism, it suffices to verify that
  the map $\ol{\Psi}$ from Equation~\eqref{eq:glueMap} is
  surjective. Fix a generator
  $y\in\CKTfuncB{\Tan_{1,v}\Tan_{2,w}}{\bk}{\bk}(a,c)$. Consider
  generators $y_1\in \CKTfuncB{\Tan_{1,v}}{\bk}{\bk}(a,c')$ and
  $y_2\in \CKTfuncB{\Tan_{2,w}}{\bk}{\bk}(c',c)$, so that:
  \begin{itemize}[leftmargin=*]
  \item  $y_1$ and $y_2$ label the closed components of $a\ou{\Tan_{1,v}}$ and $\ou{\Tan_{2,w}}\Wmirror{c}$ the same way as $y$.
  \item $y_1$ labels each arc of $\Wmirror{c'}$ the same way as $y$ labels the corresponding arc of $\Wmirror{c}$.
  \item $y_2$ labels each arc of $c'$ by $1$.
  \end{itemize} 
  Since the merge map from $\Wmirror{c'}\amalg c'$ to $\Id$ induces
  multiplication in the algebra, and $1$ is a unit, it follows that
  $\ol{\Psi}(y_1\otimes y_2)=y$.
\end{proof}

In light of Theorem~\ref{thm:CK-pairing}, Lemma~\ref{lem:iota-func-iso},
and Lemma~\ref{lem:restrict-invt}, it is natural to restrict to the
case that $h_2=m-h_1$ and $k_2=n-k_1$, exactly as Chen-Khovanov
do. So, let
\begin{align*}
  \CKTalgBig{n}&=\bigoplus_{k=0}^n\CKTalg{n}{n-k}{k} &
  \CKTfuncBig{\Tan}&=\bigoplus_{\substack{h,k\\m-n=2(h-k)}}\CKTfunc{\Tan}{m-h}{h}{n-k}{k}.
\end{align*}

\section{Review of the spectral arc algebras and modules}\label{sec:review-spectral-arc-algebra}

Hu-Kriz-Kriz used functors to the Burnside category and
Elmendorf-Mandell's $K$-theory multifunctor to give a new
construction of a Khovanov homotopy type~\cite{HKK-Kh-htpy}.  One can
show~\cite{LLS-khovanov-product} that their construction is equivalent
to our original Khovanov stable homotopy type (from~\cite{RS-khovanov}).  The
canonical nature of the Elmendorf-Mandell construction makes
Hu-Kriz-Kriz's approach well-suited to functorial aspects of Khovanov
homotopy, like tangle invariants. Therefore, our spectral refinement
of the tangle invariants~\cite{LLS-kh-tangles} was carried out using
the Hu-Kriz-Kriz viewpoint, and we continue to use that approach
here.

This section provides concise background on
some of the relevant topics. (For more detailed background,
see~\cite[\S2]{LLS-kh-tangles}.)
Section~\ref{sec:multicategories} recalls the notion of a
multicategory, the basic framework for the
construction. Section~\ref{sec:burnside-multicat} introduces a
particular target multicategory, the graded Burnside category (of the
trivial group). Section~\ref{sec:shape-multicats} discusses how one
can reinterpret the tangle invariants (both combinatorial and
spectral) in terms of multifunctors. Section~\ref{sec:Hn-spec-review}
reviews key properties of the
spectral refinements of the arc algebras and modules
(from~\cite{LLS-kh-tangles}). Section~\ref{sec:absorbing}
introduces two analogues of ideals, submodules, and quotient rings
and modules in the setting of multifunctors, one of which is used to
define the spectral platform algebras and bimodules and the other of
which is used to prove the spectral bimodules are
invariants. Section~\ref{sec:equiv-funcs} recalls a notion of
equivalence for multifunctors which guarantees that the corresponding
spectral bimodules are weakly equivalent.

To avoid excessive repetition, the material presented in this section
omits many details from our previous paper~\cite{LLS-kh-tangles}, with
which the reader is assumed to be familiar or be willing to accept as a
black box.

\subsection{Multicategories}\label{sec:multicategories}
A \emph{multicategory} (or \emph{colored operad}) is an operad with
many objects or, equivalently, a category with $n$-input morphism sets
and the obvious kinds of compositions; so
\[
\begin{tikzpicture}[xscale=2.5]
  \node (monoid) at (0,0) {monoid};
  \node (category) at (1,0) {category};
  \node (operad) at (0,-1) {operad};
  \node (multicat) at (1,-1) {multicategory.};

  \draw[->] (monoid)--(category);
  \draw[->] (monoid)--(operad);
  \draw[->] (category)--(multicat);
  \draw[->] (operad)--(multicat);
  
  \node at ($(monoid)!0.5!(multicat)$) {\begin{tikzpicture}[scale=0.2]
      \draw (1,0)--(0,0)--(0,-1);
      \fill[black] (0.5,-0.5) circle (0.1cm);
    \end{tikzpicture}};
\end{tikzpicture}
\]
(See~\cite[Definition 2.1]{EM-top-machine}.) Any monoidal category can
be viewed as a multicategory (cf.~\emph{$\star$-categories}
from~\cite{HKK-Kh-htpy}). For instance, the category of abelian groups
forms a multicategory, with $\Hom(G_1,\dots,G_n;H)$ the group
homomorphisms $G_1\otimes\cdots\otimes G_n\to H$.  Similarly,
topological spaces (or simplicial sets) forms a multicategory, with
$\Hom(X_1,\dots,X_n;Y)$ the maps $X_1\times\cdots\times X_n\to
Y$. Similarly, symmetric spectra $\mSpectra$ \cite{HSS-top-symmetric}
forms a multicategory and is naturally a simplicial multicategory
(multicategory enriched in simplicial sets).
As in categories, given a multicategory $\Cat$ we will typically write
$\Cat(x_1,\dots,x_n;y)$ for $\Hom_{\Cat}(x_1,\dots,x_n;y)$.

A \emph{multifunctor} from $\Cat$ to $\Dat$ consists of a map
$F\co\Ob(\Cat)\to\Ob(\Dat)$ and, for each
$x_1,\dots,x_n,y\in\Ob(\Cat)$, a map
$\Cat(x_1,\dots,x_n;y)\to\Dat(F(x_1),\dots,F(x_n),F(y))$, respecting
the identity maps and composition. Note that while we sometimes
consider multicategories enriched in groupoids or simplicial sets,
for multicategories
composition is always strictly associative and has strict units, and
multifunctors always strictly respect the identity 
maps and compositions; we never consider a multicategorical analogue
of lax $2$-categories.

\subsection{The Burnside multicategory}\label{sec:burnside-multicat}

The \emph{graded Burnside multicategory} $\mBurnside$ is the
multicategory enriched in categories with:
\begin{itemize}[leftmargin=*]
\item Objects finite graded sets $X$.
\item $\mBurnside(X;Y)=\Hom_{\mBurnside}(X;Y)$ the category with objects finite correspondences
  $(A,s\co A\to X,t\co A\to Y)$ of graded sets and morphisms diagrams
  of graded set maps
  \[
  \begin{tikzpicture}[yscale=1.2]
    \tikzset{arrowlabel/.style={pos=0.5,inner sep=0pt,outer sep=1pt}}
    \node (x) at (0,0) {$X$};
    \node (y) at (4,0) {$Y$};
    \node (a) at (1,1) {$A$};
    \node (b) at (3,1) {$B$};

    \draw[->] (a) -- (x) node[arrowlabel,anchor=south east] {\tiny $s$};
    \draw[->] (b) -- (x) node[arrowlabel,anchor=north west] {\tiny $s$};
    \draw[->] (a) -- (y) node[arrowlabel,anchor=north east] {\tiny $t$};
    \draw[->] (b) -- (y) node[arrowlabel,anchor=south west] {\tiny $t$};

    \draw[->] (a) -- (b) node[arrowlabel,anchor=north] {\tiny $f$} node[arrowlabel,anchor=south] {\tiny $\cong$};
  \end{tikzpicture}
  \]
  with $f$ being a bijection.
\item More generally,
  $\mBurnside(X_1,\dots,X_n;Y)=\mBurnside(X_1\times\cdots\times
  X_n;Y)$, where $X_1\times\dots\times X_n$ is a graded set with
  \[
    \gr_{X_1\times\dots\times X_n}(x_1,\dots,x_n)=\gr_{X_1}(x_1)+\cdots+\gr_{X_n}(x_n)
  \]
  (like the grading on a tensor product).
\item Composition is induced by the fiber product.
\end{itemize}
(We elide some technicalities about associativity of the Cartesian
product of sets; see, e.g.,~\cite[\S3.2.1]{LLS-kh-tangles} for one way
of handling these.) Taking the nerve of each multimorphism category
turns $\mBurnside$ into a simplicial multicategory, which we also
denote $\mBurnside$ and call the graded Burnside multicategory.

There is a multifunctor $\Forget\co \mBurnside\to \GrAbelianGroups$
which sends a set $X$ to $\bigoplus_{x\in X}\ZZ$ and a correspondence
$A$ to the $Y\times X$-matrix $[\# s^{-1}(x)\cap t^{-1}(y)]_{x\in
  X,y\in Y}$.

The Elmendorf-Mandell machine can be used to lift this functor to
spectra. In more detail, Elmendorf-Mandell's $K$-theory is a multifunctor
from the multicategory of all permutative categories to spectra, and there
is a multifunctor from the graded Burnside multicategory to the multicategory of
permutative categories which sends $X\in\mBurnside$ to
$\Sets/X$. Composing these two functors, we get a (simplicially
enriched) functor $\mBurnside\to\GrSpectra$, which we will still
denote $K$. Moreover, the following diagram commutes up to quasi-isomorphism:
\[
\begin{tikzpicture}[xscale=2]
\node (burn) at (0,0) {$\mBurnside$};
\node (spec) at (1,0) {$\GrSpectra$};
\node (ab) at (0,-1) {$\GrAbelianGroups$};
\node (kom) at (1,-1) {$\GrComplexes$.};

\draw[->] (burn) -- (ab) node[midway,anchor=east] {\small $\Forget$};
\draw[->] (burn) -- (spec) node[midway,anchor=south] {\small $K$};
\draw[->] (spec) -- (kom) node[midway,anchor=west] {\small $C_*$};
\draw[right hook->] (ab) -- (kom);
\end{tikzpicture}
\]
In particular, for any $X\in\mBurnside$, $K(X)$ is
weakly equivalent to $\bigvee_{x\in X}\SphereS$.

\subsection{Tangle invariants as multifunctors}\label{sec:shape-multicats}
As a final step before using the Elmendorf-Mandell machine, we
reformulate the tangle invariants in terms of multifunctors. Consider the
\emph{tangle shape multicategory} $\SmTshape{2m}{2n}$ with:
\begin{itemize}[leftmargin=*]
\item Three kinds of objects: pairs $(a_1,a_2)\in\Crossingless{2m}\times\Crossingless{2m}$, pairs $(b_1,b_2)\in \Crossingless{2n}\times\Crossingless{2n}$, and pairs $(a,b)\in\Crossingless{2m}\times\Crossingless{2n}$. For reasons that will become clear, we will denote the third kind of object as $(a,T,b)$, where $T$ is a placeholder.
\item A unique multimorphism in each of the following cases, and no other multimorphisms:
  \begin{align}
    (a_1,a_2),(a_2,a_3),\dots,(a_{\alpha-1},a_\alpha)&\to (a_1,a_\alpha)\\
    (b_1,b_2),(b_2,b_3),\dots,(b_{\beta-1},b_\beta)&\to (b_1,b_\beta)\\
    (a_1,a_2),\dots,(a_{\alpha-1},a_\alpha),(a_\alpha,T,b_1),(b_1,b_2),\dots,(b_{\beta-1},b_\beta)&\to (a_1,T,b_\beta)\label{eq:T-multi-morph}\\  
    \shortintertext{The special cases $\alpha=1$ and $\beta=1$ are allowed, and they correspond to $0$-input morphisms:}
    \emptyset&\to (a,a)\\
    \emptyset&\to (b,b)
  \end{align}
\end{itemize}
This category has several full subcategories:
\begin{itemize}[leftmargin=*]
\item The category $\mHshapeS{2m}\subset\SmTshape{2m}{2n}$ spanned by the objects in $\Crossingless{2m}\times\Crossingless{2m}$. This is isomorphic to the subcategory of $\SmTshape{2n}{2m}$ spanned by the objects in $\Crossingless{2m}\times\Crossingless{2m}$.
\item The category $\mCKAshapeS{m}{h_1}{h_2}=\mCKAshapeSB{m}{\bh}\subset\mHshapeS{m+h_1+h_2}$ spanned by the objects in $\rCrossinglessB{m}{\bh}\times \rCrossinglessB{m}{\bh}\subset \Crossingless{m+h_1+h_2}\times\Crossingless{m+h_1+h_2}$.
\item The category $\SmCKTshape{m}{n}{h_1}{h_2}{k_1}{k_2}=\SmCKTshapeB{m}{n}{\bh}{\bk}\subset \SmTshape{m+h_1+h_2}{n+k_1+k_2}$ spanned by objects in
  \[
    \bigl(\rCrossinglessB{m}{\bh}\times\rCrossinglessB{m}{\bh}\bigr) \amalg\bigl(\rCrossinglessB{n}{\bk}\times\rCrossinglessB{n}{\bk}\bigr) \amalg 
    \bigl(\rCrossinglessB{m}{\bh}\times \rCrossinglessB{n}{\bk}\bigr).
  \]
  (with $h_1-h_2=k_1-k_2$).
\end{itemize}

The arc algebras determine a multifunctor 
\[
  F_{\KTalg{2m}}\co \mHshapeS{2m}\to\GrAbelianGroups
\]
by defining $F_{\KTalg{2m}}(a_1,a_2)=\KTalg{2m}(a_1,a_2)$ and  
\[
  F_{\KTalg{2m}}\bigl((a_1,a_2),(a_2,a_3),\dots,(a_{\alpha-1},a_\alpha)\to (a_1,a_\alpha)\bigr)
\]
to be the iterated composition (multiplication) map 
\[
  f_{12},f_{23},\dots,f_{\alpha-1,\alpha}\mapsto f_{\alpha-1,\alpha}\circ \cdots\circ f_{23}\circ f_{12}.
\]
Conversely, the functor $F_{\KTalg{2m}}$ determines the arc algebra in
an obvious way.  Extending further, Khovanov's invariant of a tangle $\Tan$ (with ordered
crossings) is equivalent to a functor
\[
  \SmTshape{2m}{2n}\to \GrComplexes.
\]
(See also \cite[\S2.3]{LLS-kh-tangles}.)

Similarly, the platform algebra and invariant of a tangle $\Tan$ are
equivalent to functors
\begin{align*}
  \mCKAshapeSB{m}{\bh}&\to \GrAbelianGroups, &
  \SmCKTshapeB{m}{n}{\bh}{\bk}&\to\GrComplexes,
\end{align*}
respectively.

For the homotopical refinement we need two variants on these
constructions. First, we need a \emph{canonical thickening}
$\mTshape{2m}{2n}$ of $\SmTshape{2m}{2n}$, which is a simplicial
multicategory with the same objects as $\SmTshape{2m}{2n}$ but in which
each $\Hom$-set is replaced by a (particular) simplicial
set~\cite[\S2.4]{LLS-kh-tangles}, defined in terms of labeled trees of
\emph{basic multimorphisms}.

Second, we need to be able to multiply by a cube. Define $\CCat{N}\ttimes\SmTshape{2m}{2n}$ to have objects 
\[
  \bigl(\Crossingless{2m}\times\Crossingless{2m}\bigr)\amalg
  \bigl(\Crossingless{2n}\times\Crossingless{2n}\bigr)\amalg
  \bigl(\Ob(\CCat{N})\times \Crossingless{2m}\times\Crossingless{2n}\bigr).
\]
where, again, $\CCat{N}=\{0\to 1\}^N$.

The multimorphisms in $\CCat{N}\ttimes\SmTshape{2m}{2n}$ are the same as
in $\SmTshape{2m}{2n}$ except that Formula~(\ref{eq:T-multi-morph}) is replaced
with
\begin{equation}
      (a_1,a_2),\dots,(a_{\alpha-1},a_\alpha),(v,a_\alpha,T,b_1),(b_1,b_2),\dots,(b_{\beta-1},b_\beta)\to (w,a_1,T,b_\beta)
\end{equation}
whenever $v\leq w\in \CCat{N}$ (with the partial order induced by
$0<1$). This category still has $\mHshapeS{2m}$ and $\mHshapeS{2n}$ as
subcategories. The Khovanov invariant of an $N$-crossing tangle $T$ is
induced by a functor
\[
  F_{\KTfunc{}}\co \CCat{N}\ttimes \SmTshape{2m}{2n}\to\GrAbelianGroups.
\]
There is also a canonical thickening $\CCat{N}\ttimes\mTshape{2m}{2n}$~\cite[\S3.2.4]{LLS-kh-tangles}.

The category $\CCat{N}\ttimes\SmTshape{m+h_1+h_2}{n+k_1+k_2}$ (respectively $\CCat{N}\ttimes\mTshape{m+h_1+h_2}{n+k_1+k_2}$) has a full subcategory $\CCat{N}\ttimes\SmCKTshapeB{m}{n}{\bh}{\bk}$ (respectively $\CCat{N}\ttimes\mCKTshapeB{m}{n}{\bh}{\bk}$).

\subsection{The homotopical refinement of Khovanov's arc algebras and
  bimodules}\label{sec:Hn-spec-review}

By using a particular multicategory of divided cobordisms, in our
previous paper~\cite{LLS-kh-tangles} we defined multifunctors
\begin{align*}
  \mHinv{2n}\co \mHshape{2n}\to\mBurnside\\
  \mTinv{T}\co \CCat{N}\ttimes \mTshape{2m}{2n}\to \mBurnside.
\end{align*}
(These multifunctors were denoted $\ulmHinv{n}$ and  $\ulmTinvNF{T}$.)
We will not need the details of these constructions, just the following
properties:
\begin{enumerate}[label=($\Psi$-\arabic*),leftmargin=*]
\item The compositions $\Forget\circ \mHinv{2n}$ and $\Forget\circ\mTinv{T}$
  agree with the Khovanov arc algebra and module $F_{\KTalg{2n}}$ and
  $F_{\KTfunc{}}$, respectively. Further, $\mHinv{2m}$ and $\mHinv{2n}$ are the restrictions of
  $\mTinv{T}$. More precisely, the following diagrams commute:
  \[
    \qquad
    \begin{tikzpicture}[xscale=2,yscale=1.5]
      \begin{scope}[xshift=-0.5cm]
      \node (t) at (0,0) {$\mHshape{2n}$};
      \node (ts) at (0,-1) {$\mHshapeS{2n}$};
      \node (b) at (1,0) {$\mBurnside$};
      \node (a) at (1,-1) {$\GrAbelianGroups$};

      \draw[->] (t) -- (b) node[midway,anchor=south] {\tiny $\mHinv{2n}$};
      \draw[->] (t) -- (ts);
      \draw[->] (ts) -- (a) node[midway,anchor=north] {\tiny $F_{\KTalg{2n}}$};
      \draw[->] (b) -- (a) node[midway,anchor=west] {\tiny $\Forget$};
      \end{scope}

      \begin{scope}[xshift=2cm]
      \node (t) at (-0.2,0) {$\CCat{N}\ttimes\mTshape{2m}{2n}$};
      \node (ts) at (-0.2,-1) {$\CCat{N}\ttimes\SmTshape{2m}{2n}$};
      \node (b) at (1,0) {$\mBurnside$};
      \node (a) at (1,-1) {$\GrAbelianGroups$};

      \draw[->] (t) -- (b) node[midway,anchor=south] {\tiny $\mTinv{T}$};
      \draw[->] (t) -- (ts);
      \draw[->] (ts) -- (a) node[midway,anchor=north] {\tiny $F_{\KTfunc{T}}$};
      \draw[->] (b) -- (a) node[midway,anchor=west] {\tiny $\Forget$};
      \end{scope}

      \begin{scope}[xshift=4.9cm]
      \node (t) at (0,0) {$\CCat{N}\ttimes\mTshape{2m}{2n}$};
      \node (sm) at (-1,0) {$\mHshape{2m}$};
      \node (sn) at (1,0) {$\mHshape{2n}$};
      \node (b) at (0,-1) {$\mBurnside$};

      \draw[->] (t) -- (b) node[midway,anchor=west] {\tiny $\mTinv{T}$};
      \draw[->] (sm) -- (b) node[midway,anchor=north east] {\tiny $\mHinv{2m}$};
      \draw[->] (sn) -- (b) node[midway,anchor=north west] {\tiny $\mHinv{2n}$};
      \draw[right hook->] (sm) -- (t);
      \draw[left hook->] (sn) -- (t);
      \end{scope}
    \end{tikzpicture}
  \]
\item On objects, $\mHinv{2n}(a,b)$ is the set of labelings of
  the components of $a\Wmirror{b}$ by elements of $\{1,X\}$, i.e., the
  set of Khovanov generators for $a\Wmirror{b}$. Similarly,
  $\mTinv{T}(v,a,T,b)$ (respectively $\mTinv{T}(a_1,a_2)$,
  $\mTinv{T}(b_1,b_2)$) is the set of Khovanov generators for
  $aT_v\Wmirror{b}$ (respectively $a_1\Wmirror{a_2}$,
  $b_1\Wmirror{b_2}$), where $T_v$ is the $v$-resolution of $T$.
\item\label{item:Psi-mmorph} Given a multi-morphism $f\in \mTshape{2m}{2n}(x_1,\dots,x_\ell;y)$ with
  associated canonical cobordism $\Sigma$, as well as Khovanov generators
  $c_i\in \mTinv{T}(x_i)$ and $d\in \mTinv{T}(y)$, the correspondence $\mTinv{T}(f)$ satisfies
  \[
    s^{-1}(c_1,\dots,c_\ell)\cap t^{-1}(d)=\emptyset \subset \mTinv{T}(f)
  \]
  unless all of the following are satisfied:
  \begin{enumerate}[leftmargin=*]
  \item Every component of $\Sigma$ has genus $0$ or $1$.
  \item For each genus $0$ component of $\Sigma$, either:
    \begin{enumerate}[leftmargin=*]
    \item all circles in the incoming boundary are labeled $1$ by $(c_1,\dots,c_\ell)$ and exactly one circle in the outgoing boundary is labeled $1$ by $d$ or
    \item exactly one circle in the incoming boundary is labeled $X$ by $(c_1,\dots,c_\ell)$ and all circles in the outgoing boundary are labeled $X$ by $d$.
    \end{enumerate}
  \item For each genus $1$ component of $\Sigma$, all incoming circles are labeled $1$ by $(c_1,\dots,c_\ell)$ and all outgoing circles are labeled $X$ by $d$.
  \end{enumerate}
  (Since $\mHinv{2n}$ is a restriction of $\mTinv{T}$, this also applies to
  $\mHinv{2n}$.)
\end{enumerate}

Now, the spectral arc algebra is defined as follows. Composing the
functor $\mHinv{2n}$ with the map $K\from\mBurnside\to\ZZ^{\mSpectra}$
coming from Elmendorf-Mandell's $K$-theory gives a functor
$\mHshape{2n}\to \ZZ^{\mSpectra}$. Their rectification
results~\cite[Theorems 1.3 and 1.4]{EM-top-machine} then give a
multifunctor $\mHshapeS{2n}\to \ZZ^{\mSpectra}$. Finally,
reinterpreting this as in Section~\ref{sec:shape-multicats} gives a
spectral category $\KTSpecCat{2n}$ with one object per crossingless
matching; if one prefers, one can take the wedge sum of all morphisms
spaces in this category to obtain a ring spectrum $\KTSpecRing{2n}$.

Similarly, given an oriented tangle diagram $T$ with $N$ crossings, of
which $N_+$ are positive, we can consider the composition
\[
\CCat{N}\ttimes \mTshape{2m}{2n}\stackrel{\mTinv{\Tan}}{\longrightarrow}
\mBurnside \stackrel{K}{\longrightarrow} \GrmSpectra.
\]
Rectifying this composition gives a multifunctor
$G\co \CCat{N}\ttimes \SmTshape{2m}{2n}\to \GrmSpectra$. We turn this
into a functor $H\co \SmTshape{2m}{2n}\to \GrmSpectra$ by letting
$H(a_1,a_2)=G(a_1,a_2)$, $H(b_1,b_2)=G(b_1,b_2)$, and $H(a,T,b)$ be
the iterated mapping cone, over the cube $\CCat{N}$, of $G(v,a,T,b)$,
formally desuspended $N_+$ times. We can then reinterpret $H$ as a
spectral bimodule $\KTSpecBim{T}$ over $\KTSpecCat{2m}$ and
$\KTSpecCat{2n}$, as in Section~\ref{sec:shape-multicats}. This can
also be viewed as a spectral bimodule $\KTSpecBimRing{T}$ over
$\KTSpecRing{2m}$ and $\KTSpecRing{2n}$.

Abstracting somewhat, let $\Cat$ be either $\SmTshape{2m}{2n}$ or
$\SmCKTshapeB{m}{n}{\bh}{\bk}$, and let $\Dat\subset\Cat$ be one of
$\mHshapeS{2m}$ or $\mHshapeS{2n}$ (if $\Cat$ is $\SmTshape{2m}{2n}$),
or $\mCKAshapeSB{m}{\bh}$ or $\mCKAshapeSB{n}{\bk}$ (if $\Cat$ is
$\SmCKTshapeB{m}{n}{\bh}{\bk}$). A \emph{stable functor} is a pair
$({F},S)$ where ${F}\co \CCat{N}\ttimes \wt{\Cat}\to \mBurnside$ and
$S\in\ZZ$. For instance, given a tangle $T$, $(\mTinv{\Tan},N_+)$ is a
stable functor. The procedure in the previous paragraph (rectifying,
taking mapping cone along $\CCat{N}$, then desuspending $S$
times)
produces a functor $\Realize{{F}}\co \Cat\to \GrmSpectra$. If
$A=F|_{\Dat}$ then let $\CRealize{A}$ be the spectral category
obtained by restricting $\Realize{F}$ to $\Dat$ and reinterpreting as
in Section~\ref{sec:shape-multicats}. Also as in
Section~\ref{sec:shape-multicats}, the functor $\Realize{{F}}$ may be
viewed as a spectral bimodule $\CRealize{F}$ over the two spectral
categories $\CRealize{A}$.  For $F=\mTinv{\Tan}$, the spectral
category $\CRealize{A}$ is the category $\KTSpecCat{2m}$ or
$\KTSpecCat{2n}$ and the spectral bimodule $\CRealize{F}$ is
$\KTSpecBim{T}$, as defined earlier.  So, to spectrify the
Chen-Khovanov tangle invariants, all that remains is to choose the
right stable functor $({F}\co
\CCat{N}\ttimes\mCKTshapeB{m}{n}{\bh}{\bk}\to \mBurnside,S)$.

\subsection{Subfunctors and quotient functors}\label{sec:absorbing}
The following will serve as an analogue of an ideal $I$ in a ring $R$:
\begin{definition}\label{def:absorbing-subfunctor}
  Given a multicategory $\Cat$ and a multifunctor
  $F\co\Cat\to\mBurnside$, an \emph{absorbing subfunctor} of $\Cat$ is
  a collection of subsets $G(a)\subset F(a)$, $a\in\Cat$, so that for
  any $p_1,\dots,p_\ell,q\in\Cat$, $f\in \Cat(p_1,\dots,p_\ell;q)$, $x_i\in
  F(p_i)$, and $y\in F(q)$, if some $x_i\in G(p_i)$ and $y\not\in
  G(q)$ then
  \begin{equation}\label{eq:absorbing}
    s^{-1}(x_1,\dots,x_\ell)\cap t^{-1}(y)=\emptyset \subset F(f).
  \end{equation}

  Extend $G$ to a multifunctor $G\co \Cat\to\mBurnside$ by defining, for
  $f\in\Cat(p_1,\dots,p_\ell;q)$,
  \[
    G(f)=s^{-1}(G(p_1)\times\dots\times G(p_\ell))\cap t^{-1}(G(q))\subset F(f)
  \]
  with the obvious source and target maps, and $2$-morphisms induced
  by $F$ in the obvious way. The fact that $G$ respects
  multi-composition, i.e., that for an $\ell$-input morphism $f$,
  \[
    G(f\circ(g_1,\dots,g_\ell))=G(f)\circ(G(g_1),\dots,G(g_\ell)),
  \]
  follows from Equation~\eqref{eq:absorbing}.

  Given an absorbing subfunctor $G$ of $F$, there is a corresponding
  \emph{quotient functor} $F/G$ defined as follows:
  \begin{itemize}[leftmargin=*]
  \item On objects $a\in\Cat$, $(F/G)(a)=F(a)\setminus G(a)$.
  \item On morphisms $f\in \Cat(p_1,\dots,p_\ell;q)$,
    $(F/G)(f)=s^{-1}((F/G)(p_1)\times\cdots\times (F/G)(p_\ell))\cap
    t^{-1}((F/G)(q))$.
  \item On $2$-morphisms, $F/G$ is induced from $F$.
  \end{itemize}
\end{definition}

\begin{lemma}
  If $G$ is an absorbing subfunctor then $F/G$ is a multifunctor.
\end{lemma}
\begin{proof}
  This is straightforward from the definition and
  Equation~\eqref{eq:absorbing}.
\end{proof}

For us, the quotient by an absorbing subfunctor will correspond to quotienting
both a ring and a module by an ideal. There is another kind of
quotient that corresponds to leaving the ring unchanged but taking the
quotient of a module, which is useful in proving invariance under
Reidemeister moves:
\begin{definition}\label{def:insular-subfunctor}\cite[Definition 3.25]{LLS-kh-tangles}  
  Let $\Cat$ be a multicategory and let $X\subset\Ob(\Cat)$ be a
  subset of the objects so that there are no multi-morphisms out of
  $X$, in the sense that if $p_1,\dots,p_\ell,q\in\Cat$ and some
  $p_i\in X$ and $q\not\in X$ then
  $\Cat(p_1,\dots,p_\ell;q)=\emptyset$. (In the application, $X$ will be
  the collection of objects of the form $(v,a,T,b)$ in
  $\CCat{N}\ttimes\mTshape{2m}{2n}$ or
  $\CCat{N}\ttimes\mCKTshapeB{m}{n}{\bh}{\bk}$.)

  Given a functor $F\co \Cat\to\mBurnside$, an \emph{insular
    subfunctor} of $F$ (relative to $X$) is a collection of subsets
  $G(a)\subset F(a)$ for $a\in X$, such that for any objects
  $p_1,\dots,p_\ell,q\in\Cat$ with some $p_i\in X$, morphism
  $f\in\Cat(p_1,\dots,p_\ell;q)$, and elements $x_j\in F(p_j)$ and
  $y\in F(q)$, if $x_i\in G(p_i)$ and $y\not\in G(q)$ then 
  \begin{equation}
    s^{-1}(x_1,\dots,x_{\ell})\cap t^{-1}(y)=\emptyset\subset F(f).\label{eq:insular}
  \end{equation}

  Extend $G$ to a functor $G\co \Cat\to\mBurnside$ by defining
  $G(p)=F(p)$ for $p\in \Ob(\Cat)\setminus X$ and, for
  $f\in\Cat(p_1,\dots,p_\ell;q)$,
  \[
    G(f)=s^{-1}(G(p_1)\times\dots\times G(p_\ell))\cap t^{-1}(G(q))\subset F(f)
  \]
  with the obvious source and target maps, and $2$-morphisms induced
  by $F$ in the obvious way. The fact that $G$ respects
  multi-composition follows from Equation~\eqref{eq:insular}.
  
  Given an insular subfunctor $G$ of $F$ there is a \emph{quotient
    functor} $F/G\co \CCat{N}\ttimes\mTshape{m}{n}\to\mBurnside$
  defined by:
  \begin{itemize}[leftmargin=*]
  \item $(F/G)(p)=F(p)$ if $p\not\in X$,
  \item $(F/G)(p)=F(p)\setminus G(p)$ if $p\in X$,
  \item $(F/G)(f)=s^{-1}((F/G)(p_1)\times\dots\times (F/G)(p_\ell))\cap t^{-1}((F/G)(q))\subset F(f)$ for $f\in\Cat(p_1,\dots,p_\ell;q)$, and
  \item the value of $F/G$ on $2$-morphisms is induced by $F$.
  \end{itemize}
\end{definition}

\begin{lemma}
  If $G$ is an insular subfunctor then $F/G$ is a multifunctor.
\end{lemma}
\begin{proof}
  Again, this is straightforward from the definitions and Equation~\eqref{eq:insular}.
\end{proof}

If $\Cat$ is either $\CCat{N}\ttimes\mTshape{2m}{2n}$ or
$\CCat{N}\ttimes\mCKTshapeB{m}{n}{\bh}{\bk}$, $X$ is the collection of
objects of the form $\{(v,a,T,b)\}$, $F\co\Cat\to \mBurnside$, and $G$ is
an insular subfunctor of $F$, then there is a cofibration sequence of
\[
  \CRealize{G}\to \CRealize{F}\to \CRealize{F/G}
\]
of bimodules (over $\CRealize{F|_{\mHshape{2m}}}$ and
$\CRealize{F|_{\mHshape{2n}}}$ if $\Cat=\CCat{N}\ttimes\mTshape{2m}{2n}$
or over $\CRealize{F|_{\mCKAshapeB{m}{\bh}}}$ and
$\CRealize{F|_{\mCKAshapeB{n}{\bk}}}$ if
$\Cat=\CCat{N}\ttimes\mCKTshapeB{m}{n}{\bh}{\bk}$).

These two notions of subfunctor are compatible, in the following sense:
\begin{lemma}\label{lem:insular-absorbing}
  Fix a multicategory $\Cat$ and a subset $X\subset\Ob(\Cat)$ as in
  Definition~\ref{def:insular-subfunctor}.  Suppose
  $F\co\Cat\to\mBurnside$ is a multifunctor, $G$ is an insular
  subfunctor of $F$, and $H$ is an absorbing subfunctor of $F$. Then $G$
  induces an insular subfunctor $\overline{G}$ of $F/H$ via the
  formula
  \[
    \overline{G}(a)=G(a)\setminus H(a).
  \]
  for $a\in X$.
\end{lemma}
\begin{proof}
  Fix objects $p_1,\dots,p_\ell,q\in\Ob(\Cat)$ with $p_i\in X$, a
  morphism $f\in\Hom(p_1,\dots,p_\ell;q)$, and elements
  $x_j\in (F/H)(p_j)=F(p_j)\setminus H(p_j)$ and
  $y\in (F/H)(q)=F(q)\setminus H(q)$. Suppose
  $x_i\in \overline{G}(p_i)=G(p_i)\setminus H(p_i)$ and
  $y\in (F/H)(q)\setminus \ol{G}(q)=F(q)\setminus (G(q)\cup H(q))$.  Then,
  in particular, $x_i\in G(p_i)$ and $y\in F(q)\setminus G(q)$,
  so since $G$ is insular,
  \[
    s^{-1}(x_1,\dots,x_{\ell})\cap t^{-1}(y)=\emptyset,
  \]
  as desired.
\end{proof}

\subsection{Equivalent functors}\label{sec:equiv-funcs}
To prove the spectral refinements are invariant under Reidemeister
moves, we use a notion of equivalence of multifunctors to the graded Burnside
multicategory, which we spell out here.

Let $\Cat$ be either $\mTshape{2m}{2n}$ or $\mCKTshapeB{m}{n}{\bh}{\bk}$.

\begin{definition}
  A \emph{face inclusion} is a functor $i\co \CCat{M}\to\CCat{N}$
  which is injective on objects and preserves the relative
  grading. Given a face inclusion $i$, let $|i|$ be the amount by
  which $i$ shifts the absolute grading.
\end{definition}

The restriction of a functor $\CCat{N}\ttimes\Cat\to\mBurnside$ under a face inclusion is a functor $\CCat{M}\ttimes\Cat\to\mBurnside$. We can also extend functors under face inclusions:
\begin{definition}\cite[Definition 3.24]{LLS-kh-tangles}
  Let $i\co \CCat{M}\to\CCat{N}$ be a face inclusion and $F\co
  \CCat{M}\ttimes \Cat\to\mBurnside$. There is an induced functor
  $i_!F\co \CCat{N}\ttimes\Cat\to \mBurnside$ defined on objects by
  $(i_!F)(a,b)=F(a,b)$ and
  \[
    (i_!F)(v,a,T,b)=\begin{cases}F(u,a,T,b)&\text{if $v=i(u)$ is in the image of $i$,}\\
      \emptyset&\text{otherwise.}\end{cases}
  \]
  On multimorphisms, $i_!F$ is induced by $F$; see our previous paper
  for an explicit description.
\end{definition}

Given a stable functor $(F,S)\co\CCat{N}\ttimes\Cat\to\mBurnside$, for
each pair of crossingless matchings $(a,b)$ we have a cube of abelian
groups by restricting $\Forget\circ F\co \CCat{N}\ttimes \Cat\to
\GrAbelianGroups$ to the full subcategory spanned by the objects of
the form $(v,a,\Tan,b)$. Let $\Total{\Forget\circ F,S}$ be the direct
sum over $a,b$ of the total complex of this cube, with grading shifted
down by $S$. So, for instance, if $(F,S)=(\mTinv{\Tan},N_+)$ then
$\Total{\Forget\circ F,S}$ is the chain complex underlying the Khovanov
tangle invariant.

\begin{definition}\cite[\S3.5.2]{LLS-kh-tangles}
  We say that two stable functors $(F,S), (G,T)\co\CCat{N}\ttimes
  \Cat\to \mBurnside$ are \emph{simply stably equivalent} if either:
  \begin{enumerate}[leftmargin=*]
  \item There is a face inclusion $i\co \CCat{M}\to\CCat{N}$ so that $G=i_!F$ and $T=S+N-M-|i|$, or
  \item $S=T$ and there is a functor $H\co \CCat{N+1}\ttimes\Cat\to\mBurnside$ so that $F=H|_{\{0\}\times\CCat{N}}$, $G=H|_{\{1\}\times\CCat{N}}$, and the chain complex $\Total{\Forget\circ H,S}$ is acyclic.
  \end{enumerate}
  Two functors are \emph{stably equivalent} if they can be connected by a sequence of simple stable equivalences (i.e., stable equivalence is the symmetric, transitive closure of simple stable equivalence).
\end{definition}

\begin{lemma}\label{lem:equiv-gives-equiv}
  If $F_1$ and $F_2$ are stably equivalent multifunctors then
  $\CRealize{{F_1}}$ and $\CRealize{{F_2}}$ are equivalent spectral
  bimodules.
\end{lemma}
\begin{proof}
  In the case $\Cat=\mTshape{2m}{2n}$ this was proved in our previous
  paper~\cite[Proposition 4.7]{LLS-kh-tangles}. The proof for
  $\Cat=\mCKTshapeB{m}{n}{\bh}{\bk}$ is exactly the same.
\end{proof}

One way to produce stable equivalences is to produce insular
subfunctors. Suppose $G\subset F$ is an insular subfunctor. Fix an
integer $S$.
\begin{itemize}[leftmargin=*]
\item If $\CRealize{F/G}$ is contractible then $(F,S)\simeq (G,S)$.
\item If $\CRealize{G}$ is contractible then $(F,S)\simeq (F/G,S)$.
\item If $\CRealize{F}$ is contractible then
  $(G,S)\simeq (F/G,S+1)$.
\end{itemize}

\section{Spectral platform algebras and modules}\label{sec:spectral-CK}
\begin{table}
  \centering
  \begin{tabular}[tab:notation]{llp{4in}}
    \toprule
    Notation &  Meaning\\
    \midrule
    $\mBurnside$  & Graded Burnside multicategory.\\
    $\mSpectra$  & Multicategory of symmetric spectra.\\
    $\mHshapeS{2n}$  & Strict arc algebra shape multicategory.\\
    $\mHshape{2n}$  & Groupoid-enriched arc algebra shape multicategory.\\
    $\mHinv{2n}$  & Functor $\mHshape{2n}\to\mBurnside$ refining arc algebra.\\
    $\KTSpecCat{2n}$  & Spectral category refining $\KTalg{2n}$.\\
    $\SmTshape{2m}{2n}$  & Strict tangle shape multicategory.\\
    $\mTshape{2m}{2n}$  & Groupoid-enriched tangle shape multicategory.\\
    $\mTinv{\Tan}$ & Functor $\CCat{N}\ttimes \mTshape{2m}{2n}\to\mBurnside$ refining Khovanov cube of bimodules.\\
    $\KTSpecBim{\Tan}$  & Spectral bimodule refining $\KTfunc{\Tan}$.\\
    \midrule
    $\mCKAshapeSB{n}{\bk}$  & Strict platform algebra shape multicategory.\\
    $\mCKAshapeB{n}{\bk}$  & Groupoid-enriched platform algebra shape multicategory.\\
    $\mHinvResB{n}{\bk}$  & Restriction of $\mHinv{n+k_1+k_2}$ to $\mCKAshapeB{n}{\bk}$.\\
    $\mI$ & Absorbing subfunctor of $\mHinvResB{n}{\bk}$ corresponding to $\CKTidealB{n}{\bk}$.\\
    $\mAinvB{n}{\bk}$ & Quotient functor $\mHinvResB{n}{\bk}/\mI$ refining platform algebra.\\
    $\CKTSpecCatB{n}{\bk}$  & Spectral category refining $\CKTalgB{n}{\bk}$.\\
    $\CKTSpecCatBig{n}$ & Spectral category refining $\CKTalgBig{n}$, given by $\coprod_{k=0}^n\CKTSpecCat{n}{n-k}{k}$.\\
    $\SmCKTshapeB{m}{n}{\bh}{\bk}$  & Strict platform tangle shape multicategory.\\
    $\mCKTshapeB{m}{n}{\bh}{\bk}$  & Groupoid-enriched platform tangle shape multicategory.\\
    $\mTinvResB{\Tan}{\bh}{\bk}$ & Restriction of $\mTinv{\ou{\Tan}}$ to $\CCat{N}\ttimes \mCKTshapeB{m}{n}{\bh}{\bk}$.\\
    $\mJ$ & Absorbing subfunctor of $\mTinvResB{\Tan}{\bh}{\bk}$ corresponding to $\CKTfuncIdealB{\Tan}{\bh}{\bk}$.\\
    $\mCKTinvBtemp{\Tan}{\bh}{\bk}$ & Quotient functor $\mTinvResB{\Tan}{\bh}{\bk}/\mJ$ refining platform cube of bimodules.\\
    $\mCKTinvB{\Tan}{\bh}{\bk}$ & Minor modification of $\mCKTinvBtemp{\Tan}{\bh}{\bk}$ taking into account certain natural isomorphisms.\\
    $\CKTSpecBimB{\Tan}{\bh}{\bk}$ & Spectral bimodule refining $\CKTfuncB{\Tan}{\bh}{\bk}$.\\
    $\CKTSpecBimBig{\Tan}$ & Spectral bimodule refining $\CKTfuncBig{\Tan}$, given by $\displaystyle\coprod_{\substack{h,k\\m-n=2(h-k)}}\CKTSpecBim{\Tan}{m-h}{h}{n-k}{k}$.\\
    \bottomrule
  \end{tabular}
  \caption{\textbf{Notation used in the homotopical refinements.} The notions in the top are used for spectral refinements of arc algebras and bimodules, and the bottom ones are used for spectral refinements of platform algebras and bimodules. Compare with their homological versions in Table~\ref{tab:notation}.}
  \label{tab:hom-notation}
\end{table}

Just like the definitions of the ordinary platform algebras and
modules, we define the spectral platform algebras by first restricting
and then quotienting. We start by restricting---the analogues of
$\CKTalgTB{n}{\bk}$ and
$(\jmath\circ \imath^{k_1-h_1}\otimes \jmath)^*\KTfunc{\ou{\Tan}}$.

\begin{convention}\label{conv:h-is-big}
  In Section~\ref{sec:CK-background} we recalled how to associate an
  algebra $\CKTalg{n}{k_1}{k_2}$ to a tuple of integers $(n,k_1,k_2)$,
  so $\CKTalg{2n}{0}{0}$ is Khovanov's platform algebra and
  $\CKTalg{n}{k_1}{k_2}$ is isomorphic to one of the Chen-Khovanov
  algebras whenever $k_1+k_2\geq n$. In the rest of the paper, we only
  consider the case that $k_1+k_2\geq n$ (and similarly, tuples
  $(m,h_1,h_2)$ with $h_1+h_2\geq m$), so we are only considering the
  Chen-Khovanov case.
\end{convention}

Recall from Section~\ref{sec:shape-multicats} that we have full submulticategories 
\[\mCKAshapeB{n}{\bk}\subset \mHshape{n+k_1+k_2},\qquad\mCKTshapeB{m}{n}{\bh}{\bk}\subset \mTshape{m+h_1+h_2}{n+k_1+k_2},\qquad\CCat{N}\ttimes \mCKTshapeB{m}{n}{\bh}{\bk}\subset \CCat{N}\ttimes \mTshape{m+h_1+h_2}{n+k_1+k_2}.
\]
The multifunctor $\mHinv{n+k_1+k_2}$ restricts to a multifunctor 
\[
  \mHinvResB{n}{\bk}\co \mCKAshapeB{n}{\bk}\to\mBurnside.
\]

Next, fix a flat $(m,n)$-tangle $\Tan$ and pairs $\bh=(h_1,h_2)$ and
$\bk=(k_1,k_2)$ with $h_1-h_2=k_1-k_2$. Without loss of
generality, assume that $h_1\leq k_1$. Recall that $\ou{\Tan}$ is
the result of adding $k_1$ horizontal strands below $\Tan$ and
$k_2$ horizontal strands above $\Tan$. The multifunctor
$\mTinv{\ou{\Tan}}$ restricts to a multifunctor
\[
  \mTinvResB{\Tan}{\bk}{\bk}\co \mCKTshapeB{m}{n}{\bk}{\bk}\to \mBurnside.
\]
The map $\imath\co \Crossingless{m+h_1+h_2}\into \Crossingless{m+h_1+h_2+2}$ induces an isomorphism
\[
  \imath\co \mCKTshape{m}{n}{h_1}{h_2}{k_1}{k_2}\to \mCKTshape{m}{n}{h_1+1}{h_2+1}{k_1}{k_2}
\]
(cf.~Convention~\ref{conv:h-is-big}).
Let
\[
  \mTinvResB{\Tan}{\bh}{\bk}=\mTinvResB{\Tan}{\bk}{\bk}\circ \imath^{k_1-h_1}\co \mCKTshapeB{m}{n}{\bh}{\bk}\to \mBurnside.
\]

More generally, given an $(m,n)$-tangle $\Tan$ with $N$ crossings,
$\mTinv{\ou{\Tan}}$ restricts to a multifunctor
\[
  \mTinvResB{\Tan}{\bk}{\bk}\co \CCat{N}\ttimes\mCKTshapeB{m}{n}{\bk}{\bk}\to \mBurnside.
\]
The map $\imath$ induces an isomorphism
\[
  \imath\co
  \CCat{N}\ttimes\mCKTshape{m}{n}{h_1}{h_2}{k_1}{k_2}\to
  \CCat{N}\ttimes\mCKTshape{m}{n}{h_1+1}{h_2+1}{k_1}{k_2},
\]
and we again let
\[
  \mTinvResB{\Tan}{\bh}{\bk}=\mTinvResB{\Tan}{\bk}{\bk}\circ \imath^{k_1-h_1}\co \CCat{N}\ttimes\mCKTshapeB{m}{n}{\bh}{\bk}\to \mBurnside.
\]

Given $(a_1,a_2)\in \Ob(\mCKAshapeB{n}{\bk})$ define
$\mI(a_1,a_2)\subset \mHinvResB{n}{\bk}(a_1,a_2)$ to be the set of Khovanov
generators for $a_1\Wmirror{a_2}$ which are in the ideal
$\CKTidealB{n}{\bk}(a_1,a_2)$.  Similarly, given
$o\in\Ob(\CCat{N}\ttimes \mCKTshapeB{m}{n}{\bh}{\bk})$, if $o$ has the
form $(a_1,a_2)$ or $(b_1,b_2)$ define $\mJ(o)\subset \mTinvResB{\Tan}{\bh}{\bk}(o)$
to be $\mI(o)$. If $o$ has the form $(v,a,T,b)$, define $\mJ(v,a,T,b)$
to be the set of Khovanov generators for
$\imath^{k_1-h_1}(a)\ou{\Tan_v}\Wmirror{b}$ which are in the
submodule $\CKTfuncIdealB{\Tan_v}{\bh}{\bk}(a,b)$.

\begin{lemma}\label{lem:mI-absorbing}
  The subsets $\mI$ and $\mJ$ form absorbing subfunctors.
\end{lemma}
\begin{proof}
  The proof is the same as the proofs of Lemma~\ref{lem:CK-ideal} and
  Proposition~\ref{prop:CK-Ideal}. Alternatively, we can deduce this
  lemma from the combinatorial case. By
  Property~\ref{item:Psi-mmorph}, for the multi-functor $\Psi$,
  $s^{-1}(c_1,\dots,c_\ell)\cap t^{-1}(d)=\emptyset$ unless $d$ appears
  with non-zero coefficient in the product $c_1\cdot \cdots\cdot c_\ell$
  in the combinatorial Khovanov bimodule. So, it follows from the fact
  that $\CKTidealB{n}{\bk}$ is an ideal and
  $\CKTfuncIdealB{\Tan}{\bh}{\bk}$ is a submodule
  (Lemma~\ref{lem:CK-ideal} and Proposition~\ref{prop:CK-Ideal}), in
  conjunction with the fact that the functor
  $\Forget\from\BurnsideCat\to\GrAbelianGroups$ takes each morphism to a
  non-negative matrix, that $\mI$ and $\mJ$ are absorbing subfunctors.
\end{proof}

By Lemma~\ref{lem:mI-absorbing}, there are quotient functors
\begin{align*}
  \mAinvB{n}{\bk}&=\bigl(\mHinvResB{n}{\bk}/\mI\bigr)\co \mCKAshapeB{n}{\bk}\to\mBurnside\\
  \mCKTinvBtemp{\Tan}{\bh}{\bk}&=\bigl(\mTinvResB{T}{\bh}{\bk}/\mJ\bigr)\co \CCat{N}\ttimes \mCKTshapeB{m}{n}{\bh}{\bk}\to\mBurnside.
\end{align*}

Now, we fiddle around a little to avoid having to keep track of
natural isomorphisms between functors.
There is a canonical isomorphism of multicategories
\[
  \gamma=\iota^{h_2-h}\co \CCat{N}\ttimes\mCKTshape{m}{n}{m-h}{h}{n-k}{k}\to \CCat{N}\ttimes\mCKTshape{m}{n}{h_1}{h_2}{k_1}{k_2}
\]
where $m-2h=h_1-h_2$ and $n-2k=k_1-k_2$. We would like to say that
$\mCKTinvBtemp{\Tan}{\bh}{\bk}\circ
\gamma|_{\mCKAshape{m}{m-h}{h}}=\mAinv{m}{m-h}{h}$ and 
$\mCKTinvBtemp{\Tan}{\bh}{\bk}\circ
\gamma|_{\mCKAshape{n}{n-k}{k}}=\mAinv{n}{n-k}{k}$. This is not quite
true, but is true up to the following notion of natural isomorphism:

\begin{definition}\cite[Definition 3.22]{LLS-kh-tangles}
  Let $\Cat$ be a multicategory enriched in groupoids and
  $F,G\co\Cat\to\mBurnside$ multifunctors.  A \emph{natural
    isomorphism} $\eta\co F\to G$ consists of:
  \begin{itemize}[leftmargin=*]
  \item For each object $x\in\Cat$ a bijection of graded sets
    $\eta_x\co F(x)\to G(x)$, and
  \item For each multimorphism $f\in\Cat(x_1,\dots,x_n;y)$, a bijection of
    graded sets $\eta_f\co F(f)\to G(f)$
  \end{itemize}
  such that:
  \begin{enumerate}[leftmargin=*]
  \item For any objects $x_1,\dots,x_n,y\in\Cat$ and multimorphism
    $f\in\Cat(x_1,\dots,x_n;y)$, the following diagram commutes:
    \[
      \begin{tikzpicture}
        \node at (0,0) (Ff) {$F(f)$};
        \node at (-2,-1) (Fx) {$F(x_1)\times\cdots\times F(x_n)$};
        \node at (2,-1) (Fy) {$F(y)$};
        \node at (-2,-2) (Gx) {$G(x_1)\times\cdots\times G(x_n)$};
        \node at (2,-2) (Gy) {$G(y)$};
        \node at (0,-3) (Gf) {$G(f)$};
        \draw[->] (Ff) to (Fx);
        \draw[->] (Ff) to (Fy);
        \draw[->] (Fx) to node[left]{$\eta_{x_1}\times\cdots\times\eta_{x_n}$} (Gx);
        \draw[->] (Fy) to node[right]{$\eta_y$} (Gy);
        \draw[->] (Gf) to (Gx);
        \draw[->] (Gf) to (Gy);
        \draw[->] (Ff) to node[right]{$\eta_f$} (Gf);
      \end{tikzpicture}
    \]
  \item For every $f,g\in\Cat(x_1,\dots,x_n;y)$ and $\phi\in\Hom(f,g)$,
    \[
      \eta_g\circ F(\phi)=G(\phi)\circ \eta_f\co F(f)\to G(g).
    \]
  \item For every $g\in\Cat(y_1,\dots,y_n;z)$ and $f_1,\dots,f_n$ with $f_i\in\Cat(x_{i,1},\dots,x_{i,m_i},y_i)$,
    \[
      \eta_{g\circ(f_1,\dots,f_n)}=\eta_g\circ(\eta_{f_1},\dots,\eta_{f_n}).
    \]
  \end{enumerate}
\end{definition}

\begin{lemma}\label{lem:iota-spec-alg-map}
  If $\imath$ denotes the canonical isomorphism
  $\mCKAshape{m}{h_1}{h_2}\stackrel{\cong}{\longrightarrow}
  \mCKAshape{m}{h_1+1}{h_2+1}$ induced by
  $\imath\co\Crossingless{m+h_1+h_2}\into\Crossingless{m+h_1+h_2+2}$
  then there is a natural isomorphism between $\mAinv{m}{h_1}{h_2}$
  and $\mAinv{m}{h_1+1}{h_2+1}\circ\imath$. Further, this isomorphism
  sends $\mI(o)\subset\mHinvRes{m}{h_1}{h_2}(o)$ bijectively to
  $\mI(\imath(o))\subset \mHinvRes{m}{h_1+1}{h_2+1}(\imath(o))$ for
  each $o\in\Ob(\mCKAshape{m}{h_1}{h_2})$.  Thus,
  $\mCKTinvBtemp{\Tan}{\bh}{\bk}\circ \gamma|_{\mCKAshape{m}{m-h}{h}}$
  and $\mAinv{m}{m-h}{h}$ (respectively
  $\mCKTinvBtemp{\Tan}{\bh}{\bk}\circ \gamma|_{\mCKAshape{n}{n-k}{k}}$
  and $\mAinv{n}{n-k}{k}$)
  are naturally isomorphic.
\end{lemma}
\begin{proof}
  This is immediate from the definitions.
\end{proof}

\begin{definition}\label{def:make-restrict-strict}
  Define
  $\mCKTinvB{\Tan}{\bh}{\bk}\co \CCat{N}\ttimes
  \mCKTshape{m}{n}{m-h}{h}{n-k}{k}\to\mBurnside$ as follows. Let
  $\eta$ (respectively $\xi$) be the natural isomorphism from
  $\mAinv{m}{m-h}{h}$ to
  $\mCKTinvBtemp{\Tan}{\bh}{\bk}\circ\gamma|_{\mCKAshape{m}{m-h}{h}}$
  (respectively $\mAinv{n}{n-k}{k}$ to
  $\mCKTinvBtemp{\Tan}{\bh}{\bk}\circ\gamma|_{\mCKAshape{n}{n-k}{k}}$)
  from Lemma~\ref{lem:iota-spec-alg-map}.
  \begin{itemize}[leftmargin=*]
  \item For objects $x\in\Ob(\mCKAshape{m}{m-h}{h})$, define
    $\mCKTinvB{\Tan}{\bh}{\bk}(x)=\mAinv{m}{m-h}{h}(x)$. 
  \item For objects $x\in\Ob(\mCKAshape{n}{n-k}{k})$, define
    $\mCKTinvB{\Tan}{\bh}{\bk}(x)=\mAinv{n}{n-k}{k}(x)$.
  \item For all other objects $x$, define
    $\mCKTinvB{\Tan}{\bh}{\bk}(x)=\mCKTinvBtemp{\Tan}{\bh}{\bk}(\gamma(x))$.
  \item Given objects
    $x_1,\dots,x_n,y\in\Ob(\CCat{N}\ttimes\mCKTshape{m}{n}{m-h}{h}{n-k}{k})$ and a basic
    multimorphism (see~\cite[\S2.4]{LLS-kh-tangles})
    $f\in \Hom_{\CCat{N}\ttimes\mCKTshape{m}{n}{m-h}{h}{n-k}{k}}(x_1,\dots,x_n;y)$:
    \begin{itemize}[leftmargin=*]
    \item if all the $x_i$ (and hence $y$) are in
      $\Ob(\mCKAshape{m}{m-h}{h})$ define
      $
        \mCKTinvB{\Tan}{\bh}{\bk}(f)=\mAinv{m}{m-h}{h}(f),
      $
    \item if all the $x_i$ (and hence $y$) are in $\Ob(\mCKAshape{n}{n-k}{k})$ define
      $
        \mCKTinvB{\Tan}{\bh}{\bk}(f)=\mAinv{n}{n-k}{k}(f),
      $
    \item and otherwise, if $x_1,\dots,x_i\in\Ob(\mCKAshape{m}{m-h}{h})$ and $x_{i+2},\dots,x_n\in \Ob(\mCKAshape{n}{n-k}{k})$,
    define
    $
      \mCKTinvB{\Tan}{\bh}{\bk}(f)=\mCKTinvBtemp{\Tan}{\bh}{\bk}(\gamma(f)),
    $
    viewed as a correspondence from
    \[
      \mAinv{m}{m-h}{h}(x_{1})\times\cdots\times \mAinv{m}{m-h}{h}(x_{i})
      \times\mCKTinvBtemp{\Tan}{\bh}{\bk}(\gamma(x_{i+1}))\times
      \mAinv{n}{n-k}{k}(x_{i+2})\times\cdots\times \mAinv{n}{n-k}{k}(x_{n})
    \]
    to $\mCKTinvBtemp{\Tan}{\bh}{\bk}(\gamma(y))$ by composing the
    source map with
    $\eta_{x_1}^{-1}\times\cdots\times\eta_{x_i}^{-1}\times\Id\times\xi_{x_{i+1}}^{-1}\times\cdots\times\xi_{x_n}^{-1}$.
    \end{itemize}    
  \item On general multimorphisms, $\mCKTinvB{\Tan}{\bh}{\bk}$ is the
    composition of its values on basic multimorphisms.
  \item On $2$-morphisms, $\mCKTinvB{\Tan}{\bh}{\bk}$ is induced by
    the values $\mCKTinvBtemp{\Tan}{\bh}{\bk}$, $\mAinvB{m}{\bh}$,
    $\mAinvB{n}{\bk}$, and the $\eta_f$ and $\xi_f$.
  \end{itemize}
\end{definition}

Consider the stable functor $(\mCKTinvB{\Tan}{\bh}{\bk},N_+)$ (where $N_+$ is the
number of positive crossings of $T$). Define
\begin{align*}
  \CKTSpecCat{n}{k_1}{k_2}&=\CRealize{\mAinv{n}{k_1}{k_2}} &
  \CKTSpecBimB{\Tan}{\bh}{\bk}&=\CRealize{\mCKTinvB{\Tan}{\bh}{\bk}}.
\end{align*}
\begin{lemma}
  The spectra $\CKTSpecBimB{\Tan}{\bh}{\bk}(a,b)$ form a spectral
  bimodule over $\CKTSpecCatB{m}{\bh}$ and $\CKTSpecCatB{n}{\bk}$.
\end{lemma}
\begin{proof}
  In view of Lemma~\ref{lem:iota-spec-alg-map} and the discussion in
  Section~\ref{sec:shape-multicats}, all that remains is to verify
  that rectifying $K\circ \mCKTinvB{\Tan}{\bh}{\bk}$ and then
  restricting to the sub-multicategory $\mCKAshapeSB{n}{\bk}$ is the
  same as first restricting to $\mCKAshapeB{n}{\bk}$ and then
  rectifying. This follows from the fact that
  $\mCKAshapeB{n}{\bk}\subset \mCKTshapeB{m}{n}{\bh}{\bk}$ has no
  morphisms in (i.e., is \emph{blockaded} in the language
  of~\cite[Proposition 2.39]{LLS-kh-tangles})~\cite[Lemma
  2.44]{LLS-kh-tangles}.
\end{proof}
Let
\[
  \CKTSpecCatBig{n}=\coprod_{k=0}^n\CKTSpecCat{n}{n-k}{k}.
\]
Let $\CKTSpecBimBig{\Tan}$ denote the bimodule over
$\CKTSpecCatBig{m}$ and $\CKTSpecCatBig{n}$ induced by the various
bimodules $\CKTSpecBim{\Tan}{m-h}{h}{n-k}{k}$.

Composing the singular chain functor $C_*$ with
$\CKTSpecCatB{n}{\bk}$ gives a differential graded category
(category enriched in chain complexes), and composing $C_*$ with
$\CKTSpecBimB{\Tan}{\bh}{\bk}$ gives a bimodule over
$C_*\CKTSpecCatB{m}{\bh}$ and
$C_*\CKTSpecCatB{n}{\bk}$. Since $\CKTSpecCatB{n}{\bk}$
has finitely many objects, by taking the direct sum of the morphism
spaces we can view $C_*\CKTSpecCatB{n}{\bk}$ as a differential
graded algebra and $C_*\CKTSpecBimB{\Tan}{\bh}{\bk}$ as a differential
graded module over it.

The three main theorems are:

\begin{theorem}\label{thm:de-spectrify}
  There is a quasi-isomorphism $C_*\CKTSpecCatB{n}{\bk}\simeq \CKTalgB{n}{\bk}$ and, for any tangle $\Tan$, a quasi-isomorphism $C_*\CKTSpecBimB{\Tan}{\bh}{\bk}\simeq \CKTfuncB{\Tan}{\bh}{\bk}$ intertwining the module structures in the obvious sense.
\end{theorem}

\begin{theorem}\label{thm:invariance}
  If $\Tan_1$ and $\Tan_2$ are equivalent tangles then $\CKTSpecBimB{\Tan_1}{\bh}{\bk}$ and $\CKTSpecBimB{\Tan_2}{\bh}{\bk}$ are weakly equivalent spectral modules.
\end{theorem}

\begin{theorem}\label{thm:pairing}
  If $\Tan_1$ is an $(m,n)$-tangle and $\Tan_2$ is an $(n,p)$-tangle then for
  any $h_1,h_2,k_1,k_2,\ell_1,\ell_2$ with
  \[
    h_1-h_2=k_1-k_2=\ell_1-\ell_2,\ 
    h_1+h_2\geq m,\ k_1+k_2\geq n,\ \text{and }\ \ell_1+\ell_2\geq p
  \]
  there is a weak equivalence of spectral bimodules
  \[
    \CKTSpecBimB{\Tan_1\Tan_2}{\bh}{\bl}\simeq \CKTSpecBimB{\Tan_1}{\bh}{\bk}\DTP_{\CKTSpecCat{n}{n-k}{k}}\CKTSpecBimB{\Tan_2}{\bk}{\bl},
  \]
  where the right side denotes the (derived) tensor product of spectral
  bimodules.
\end{theorem}

\begin{proof}[Proof of Theorem~\ref{thm:de-spectrify}]
  The proof is the same as for the arc algebras~\cite[Proposition 4.2]{LLS-kh-tangles} and is left to the reader.
\end{proof}

\begin{lemma}
  \label{lem:restrict-invt-spectra}
  There are equivalences
  \[
    \CKTSpecBim{\Tan}{h_1}{h_2}{k_1}{k_2}\simeq
    \CKTSpecBim{\Tan}{h_1+1}{h_2+1}{k_1}{k_2} \simeq
    \CKTSpecBim{\Tan}{h_1}{h_2}{k_1+1}{k_2+1}
  \]
  of bimodules over $\CKTSpecCat{m}{m-h}{h}$ and
  $\CKTSpecCat{n}{n-k}{k}$.
\end{lemma}
\begin{proof}
  It is immediate from the definitions that there is a natural
  isomorphism
  $\mCKTinv{\Tan}{h_1}{h_2}{k_1}{k_2}\to
  \mCKTinv{\Tan}{h_1+1}{h_2+1}{k_1}{k_2}$ of multifunctors from
  $\CCat{N}\ttimes\mCKTshape{m}{n}{m-h}{h}{n-k}{k}$ to
  $\mBurnside$. This, in turn, implies that
  $\mCKTinv{\Tan}{h_1}{h_2}{k_1}{k_2}$ and
  $\mCKTinv{\Tan}{h_1+1}{h_2+1}{k_1}{k_2}$ are stably equivalent.
  So, the first statement follows from
  Lemma~\ref{lem:equiv-gives-equiv}. The proof of the second statement
  is similar.
\end{proof}

\begin{proof}[Proof of Theorem~\ref{thm:invariance}]
  Reordering crossings induces an automorphism of the cube and a
  corresponding equivalence of homotopy colimits. For invariance under
  Reidemeister moves, we lift the proof of
  Theorem~\ref{thm:comb-invariance}.  By
  Lemma~\ref{lem:restrict-invt-spectra}, we may assume
  $(h_1,h_2)=(k_1,k_2)$. As in
  Theorem~\ref{thm:comb-invariance}, we focus on a Reidemeister II
  move; the other cases are similar. With notation as in the proof of
  Theorem~\ref{thm:comb-invariance}, it follows from the definitions
  that the subcomplex $C_1$ corresponds to an insular subfunctor
  $F_1$ of $\mTinv{\ou{\Tan}}$
  (Definition~\ref{def:insular-subfunctor}). The quotient functor
  $F_2=\mTinv{\ou{\Tan}}/F_1$ corresponds to the complex $C_2$, and
  has a further insular subfunctor $F_4\subset F_2$ naturally
  isomorphic to $\mTinv{\ou{\Tan'}}$ so that $F_3=F_2/F_4$ corresponds
  to the acyclic complexes $C_3$.

  Each $F_i$ restricts to a functor $G_i\co \CCat{N}\ttimes
  \mCKTshapeB{m}{n}{\bk}{\bk}\to \mBurnside$, and $G_1$ is an insular
  subfunctor of $\mTinvResB{\Tan}{\bk}{\bk}$ with quotient functor $G_2$,
  while $G_4$ is an insular subfunctor of $G_2$ with quotient functor
  $G_3$. By Lemma~\ref{lem:insular-absorbing}, $G_1$ induces an
  insular subfunctor $\overline{G}_1$ of
  $\mTinvResB{\Tan}{\bk}{\bk}/\mI=\mCKTinvB{\Tan}{\bk}{\bk}$, and $G_4$
  induces an insular subfunctor $\overline{G}_4$ of
  $\overline{G}_2=\mCKTinvB{\Tan}{\bk}{\bk}/\overline{G}_1$.  Applying the
  realization procedure gives a zig-zag of spectral bimodules
  $\CKTSpecBimB{\Tan}{\bk}{\bk}\to \CRealize{\ol{G}_2}\leftarrow
  \CKTSpecBimB{\Tan'}{\bk}{\bk}$.  From the proof of
  Theorem~\ref{thm:comb-invariance}, these maps induce isomorphisms on
  homology, and hence are stable homotopy equivalences, as desired.
\end{proof}

\begin{proof}[Proof of Theorem~\ref{thm:pairing}]
  By Lemma~\ref{lem:restrict-invt-spectra}, it suffices to prove
  Theorem~\ref{thm:pairing} when $h_1=k_1=\ell_1$ (and so
  $h_2=k_2=\ell_2$).
  
  We start by recalling the proof of the gluing theorem for the
  spectral Khovanov bimodules~\cite[\S5]{LLS-kh-tangles}. We
  introduced a \emph{gluing shape multicategory} $\mGlueS{2m}{2n}{2p}$
  (denoted $\mc{U}_{m,n,p}^0$ in \cite{LLS-kh-tangles}) with six
  kinds of objects: pairs $(a_1,a_2)\in
  \Crossingless{2m}\times\Crossingless{2m}$, $(b_1,b_2)\in
  \Crossingless{2n}\times\Crossingless{2n}$, $(c_1,c_2)\in
  \Crossingless{2p}\times\Crossingless{2p}$, triples $(a,T_1,b)$ with
  $a\in\Crossingless{2m}$ and $b\in\Crossingless{2n}$, triples
  $(b,T_2,c)$ with $b\in\Crossingless{2n}$ and
  $c\in\Crossingless{2p}$, and triples $(a,T_1T_2,c)$ with
  $a\in\Crossingless{2m}$ and $c\in\Crossingless{2p}$. The categories
  $\SmTshape{2m}{2n}$, $\SmTshape{2n}{2p}$, and $\SmTshape{2m}{2p}$ are full
  subcategories of $\mGlueS{2m}{2n}{2p}$, and there is also a unique
  multimorphism
  \begin{multline*}
    (a_1,a_2),\dots,(a_{\alpha-1},a_\alpha),(a_\alpha,T_1,b_1),(b_1,b_2),\dots,(b_{\beta-1},b_\beta),(b_\beta,T_2,c_1),(c_1,c_2),\dots,(c_{\gamma-1},c_\gamma)\\\to (a_1,T_1T_2,c_\gamma).
  \end{multline*}
  There is also a thickened version $\mGlue{2m}{2n}{2p}$ and a thickened
  product with the cube $\CCat{N_1+N_2}\ttimes\mGlue{2m}{2n}{2p}$. We
  then construct a functor $F\co
  \CCat{N_1+N_2}\ttimes\mGlue{2m}{2n}{2p}\to\mBurnside$ extending
  $\mTinv{\Tan_1}$, $\mTinv{\Tan_2}$, and $\mTinv{\Tan_1\Tan_2}$. The functor
  $F$ induces a map of spectral bimodules
  \[
    \KTSpecBim{\Tan_1}\DTP_{\KTSpecCat{2n}}\KTSpecBim{\Tan_1}\to \KTSpecBim{\Tan_1\Tan_2}
  \]
  and the induced map of singular chain complexes agrees with the map
  of Khovanov complexes of bimodules (and so is an equivalence by
  Whitehead's theorem).
  
  Now, let $\mCKGlueS{m}{n}{p}$ be the full subcategory of
  $\mGlueS{m+k_1+k_2}{n+k_1+k_2}{p+k_1+k_2}$ spanned by objects as
  above but with $a_i\in\rCrossinglessB{m}{\bk}$,
  $b_i\in\rCrossinglessB{n}{\bk}$, and
  $c_i\in\rCrossinglessB{p}{\bk}$; define the thickened version
  $\mCKGlue{m}{n}{p}$ similarly. Let $\wt{F}$ be the restriction of
  $F$ to $\CCat{N_1+N_2}\ttimes\mCKGlue{m}{n}{p}$. For an object $o$ of
  $\CCat{N_1+N_2}\ttimes\mCKGlue{m}{n}{p}$ define $\mK(o)\subset F(o)$ to be
  $\mJ_{\Tan_1}(o)$ if $o\in \Ob(\mCKTshapeB{m}{n}{\bk}{\bk})$,
  $\mJ_{\Tan_2}(o)$ if $o\in\Ob(\mCKTshapeB{n}{p}{\bk}{\bk})$, and
  $\mJ_{\Tan_1\Tan_2}(o)$ if $o\in\Ob(\mCKTshapeB{m}{p}{\bk}{\bk})$.
  (These definitions of $\mK$
  agree on the overlaps of these subcategories.)

  The proof of Lemma~\ref{lem:glue-is-chain-map} shows that $\mK$ is
  an absorbing subfunctor.

  As in Section~\ref{sec:Hn-spec-review} we can realize the quotient
  functor $\wt{F}/\mK$ to obtain a functor $\Realize{\wt{F}/\mK}\co
  \mCKGlueS{m}{n}{p}\to\GrSpectra$. Similarly to
  Section~\ref{sec:shape-multicats}, we can
  reinterpret $\Realize{\wt{F}/\mK}$ as a map of spectral bimodules
  \begin{equation}\label{eq:pair-spec-bim-pf}
    \CKTSpecBimB{\Tan_1}{\bk}{\bk}\DTP_{\CKTSpecCatB{n}{\bk}}\CKTSpecBimB{\Tan_2}{\bk}{\bk}\to
    \CKTSpecBimB{\Tan_1\Tan_2}{\bk}{\bk}.
  \end{equation}
  (This involves a little fiddling as in Definition~\ref{def:make-restrict-strict}.)
  As in the arc algebra case~\cite[Lemma 5.6]{LLS-kh-tangles}, taking
  singular chains this is the gluing map
  \[
    \CKTfuncB{{\Tan_1}}{\bk}{\bk}\DTP_{\CKTalgB{n}{\bk}}\CKTfuncB{{\Tan_2}}{\bk}{\bk}\to 
    \CKTfuncB{{\Tan_1\Tan_2}}{\bk}{\bk}
  \]
  from Theorem~\ref{thm:CK-pairing}. By Whitehead's theorem, the
  map~\eqref{eq:pair-spec-bim-pf} is a weak equivalence, as desired.
\end{proof}

\section{Topological Hochschild homology}\label{sec:THH}
Let $\Tan$ be an $(n,n)$-tangle. We can form the topological
Hochschild homology of $\CKTSpecCatBig{n}$ with coefficients in
$\CKTSpecBimBig{\Tan}$, which we write $\THH(\CKTSpecCatBig{n};\CKTSpecBimBig{\Tan})$ or
$\THH(\CKTSpecBimBig{\Tan})$. The spectral categories
$\CKTSpecCatBig{n}$ are pointwise cofibrant (see~\cite[Lemma
4.5]{LLS-kh-tangles}), so the topological Hochschild homology can be
obtained as the homotopy colimit of the diagram
\begin{align*}
    \cdots
    &\mathrel{\substack{\textstyle\rightarrow\\[-0.5ex]
    \textstyle\rightarrow \\[-0.5ex]
    \textstyle\rightarrow\\[-0.5ex]
    \textstyle\rightarrow}}
    \!\!\!\!\!\!\coprod_{a_1,a_2,a_3\in\Ob(\CKTSpecCatBig{n})}\!\!\!\!\!\!\CKTSpecBimBig{\Tan}(a_3,a_1)\smas\CKTSpecCatBig{n}(a_1,a_2)\smas \CKTSpecCatBig{n}(a_2,a_3)\\
    &\mathrel{\substack{\textstyle\rightarrow\\[-0.5ex]
    \textstyle\rightarrow \\[-0.5ex]
    \textstyle\rightarrow}}
    \!\!\!\!\coprod_{a_1,a_2\in\Ob(\CKTSpecCatBig{n})}\!\!\!\!\CKTSpecBimBig{\Tan}(a_2,a_1)\smas\CKTSpecCatBig{n}(a_1,a_2)\rightrightarrows\!\!\coprod_{a_1\in\Ob(\Cat)}\!\!\CKTSpecBimBig{\Tan}(a_1,a_1)
\end{align*}
where the horizontal maps are given by the compositions in
$\CKTSpecCatBig{n}$ and its actions on $\CKTSpecBimBig{\Tan}$
\cite[Proposition 3.5]{BM-top-spectral}.

\begin{proposition}
  There is an isomorphism
  \[
    H_*\THH(\CKTSpecBimBig{\Tan})\cong \HH_*(\CKTfuncBig{\Tan})
  \]
\end{proposition}
\begin{proof}
  The proof is the same as the analogous result for the spectral
  Khovanov bimodules~\cite[Proposition 7.5]{LLS-kh-tangles}.
\end{proof}

Recall that the Hochschild homology of the Chen-Khovanov bimodules has
another interpretation. Given an $(n,n)$-tangle $\Tan$, we can form
the closure $\anclose{\Tan}$ of $\Tan$ in the annulus
$S^1\times[0,1]$. Asaeda-Przytycki-Sikora constructed a Khovanov
homology for links in thickened surfaces~\cite{APS-kh-surfaces} which, in particular, gives
an invariant $\AKh(\anclose{\Tan})$, the \emph{annular
  Khovanov homology} of $\anclose{\Tan}$. (This case was further
studied by Roberts~\cite{Roberts-kh-dcov}, Grigsby-Wehrli~\cite{GW-kh-sutured},
and others.) Specifically,
there is a filtration on the Khovanov complex of the closure of $\Tan$
in $\RR^3$, coming from using the labels of circles by $1$ or $X$ to
orient them and then considering the winding number around the
axis. The invariant $\AKh(\anclose{\Tan};\ell)$ is the homology of the
associated graded complex to this filtration, in winding number grading $\ell$. As such,
$\AKh(\anclose{\Tan})$ is tri-graded, by the homological, quantum, and
winding number gradings.

\begin{convention}
  Fix $k$ with $0\leq k\leq n$ and let $\bbk=(n-k,k)$.
\end{convention}

Beliakova-Putyra-Wehrli relate the annular Khovanov homology to the Chen-Khovanov invariants:
\begin{theorem}\cite[Theorem C]{BPW-Kh-HH}\label{thm:BPW}
  There is an isomorphism 
  \begin{equation}\label{eq:BPW}
    \HH_*(\CKTalgB{n}{\bbk};\CKTfuncB{\Tan}{\bbk}{\bbk})\cong \AKh(\anclose{\Tan};n-2k)\{n-2k\}.
  \end{equation}
\end{theorem}
(A special case was proved earlier by Auroux-Grigsby-Wehrli
in~\cite{AGW-kh-HH}.)

In order to prove a spectral refinement of Theorem~\ref{thm:BPW}, we
need the explicit map 
\[
  \HC_*(\CKTalgB{n}{\bbk};\CKTfuncB{\Tan}{\bbk}{\bbk})\to \AKhCx(\anclose{\Tan};n-2k)\{n-2k\}
\]
inducing the isomorphism~\eqref{eq:BPW}. While
Beliakova-Putyra-Wehrli's proof does not explicitly give the map,
their ideas easily extend to do so. We emphasize that we do not give
an independent proof of Theorem~\ref{thm:BPW}: for instance, the proof
of Lemma~\ref{lem:HH-id-braid} relies on Theorem~\ref{thm:BPW}; we
merely construct an explicit isomorphism.

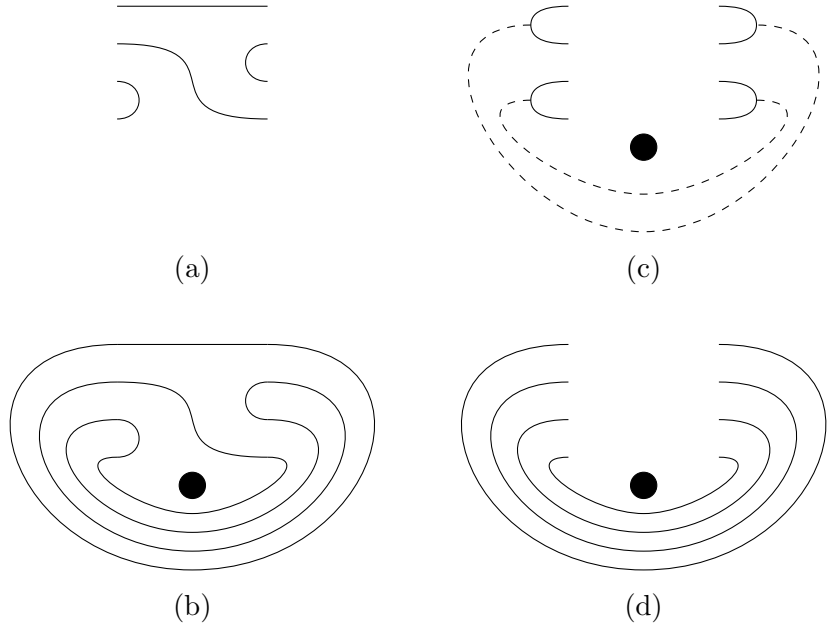
\begin{figure}
  \centering
  \begin{tikzpicture}
    \begin{scope}[xshift=0, yshift=0,xscale=.5,yscale=.5]
      \draw (0,0) to[out=0,in=0, looseness=2] (0,1);
      \draw (0,2) to[out=0,in=180, looseness=2] (4,0);
      \draw (0,3) to (4,3);
      \draw (4,1) to[out=180,in=180, looseness=2] (4,2);
      \node at (2,-4) (alabel) {(a)};
    \end{scope}
    %
    \begin{scope}[xshift=0,yshift=-4.5cm,xscale=.5,yscale=.5]
      \draw (0,0) to[out=0,in=0, looseness=2] (0,1);
      \draw (0,2) to[out=0,in=180, looseness=2] (4,0);
      \draw (0,3) to (4,3);
      \draw (4,1) to[out=180,in=180, looseness=2] (4,2);
      \draw (0,3) to[out=180,in=180, looseness=2] (2,-3) to[out=0,in=0, looseness=2] (4,3);
      \draw (0,2) to[out=180,in=180, looseness=2] (2,-2.5) to[out=0,in=0, looseness=2] (4,2);
      \draw (0,1) to[out=180,in=180, looseness=2] (2,-2) to[out=0,in=0, looseness=2] (4,1);
      \draw (0,0) to[out=180,in=180, looseness=1.5] (2,-1.5) to[out=0,in=0, looseness=1.5] (4,0);
      \draw[fill] (2,-0.75) circle [radius=.35];
      \node at (2,-4) (blabel) {(b)};
    \end{scope}
    %
    \begin{scope}[xshift=6cm,xscale=.5,yscale=.5]
      \draw (0,0) to[out=180,in=270] (-1,.5) to[out=90,in=180] (0,1);
      \draw (0,2) to[out=180,in=270] (-1,2.5) to[out=90,in=180] (0,3);
      \draw (4,0) to[out=0,in=270] (5,.5) to[out=90,in=0] (4,1);
      \draw (4,2) to[out=0,in=270] (5,2.5) to[out=90,in=0] (4,3);
      \draw[fill] (2,-0.75) circle [radius=.35];
      \draw[dashed] (-1,.5) to[out=180,in=180,looseness=1.5] (2,-2) to[out=0,in=0,looseness=1.5] (5,.5);
      \draw[dashed] (-1,2.5) to[out=180,in=180,looseness=1.5] (2,-3) to[out=0,in=0,looseness=1.5] (5,2.5);
      \node at (2,-4) (clabel) {(c)};
    \end{scope}
    %
    \begin{scope}[xshift=6cm,yshift=-4.5cm,xscale=.5,yscale=.5]
      \draw (0,3) to[out=180,in=180, looseness=2] (2,-3) to[out=0,in=0, looseness=2] (4,3);
      \draw (0,2) to[out=180,in=180, looseness=2] (2,-2.5) to[out=0,in=0, looseness=2] (4,2);
      \draw (0,1) to[out=180,in=180, looseness=2] (2,-2) to[out=0,in=0, looseness=2] (4,1);
      \draw (0,0) to[out=180,in=180, looseness=1.5] (2,-1.5) to[out=0,in=0, looseness=1.5] (4,0);
      \draw[fill] (2,-0.75) circle [radius=.35];
      \node at (2,-4) (clabel) {(d)};
    \end{scope}    
  \end{tikzpicture}
  \caption{\textbf{The annular closure.} (a) A (flat) tangle $\Tan$. (b)
    The annular closure $\anclose{\Tan}$. (c) $\Wmirror{a}\amalg a$
    for $a$ a particular crossingless matching, together with the
    cores of the $1$-handles (dashed) in the saddle cobordism from
    $\Wmirror{a}\amalg a$ to the identity braid. (d) The identity
    braid, inside the annulus.}
  \label{fig:ann-close}
\end{figure}

Let $\Tan$ be an $(n,n)$-tangle, and
$\anclose{\Tan}\subset \RR^2\setminus\{(0,0)\}$ the annular closure of
$\Tan$. (See Figure~\ref{fig:ann-close}.) We can view $\anclose{\Tan}$ as lying in $\RR^2$ and so, in
particular, can consider $\KhCx([\Tan])$. Given
$a\in\rCrossinglessB{n}{\bbk}$, define the map
\[
  A=A_{a,\Tan}\co \KTfunc{\ou{\Tan}}(a,a) \to
  \KhCx(\anclose{\ou{\Tan}})
\]
to be the map associated to the saddle cobordism from
$\Wmirror{a}\amalg a$ to the identity braid around the
annulus. (Again, see Figure~\ref{fig:ann-close}.)

\begin{proposition}\label{prop:A-map}
  The maps $A_{a,\Tan}$ satisfy the following properties:
  \begin{enumerate}[label=(\arabic*),leftmargin=*]
  \item \label{item:A-chain-map} Each $A_{a,\Tan}$ is a chain map.
  \item \label{item:A-filtered} The image of $A_{a,\Tan}$ lies in annular filtration $\leq 0$.
  \item \label{item:A-coinv} Given $a,b\in\rCrossinglessB{n}{\bbk}$, the following diagram commutes:
    \[
      \xymatrix{
        \KhCx(a\ou{\Tan}\Wmirror{b}\amalg b\Wmirror{a})\{2n\}\ar[r]^-{s_b}\ar[d]_{s_a} &
        \KhCx(a\ou{\Tan}\Wmirror{a})\{n\}\ar[d]^{A_{a,\Tan}}\\
        \KhCx(b\ou{\Tan}\Wmirror{b})\{n\}\ar[r]_-{A_{b,\Tan}} & \KhCx(\anclose{\ou{\Tan}}).
      }
    \]
    Here, the arrows labeled $s_a$ and $s_b$ are induced by the saddle
    cobordisms from $\Wmirror{a}\amalg a$ to the identity braid and
    $\Wmirror{b}\amalg b$ to the identity braid.
  \item \label{item:A-trace} Given $(n,n)$-tangles $\Tan_1$ and
    $\Tan_2$ and $a,b\in\rCrossinglessB{n}{\bbk}$, the following
    diagram commutes:
    \[
      \xymatrix{
        \KhCx(a\ou{\Tan_1}\Wmirror{b}\amalg b\ou{\Tan_2}\Wmirror{a})\{2n\}\ar[r]^-{s_b}\ar[d]_{s_a} &
        \KhCx(a\ou{\Tan_1\Tan_2}\Wmirror{a})\{n\}\ar[d]^{A_{a,\Tan}}\\
        \KhCx(b\ou{\Tan_2\Tan_1}\Wmirror{b})\{n\}\ar[r]_-{A_{b,\Tan}} & \KhCx(\anclose{\ou{\Tan_1\Tan_2}}=\anclose{\ou{\Tan_2\Tan_1}}).
      }
    \]
    Again, the arrows labeled $s_a$ and $s_b$ are induced by the saddle
    cobordisms from $\Wmirror{a}\amalg a$ to the identity braid and
    $\Wmirror{b}\amalg b$ to the identity braid.
  \end{enumerate}
\end{proposition}
\begin{proof}
  Point~\ref{item:A-chain-map} follows from far-commutativity of the
  cobordism maps. Point~\ref{item:A-filtered} follows from the facts
  that the cobordism maps respect the annular filtration and
  $a\ou{\Tan}\Wmirror{a}$ has winding number $0$.
  Point~\ref{item:A-coinv} is a special case of
  Point~\ref{item:A-trace}. Point~\ref{item:A-trace} again follows
  from far-commutativity of the cobordism maps.
\end{proof}

Let $\Filt_{\leq 0}\KhCx(\anclose{\ou{\Tan}})$ be the subcomplex of
$\KhCx(\anclose{\ou{\Tan}})$ in filtration $\leq 0$. By
Point~\ref{item:A-filtered} of Proposition~\ref{prop:A-map}, the image
of $A_{a,\Tan}$ is contained in $\Filt_{\leq 0}\KhCx(\anclose{\ou{\Tan}})$. Let
\[
  B\co \Filt_{\leq 0}\KhCx(\anclose{\ou{\Tan}})\to \Filt_{\leq 0}\KhCx(\anclose{\ou{\Tan}})/\Filt_{<0}\KhCx(\anclose{\ou{\Tan}})
\]
be projection to the associated graded complex.

Let $P_{L}$ be the image of the lower-left platform for $\ou{\Tan}$ in
the annulus $\RR^2\setminus\{(0,0)\}$ and let $P_{U}$ be the image of
the upper-left platform.  In $\anclose{\ou{\Tan}}$, there are several kinds
of circles:
\begin{enumerate}[label=(A-\roman*),leftmargin=*]
\item Circles which are disjoint from $P_L$ and $P_U$. 
\item Circles which pass through $P_L$ once and are disjoint from $P_U$. Call these \emph{lower horizontal circles}.
\item Circles which pass through $P_U$ once and are disjoint from
  $P_L$. Call these \emph{upper horizontal circles}.
\end{enumerate}

Observe that the complex $\KhCx(\anclose{\ou{\Tan}})$ decomposes as a
direct sum corresponding to the different ways of labeling the upper
and lower horizontal circles.
Define
\[
  C\co \Filt_{\leq 0}\KhCx(\anclose{\ou{\Tan}})/\Filt_{<0}\KhCx(\anclose{\ou{\Tan}})
  \to \AKhCx(\anclose{\Tan};n-2k)\{n-2k\}
\]
to be the result of projecting to the summand
where each of the $k$ upper horizontal circles is labeled $1$ and each
of the $(n-k)$ lower horizontal circles is labeled $X$, and then
forgetting the lower and upper horizontal circles. It is clear that
the image of this map lies in the summand with winding number grading
$n-2k$.

\begin{lemma}\label{lem:CBA-descends}
  The composition $C\circ B\circ A$ vanishes on
  \[
    \CKTfuncIdealB{\Tan}{\bbk}{\bbk}(a,a)\subset \KTfunc{[\ou{\Tan}]}(a,a).
  \]
\end{lemma}
\begin{proof}
  Unsurprisingly, the proof is a case analysis. Fix a generator
  $y\in \CKTfuncIdealB{\Tan}{\bbk}{\bbk}(a,a)$. Suppose first that
  $(a\ou{\Tan}\Wmirror{a},y)$ has a type II circle $Z$ labeled $X$. If
  the circle passes through the upper platforms then $A_{a,\Tan}(y)$
  will have a upper horizontal circle labeled $X$, so
  $C(B(A_{a,\Tan}(y)))=0$. If the circle $Z$ passes through the lower
  platforms, notice that at some point in the saddles corresponding to
  $A_{a,\Tan}$, $Z$ either splits into two essential circles labeled
  $X$ or merges with an essential circle labeled $1$ to form an
  essential circle labeled $X$. In either case, the annular filtration
  strictly decreases, so $B(A_{a,\Tan}(y))=0$.
  
  Next, suppose $(a\ou{\Tan}\Wmirror{a},y)$ has a type III circle $Z$
  passing through the upper platforms. If $Z$ is labeled $X$ then
  $A_{a,\Tan}(y)$ will have an upper horizontal circle labeled $X$, so
  $C(B(A_{a,\Tan}(y)))=0$. More generally, let $P$ and $Q$ be two
  points on the intersection of $Z$ and one of the upper
  platforms. During the saddles in $A_{a,\Tan}$, if a circle
  containing $P$ or $Q$ is ever labeled $X$ then we have
  $C(B(A_{a,\Tan}(y)))=0$. However, since $P$ and $Q$ end up on
  separate circles, at some point the saddle cobordism must be a split
  with $P$ and $Q$ ending on opposite components. Since the split map
  sends $1$ to $1\otimes X+X\otimes 1$, one of these components will
  be labeled $X$.

  Finally, suppose $(a\ou{\Tan}\Wmirror{a},y)$ has a type III circle
  $Z$ passing through the lower platforms. If $Z$ is labeled $X$, the
  same analysis as in the type II circle case implies
  $C(B(A_{a,\Tan}(y)))=0$. So, suppose $Z$ is labeled $1$. Let $P$ and
  $Q$ be two points at the intersection of $Z$ and a lower
  platform. Eventually, both $P$ and $Q$ must lie on circles labeled
  $X$, or else $C(B(A_{a,\Tan}(y)))=0$. If $P$ (or $Q$) is ever on an
  inessential circle labeled $X$ then $P$ cannot later be on an
  essential circle labeled $X$ without decreasing the annular
  filtration. Similarly, if $P$ (or $Q$) is ever on an essential
  circle labeled $1$ then $P$ can never later be on an essential
  circle labeled $X$. If $P$ and $Q$ are on the same essential circle
  labeled $X$, then (using the previous two observations) there is no
  way for $P$ and $Q$ to end up on different essential circles labeled
  $X$. But now we have ruled out all possibilities: $P$ and $Q$ start
  on the same inessential circle labeled $1$, and the only changes
  that can happen are for them to next be on the same inessential
  circle labeled $X$, the same essential circle labeled $1$, the same
  essential circle labeled $X$, or different essential circles one of
  which is labeled $1$.
\end{proof}

\begin{definition}\label{def:BPW-gluing-map}
  By Lemma~\ref{lem:CBA-descends}, $C\circ B\circ A$ descends to a map
  \begin{equation}\label{eq:BPW-gluing-map-1}
    \Xi_0\co \bigoplus_{a\in\rCrossinglessB{n}{\bbk}}\CKTfuncB{\Tan}{\bbk}{\bbk}(a,a)\to \AKhCx(\anclose{\Tan};n-2k)\{n-2k\}
  \end{equation}
  which we call the \emph{annular gluing map}.
\end{definition}
If $\HC_*(\CKTalgB{n}{\bbk};\CKTfuncB{\Tan}{\bbk}{\bbk})$ denotes the standard Hochschild
complex, which is the total complex of the bicomplex
\begin{align*}
  \cdots&\to
          \bigoplus_{a_1,a_2,a_3\in\rCrossinglessB{n}{\bbk}}\CKTfuncB{\Tan}{\bbk}{\bbk}(a_1,a_2)\otimes_\ZZ\CKTalgB{n}{\bbk}(a_2,a_3)\otimes_\ZZ\CKTalgB{n}{\bbk}(a_3,a_1)
  \\          
  &\to \bigoplus_{a_1,a_2\in\rCrossinglessB{n}{\bbk}}\CKTfuncB{\Tan}{\bbk}{\bbk}(a_1,a_2)\otimes_\ZZ \CKTalgB{n}{\bbk}(a_2,a_1)
  \to
    \bigoplus_{a_1\in\rCrossinglessB{n}{\bbk}}\CKTfuncB{\Tan}{\bbk}{\bbk}(a_1,a_1),
\end{align*}
then there is an induced map
\[
  \Xi\co \HC_*(\CKTalgB{n}{\bbk};\CKTfuncB{\Tan}{\bbk}{\bbk})\to
  \AKhCx(\anclose{\Tan};n-2k)\{n-2k\}
\]
defined by projecting to $\CKTfuncB{\Tan}{\bbk}{\bbk}$ and then applying the
map $\Xi_0$ from Equation~\eqref{eq:BPW-gluing-map-1}.
\begin{lemma}\label{lem:Xi-chain-map}
  The map $\Xi$ is a chain map.
\end{lemma}
\begin{proof}
  This follows from the fact that $\Xi_0$ is a chain map (since $A$ is
  a chain map by Item~\ref{item:A-chain-map} of
  Proposition~\ref{prop:A-map} and it is immediate from their
  definitions that $B$ and $C$ are chain maps), and
  Item~\ref{item:A-coinv} of Proposition~\ref{prop:A-map}, which
  implies that $A$ vanishes on the image of $\bigoplus_{a_1,a_2\in\rCrossinglessB{n}{\bbk}}\CKTfuncB{\Tan}{\bbk}{\bbk}(a_1,a_2)\otimes_\ZZ \CKTalgB{n}{\bbk}(a_2,a_1)$.
\end{proof}

\begin{lemma}\label{lem:Xi-trace}
  Let $\Tan_1$ be an $(m,n)$-tangle and $\Tan_2$ and $(n,m)$-tangle. Then for any $h,k$ with $m-n=2(h-k)$, the following diagram commutes:
  \[
    \xymatrix{
      \HC_*(\CKTalg{m}{m-h}{h};\CKTfunc{\Tan_1\Tan_2}{m-h}{h}{m-h}{h})\ar[r]^-{\simeq}\ar[d]_{\Xi} & \HC_*(\CKTalg{n}{n-k}{k};\CKTfunc{\Tan_2\Tan_1}{n-k}{k}{n-k}{k})\ar[d]^{\Xi}\\
      \AKhCx(\anclose{\Tan_1\Tan_2};m-2h)\{m-2h\}\ar[r]_-{\cong} & \AKhCx(\anclose{\Tan_2\Tan_1};n-2k)\{n-2k\}.
    }
  \]
  Here, the top horizontal map is induced by Theorem~\ref{thm:CK-pairing} and
  cyclic symmetry of Hochschild homology and the bottom by the fact that the
  closures of $\Tan_1\Tan_2$ and $\Tan_2\Tan_1$ are isotopic links (in fact,
  link diagrams) in the annulus. (Note that $m-2h=n-2k$.)
\end{lemma}
\begin{proof}
  This follows from the definitions and Item~\ref{item:A-trace} of
  Proposition~\ref{prop:A-map}.
\end{proof}

\begin{lemma}\label{lem:HH-remove-U}
  Let $\Tan$ be an $(n,n)$-tangle diagram which is the union of an
  $(n,n)$-tangle $\Tan'$ and an unknotted circle $U$ disjoint from $\Tan'$. Then
  the following diagram commutes:
  \[
    \xymatrix{
      \HC_*(\CKTalgB{n}{\bbk};\CKTfuncB{\Tan}{\bbk}{\bbk})\ar[r]^-\cong\ar[d]_{\Xi} & \HC_*(\CKTalgB{n}{\bbk};\CKTfuncB{\Tan'}{\bbk}{\bbk})\otimes V\ar[d]^{\Xi\otimes\Id}\\
      \AKhCx(\anclose{\Tan})\{n-2k\}\ar[r]_-\cong & \AKhCx(\anclose{\Tan'})\{n-2k\}\otimes V.
    }
  \]
  Here, the top horizontal map is induced by the obvious isomorphism $\CKTfuncB{\Tan}{\bbk}{\bbk}\cong \CKTfuncB{\Tan'}{\bbk}{\bbk}\otimes V$ and the bottom horizontal map is also the obvious isomorphism.
\end{lemma}
\begin{proof}
  This is immediate from the definitions.
\end{proof}

\begin{lemma}\label{lem:HH-id-braid}
  Suppose $\Tan$ is the $(n,n)$-tangle diagram consisting of $n$ horizontal
  strands. Then the map
  $\Xi_*\co \HH_*(\CKTalgB{n}{\bbk};\CKTfuncB{\Tan}{\bbk}{\bbk})\to
  \AKh(\anclose{\Tan};n-2k)\{n-2k\}$ is an isomorphism.
\end{lemma}
\begin{proof}
  Let $\mc{S}$ be the set of all subsets of $\{1,2,\dots,n\}$ of size
  $k$. Define a partial order $\preceq$ on $\mc{S}$ by declaring
  $S=\{s_1<s_2<\dots<s_k\}\preceq T=\{t_1<t_2<\dots<t_k\}$ if $s_i\leq
  t_i$ for all $1\leq i\leq k$.

  Since $\Tan$ is a flat tangle,
  $\AKh(\anclose{\Tan};n-2k)=\AKhCx(\anclose{\Tan};n-2k)$. Moreover,
  in order to have winding number grading $n-2k$, exactly $k$ of the
  $n$ circles in $\anclose{\Tan}$ have to be labeled $X$. Therefore,
  after numbering the strands of $T$ by $1,2,\dots,n$ from bottom to
  top, $\AKh(\anclose{\Tan};n-2k)$ can be identified with the free
  $\ZZ$-module generated by $\mc{S}$: the generator corresponding to
  $S\in\mc{S}$ labels the circles in $S$ by $X$, and the remaining
  circles by $1$. We will view $\AKh(\anclose{\Tan};n-2k)$ as a filtered group, with the
  filtration given by the partial order $\preceq$ on $\mc{S}$.

  Since Beliakova-Putyra-Wehrli have already established that
  $\HH_*(\CKTalgB{n}{\bbk};\CKTfuncB{\Tan}{\bbk}{\bbk})\cong
  \AKh(\anclose{\Tan};n-2k)\{n-2k\}$, it is enough to show that the map
  $\Xi_*\co \HH_*(\CKTalgB{n}{\bbk};\CKTfuncB{\Tan}{\bbk}{\bbk})\to
  \AKh(\anclose{\Tan};n-2k)\{n-2k\}$ is surjective. Since $T$ is a flat
  tangle, the chain complex
  $\bigoplus_{a\in\rCrossinglessB{n}{\bbk}}\CKTfuncB{\Tan}{\bbk}{\bbk}(a,a)$
  has no differential, so it is enough to show that the annular gluing
  map $\Xi_0$ from Definition~\ref{def:BPW-gluing-map} is
  surjective. Given $a\in\rCrossinglessB{n}{\bbk}$, let $y_a$ be the
  generator of $\CKTfuncB{\Tan}{\bbk}{\bbk}(a,a)$ where each circle of
  $a\ou{\Tan}\Wmirror{a}$ is labeled $1$. Let
  \[
    M=\langle\{y_a\mid a\in \rCrossinglessB{n}{\bbk}\}\rangle\subset \bigoplus_{a\in\rCrossinglessB{n}{\bbk}}\CKTfuncB{\Tan}{\bbk}{\bbk}(a,a).
  \]
  We will show that $\Xi_0|_M$ is an isomorphism (and therefore,
  $\Xi_0$ is surjective).

  Recall that $\rCrossinglessB{n}{\bbk}$ is in canonical bijection
  with $\mc{S}$~\cite[\S6]{CK-kh-tangle}, as follows: for any
  $a\in\rCrossinglessB{n}{\bbk}$, the corresponding element
  $S_a\in\mc{S}$ is the subset of the $n$ non-platform points
  (numbered $1,2,\dots,n$ from bottom to top) which are matched to a
  higher point by $a$. Therefore, $\preceq$ induces a filtration
  $\preceq$ on $M$ by $y_a\preceq y_b$ if and only if $S_a\preceq S_b$.

  So, the generators of $M=\ZZ\langle \{y_a\}\rangle$ and
  $\AKh(\anclose{\Tan};n-2k)=\ZZ\langle\{S_a\}\rangle$ are in
  bijection with each other, via $y_a\leftrightarrow S_a$.  We will
  prove $\Xi_0|_M$ is a filtered map, and the associated graded piece
  of the map sends each generator of $M$ to the corresponding
  generator of $\AKh(\anclose{\Tan};n-2k)$; it follows that $\Xi_0|_M$
  is an isomorphism.
  
  Consider a generator $y_a$ of $M$.
  Recall that $\Xi_0$ is a composition of
  three maps, $C\circ B\circ A$. The map $A$ is a composition of $n$
  splits, each splitting a non-essential circle labeled
  $1$. Therefore, $A$ preserves the winding number grading, and so we
  do not need the map $B$. The map $C$ projects onto the summand where
  each circle passing through the lower (respectively upper) platform
  is labeled $X$ (respectively $1$). 

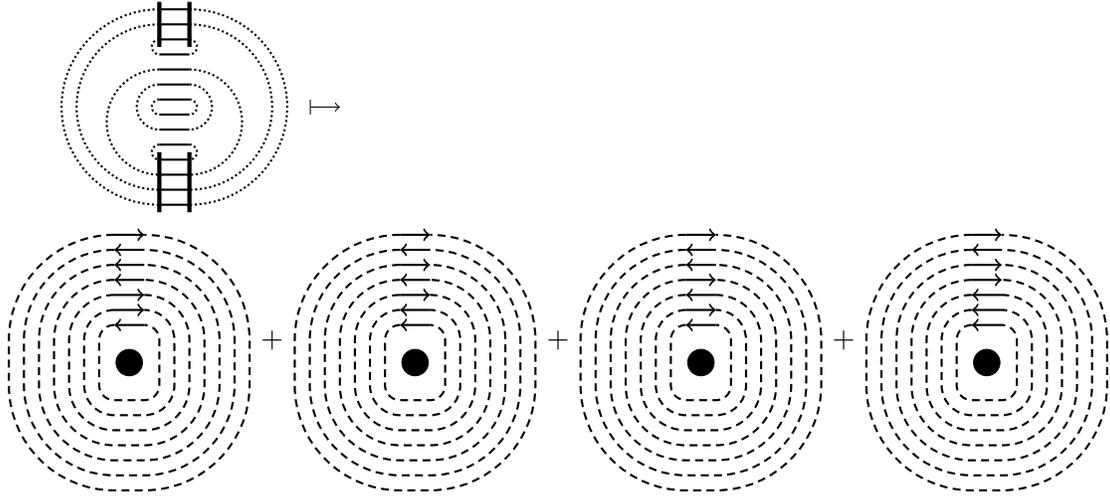
\begin{figure}
  \begin{tikzpicture}[scale=0.2]
    \tikzset{platform/.style={ultra thick}}
    \tikzset{matchings/.style={thick,densely dotted}}
    \tikzset{tangle/.style={thick}}
    \tikzset{anclose/.style={thick,densely dashed}}

    \begin{scope}[yshift=12cm,xshift=12cm]
      \foreach \i in {-3,-2,...,10}{
        \draw[tangle] (0,\i) --++(2,0);
      }

      \foreach \i in {0,2}{
        \draw[platform] (\i,-3.5)--(\i,0.5);
        \draw[platform] (\i,7.5)--(\i,10.5);
      }
      
      \foreach \d/\j in {1/4,3/5,1/8,1/1,7/6,11/9,13/10}{
        \draw[matchings] (0,\j) arc (90:270:0.5*\d cm);
        \draw[matchings] (2,\j) arc (90:-90:0.5*\d cm);
      }
    \draw[|->] (10,3.5)--++(2,0);
    \end{scope}

    \foreach \i in {1,2,3}{
      \node at (0.5cm+19*\i cm,0) {$+$};}

    \foreach \a/\b/\c/\d [count=\cc from 1] in {2/3/4/5,5/3/4/2,2/4/3/5,5/4/3/2}{

    \begin{scope}[xshift=-10cm+19*\cc cm]
      \draw[tangle,<-] (0,1) --++(2,0);
      \draw[tangle,<-] (0,6) --++(2,0);
      \draw[tangle,->] (0,7) --++(2,0);
      
      \draw[tangle,->] (0,\a) --++(2,0);
      \draw[tangle,->] (0,\b) --++(2,0);
      \draw[tangle,<-] (0,\c) --++(2,0);
      \draw[tangle,<-] (0,\d) --++(2,0);

      \foreach \j in {1,2,...,7}{
        \draw[anclose] (0,\j) arc (90:180:\j cm) --++(0,-3) arc (180:270:\j cm) --++(2,0) arc(-90:0:\j cm)--++(0,3) arc(0:90:\j cm);
      }
      
      \draw[fill] (1,-1.5) circle [radius=.875];

    \end{scope}
  }

  \end{tikzpicture}
  \caption{\textbf{An example of the map $\Xi_0|_M$.} Here, $n=7,k=3$
    and we are starting with generator $y_a$ corresponding to the subset
    $S_a=\{2,3,7\}$. The map sends $y_a$ to a sum of four generators,
    corresponding to the subsets $\{2,3,7\}$, $\{3,5,7\}$,
    $\{2,4,7\}$, and $\{4,5,7\}$. (Circles labeled $1$ are oriented
    counter-clockwise and circles labeled $X$ are oriented
    clockwise. The arcs of $\ou{\Tan}$ are solid, the platforms
    are thick, the arcs in $a$ and $\Wmirror{a}$ are
    dotted, and the new arcs in the annular closure $\anclose{\Tan}$
    are dashed.)}\label{fig:identity-tangle-isomorphism}
\end{figure}

  Circles in $a\ou{\Tan}\Wmirror{a}$ are of the following four types.
  \begin{itemize}[leftmargin=*]
  \item Circles $Z$ that pass through both the upper and the lower
    platform. Under the map $\Xi_0$, $Z$ splits into an upper
    horizontal circle labeled $1$ and a lower horizontal circle
    labeled $X$, which are then forgotten.
  \item Circles $Z$ that pass through only the upper platform. Assume
    $Z$ contains the $i\th$ strand of $T$, and consequently, $i\in
    S_a$. Under the map $\Xi_0$, $Z$ splits into an upper horizontal
    circle labeled $1$ (which is then forgotten) and the $i\th$
    component of $\anclose{\Tan}$ labeled $X$.
  \item Circles $Z$ that pass through only the lower platform. Assume
    $Z$ contains the $i\th$ strand of $T$, and consequently, $i\in
    \{1,2,\dots,n\}\setminus S_a$. Under the map $\Xi_0$, $Z$ splits
    into a lower horizontal circle labeled $X$ (which is then
    forgotten) and the $i\th$ component of $\anclose{\Tan}$ labeled
    $1$.
  \item Circles $Z$ that are disjoint from the platforms. Assume $Z$
    contains the $i\th$ and $j\th$ strand of $T$, with $i<j$, and
    consequently, $i\in S_a$ and $j\in\{1,2,\dots,n\}\setminus
    S_a$. Under the map $\Xi_0$, $Z$ splits into the $i\th$ and $j\th$
    component of $\anclose{\Tan}$, one labeled $1$ and the other
    labeled $X$.
  \end{itemize}
  In the first three cases, the map $\Xi_0$ sends the (component of the) generator $y_a$
  to the corresponding (component of the) generator $S_a$. In the last case, $\Xi_0$
  sends $y_a$ to a sum of two generators---one corresponding to the
  same subset $S_a$ and one corresponding to
  $T=S_a\cup\{j\}\setminus\{i\}$---and we have $S_a\prec
  T$. Therefore, the map increases or preserves the filtration given
  by $\preceq$, and the associated graded piece of the map sends each
  generator to the corresponding generator. See
  Figure~\ref{fig:identity-tangle-isomorphism} for an example of this
  map.
\end{proof}

\begin{theorem}\label{thm:Xi-is-BPW}
  The map $\Xi$ induces the isomorphism from Theorem~\ref{thm:BPW}.
\end{theorem}
\begin{proof}
  First, assume that $\Tan$ is a flat $(n,n)$-tangle. We prove the
  result by induction on $n$. By Lemma~\ref{lem:HH-remove-U}, we can
  assume that $\Tan$ has no closed components. So, if $n=0$, $\Tan$ is
  empty and the result is trivial. Next, for general $n$, if $\Tan$ is
  the identity braid, the result is
  Lemma~\ref{lem:HH-id-braid}. Otherwise, we can decompose $\Tan$ as
  $\Tan_1\Tan_2$ where $\Tan_1$ is a flat $(n,m)$-tangle, $\Tan_2$ is
  a flat $(m,n)$-tangle, and $m<n$. By Lemma~\ref{lem:Xi-trace}, $\Xi$
  is an isomorphism for $\Tan_1\Tan_2$ if and only if $\Xi$ is an
  isomorphism for $\Tan_2\Tan_1$ which is true by induction.

  Next, for a general tangle $\Tan$, note that each generator of
  $\HC_*(\CKTalgB{n}{\bbk};\CKTfuncB{\Tan}{\bbk}{\bbk})$ lies over some vertex
  $v$ of the cube.  Consider the filtrations on
  $\HC_*(\CKTalgB{n}{\bbk};\CKTfuncB{\Tan}{\bbk}{\bbk})$ and
  $\AKhCx(\anclose{\Tan};n-2k)$ by $|v|$, the grading
  on the cube. The map $\Xi$ respects this filtration and, by the previous
  case, induces an isomorphism at the $E_1$-page of the associated
  spectral sequence. Thus, $\Xi$ is a quasi-isomorphism, as desired.
\end{proof}

We give a spectral refinement of this result. Before stating the main theorem, we observe:
\begin{lemma}
  The topological Hochschild homology of $\CKTSpecBimB{\Tan}{\bbk}{\bbk}$ is an
  invariant of the annular closure of $\Tan$.
\end{lemma}
\begin{proof}
  It suffices to verify that $\THH(\CKTSpecBimB{\Tan}{\bbk}{\bbk})$ is invariant
  under Reidemeister moves and cyclic rotation of $\Tan$. Invariance
  under Reidemeister moves is Theorem~\ref{thm:invariance}. The fact that
  \[
    \THH(\CKTSpecBimB{\Tan_1\Tan_2}{\bbk}{\bbk})\simeq \THH(\CKTSpecBimB{\Tan_2\Tan_1}{\bbk}{\bbk})
  \]
  follows from Theorem~\ref{thm:pairing} and the fact that topological
  Hochschild homology is a trace.
\end{proof}

Given a link $L\subset S^1\times\DD^2$, the winding number filtration
on $\KhCx(L\subset \RR^3)$ induces a filtration on the Khovanov
spectrum $\KhSpace{L}$. The associated graded spectrum
$\AKhSpace{L}^\ell$ in winding number grading
$\ell$ is a spectral refinement of annular Khovanov homology
$\AKh(L;\ell)$; verifying that the homotopy type of this associated
graded spectrum is an invariant of the annular link is
straightforward. (See also~\cite{SZ-kh-localization}.)

\begin{theorem}\label{thm:THH-gluing}
  There is a weak equivalence
  \[
    \THH(\CKTSpecBimB{\Tan}{\bbk}{\bbk})\simeq \AKhSpace{\anclose{\Tan}}^{n-2k}\{n-2k\}
  \]
  of bigraded spectra.
\end{theorem}

The last ingredient in the proof of Theorem~\ref{thm:THH-gluing} is a mild
extension of the divided cobordism category $\CobD$ from our previous
paper~\cite[\S3.1]{LLS-kh-tangles} to the annulus, and an extension of
the Khovanov-Burnside functor to this divided cobordism category in the
presence of platforms. To have strict identities and make composition
strictly associative while not destroying interesting topology, we
will quotient by a particular class of diffeomorphisms:
\begin{definition}
  Let $\Diff'(S^1)$ denote the group of orientation-preserving
  diffeomorphisms $\phi\co S^1\to S^1$ so that there is some
  $\epsilon=\epsilon(\phi)>0$ with $\phi|_{B_\epsilon(1)}=\Id$. (Here,
  $1\in S^1\subset\CC$ and $B_\epsilon(1)$ is an interval around $1$.)

  Let $\Diff'([0,1]\times S^1)$ denote the group of
  orientation-preserving diffeomorphisms
  $\phi\from [0,1]\times S^1\to [0,1]\times S^1$ so that there is some
  $\epsilon=\epsilon(\phi)>0$ and some $\psi_0,\psi_1\in\Diff'(S^1)$
  so that $\phi|_{[0,1]\times B_\epsilon(1)}=\Id$, and
  $\phi(p,q)=(p,\psi_0(q))$ for all $p\in[0,\epsilon)$, and
  $\phi(p,q)=(p,\psi_1(q))$ for all $p\in(1-\epsilon,1]$. (That is,
  $\phi$ is the identity near $1$ and is invariant in the
  $[0,1]$-direction near the boundary.)
\end{definition}
(Compare~\cite[Definitions 2.46 and 2.47]{LLS-kh-tangles}.)

\begin{definition}\label{def:CobD-annulus}
  Let $\annulus= S^1\times(-1,1)$ denote the annulus. The
  \emph{divided cobordism category of the annulus}, $\CobD(\annulus)$,
  is defined as follows:
  \begin{itemize}[leftmargin=*]
  \item An object of $\CobD(\annulus)$ consists of:
    \begin{itemize}[leftmargin=*]
    \item A smooth, closed $1$-manifold $Z$ embedded in $\annulus$.
    \item A collection of disjoint, closed arcs 
      $A\subset Z$ such that $I=Z\setminus A$ is also a union of disjoint arcs.
      We call components of $A$ \emph{active arcs} and components of $I$
      \emph{inactive arcs}.
    \end{itemize}
    We declare two objects $(Z,A)$ and $(Z',A)$ to be equivalent if
    there is a $\phi\in\Diff'(S^1)$ so that
    $(\phi\times \Id_{(-1,1)})(Z,A)=(Z',A')$.
  \item  A morphism from $(Z,A)$ to $(Z',A')$ is an equivalence class of
    pairs $(\Sigma,\Gamma)$ where
  \begin{itemize}[leftmargin=*]
  \item $\Sigma$ is a smoothly embedded cobordism in
    $[0,1]\times \annulus$ from $Z$ to $Z'$ which is vertical
    (invariant in the $[0,1]$-direction) near $\{0,1\}\times\annulus$
    and $[0,1]\times\{1\}\times(-1,1)$.
  \item $\Gamma\subset \Sigma$ is a collection of properly embedded
    arcs in $\Sigma$, also vertical near $\{0,1\}\times\annulus$, with
    $(\bdy A\cup \bdy A')= \bdy\Gamma$, and so that every component of
    $\Sigma\setminus \Gamma$ has one of the following forms:
    \begin{enumerate}[label=(\Roman*),leftmargin=*]
    \item\label{item:type1} A rectangle, with two sides components of
      $\Gamma$ and two sides components of $A\cup A'$.
    \item\label{item:type2} A $(2n+2)$-gon, with $(n+1)$ sides
      components of $\Gamma$, one side an arc component of $I'$, and the
      other $n$ sides arc components of $I$. (The integer $n$ is allowed
      to be zero.)
    \end{enumerate}
    We call the components of $\Gamma$ \emph{divides}.
  \end{itemize}
  The pairs $(\Sigma,\Gamma)$ and $(\Sigma',\Gamma')$ are equivalent
  if there is a $\phi\in\Diff'([0,1]\times S^1)$
  with 
  $(\phi\times\Id_{(-1,1)})(\Sigma)=\Sigma'$, and
  $(\phi\times\Id_{(-1,1)})(\Gamma)=\Gamma'$.
\item There is a unique $2$-morphism from $(\Sigma,\Gamma)$ to
  $(\Sigma',\Gamma')$ whenever (some representative of the equivalence
  class of) $(\Sigma,\Gamma)$ is isotopic to (some representative of
  the equivalence class of) $(\Sigma',\Gamma')$ rel boundary and
  $[0,1]\times\{1\}\times(-1,1)$.
\item Composition of divided cobordisms is defined as follows. Given
  $(\Sigma,\Gamma)\co (Z,A)\to(Z',A')$ and $(\Sigma',\Gamma')\co
  (Z',A')\to(Z'',A'')$, choose a representative of the equivalence
  class of $(Z',A')$ and representatives of the equivalence classes
  $(\Sigma,\Gamma)$ and $(\Sigma',\Gamma')$ which end / start at this
  representative of $(Z',A')$.  Define
  $(\Sigma',\Gamma')\circ(\Sigma,\Gamma)$ to be
  $(\Sigma'\circ\Sigma,\wt{\Gamma}'\circ\Gamma)$. The same proof from
  \cite{LLS-kh-tangles} shows that composition is well-defined.
\end{itemize}
\end{definition}

Recall that a multicategory $\Cat$ has a canonical groupoid
enrichment~\cite[\S 2.4.1]{LLS-kh-tangles}. In the case that $\Cat$ is
an ordinary category---the case of interest in this section---the
canonical groupoid enrichment $\Cat'$ has the same objects as
$\Cat$, $1$-morphisms $\Cat'(x,y)$ the set of finite sequences of
morphisms
$x\stackrel{f_1}{\longrightarrow}
z_1\stackrel{f_2}{\longrightarrow}\cdots\stackrel{f_{k-1}}{\longrightarrow}
z_k\stackrel{f_k}{\longrightarrow} y$, and a unique $2$-morphism
$(f_1,\dots,f_k)\to (g_1,\dots,g_\ell)$ whenever
$f_k\circ\cdots\circ f_1=g_\ell\circ\cdots\circ g_1$. The relevance
of this enrichment is that there is a (strict) $2$-functor
$V_{\HKKa}\co \CobD'\to\BurnsideCat$, the graded Burnside
category~\cite[\S3.4]{LLS-kh-tangles}. The definition of
$V_{\HKKa}$ extends immediately to a functor
$V_{\HKKa}\co\CobD(\annulus)'\to\BurnsideCat$.

Recall that to define a functor $\mTshape{m}{n}\to \CobD$ induced by
a tangle $\Tan$, we introduce some extra decorations on $\Tan$:
\begin{definition}\cite[\S3.3]{LLS-kh-tangles}
  A \emph{poxed tangle} is a tangle diagram $\Tan$ together with a
  collection of marked points (\emph{pox}) on the arcs in $\Tan$ so
  that for every resolution $\Tan_v$ of $\Tan$, there is at least one
  pox on each closed circle of $\Tan_v$.

  Call a poxed tangle $\Tan$ \emph{sufficiently poxed} if for every
  resolution $\Tan_v$ of $\Tan$ there is at least one pox on each
  closed circle of the annular closure $\anclose{\Tan_v}$.
\end{definition}

Given a poxed $(n,n)$-tangle $\Tan$, there is an induced functor
$\CCat{N}\ttimes \mTshape{2n}{2n}\to \CobD$ which restricts to a functor
$\CCat{N}\ttimes \mCKTshapeB{n}{n}{\bbk}{\bbk}\to
\CobD$~\cite[\S3.4]{LLS-kh-tangles}. Further, the composition
$\CCat{N}\ttimes \mTshape{2n}{2n}\to \CobD\to \mBurnside$ is independent
of the choice of pox. If $\Tan$ is sufficiently poxed, there is also
an induced functor $\CCat{N}\to \CobD(\annulus)$ coming from the
annular closure $\anclose{\Tan}$ of $\Tan$. Again, the composition
$\CCat{N}\to\CobD(\annulus)\to\mBurnside$ is independent of the choice
of pox.

\begin{proof}[Proof of Theorem~\ref{thm:THH-gluing}]
  It suffices to show that there is a map of spectra
  \[
    \THH(\CKTSpecBimB{\Tan}{\bbk}{\bbk})\to \AKhSpace{\anclose{\Tan}}^{n-2k}\{n-2k\}
  \]
  so that the induced map on Hochschild complexes agrees with the map
  from Definition~\ref{def:BPW-gluing-map}.

  Fix a collection of pox on $\Tan$ making $\Tan$ sufficiently
  poxed. When talking about $\ou{\Tan}$, below, add a pox in the
  middle of each horizontal strand added to $\Tan$.
  
  Let $\DeltaInj$ be the subcategory of the simplex category generated
  by the face maps. That is, $\DeltaInj$ has objects the positive
  integers and $\Hom(p,q)$ the order-preserving injections
  $\{0,\dots,p-1\}\into \{0,\dots,q-1\}$.
  Then the topological Hochschild homology
  $\THH(\CKTSpecCatB{n}{\bbk};\CKTSpecBimB{\Tan}{\bbk}{\bbk})$ is the
  homotopy colimit of a diagram $\DeltaInj^\op\to\GrSpectra$.

  We will reformulate this homotopy colimit over a larger
  diagram, but first we need some more notation from the guts of the
  construction. Recall that $\CCat{N}_+$ is the result of adding one
  object $*$ to $\CCat{N}$ and a morphism $v\to *$ for each vertex $v$
  except $\vec{1}$. Given a functor $G\co \CCat{N}\to\GrSpectra$, we can
  extend $G$ to a functor $G_+\co \CCat{N}_+\to\GrSpectra$ by declaring
  that $G(*)=\{\pt\}$. Then $\hocolim G_+$ is the iterated mapping
  cone of $G$. Consider the functor
  \[
    (\overline{K\circ \mCKTinvB{\Tan}{\bbk}{\bbk}})_+\co
    \CCat{N}_+\ttimes\SmCKTshapeB{n}{n}{\bbk}{\bbk}\to\GrSpectra
  \]
  where $K$ is Elmendorf-Mandell's $K$-theory,
  $\mCKTinvB{\Tan}{\bbk}{\bbk}$ is as in Section~\ref{sec:spectral-CK},
  and the bar denotes applying Elmendorf-Mandell's rectification
  construction, to obtain an honest functor. To shorten notation, let
  \[
    \oCKTSpecBimB{\Tan_v}{\bbk}{\bbk}(a_1,a_2)=(\overline{K\circ \mCKTinvB{\Tan}{\bbk}{\bbk}})_+(v,a_1,T,a_2).
  \]
  (The spectrum $\oCKTSpecBimB{\Tan_v}{\bbk}{\bbk}(a_1,a_2)$ is
  weakly equivalent to $\CKTSpecBimB{\Tan_v}{\bbk}{\bbk}(a_1,a_2)\{-|v|-N_++2N_-\}$;
  the difference arises based on when the rectification
  construction was applied.)
  
  Now, let $\Cat$ be the category with one object
  for every finite sequence
  $a_1,\dots,a_\alpha\in\rCrossinglessB{n}{\bbk}$, $\alpha\geq1$, and with a
  unique morphism
  $(a_1,\dots,a_\alpha)\to (a_1,\dots,\widehat{a_i},\dots,a_\alpha)$ for each
  $i$, composing in the obvious way. Define a functor
  $F\co\CCat{N}_+\times \Cat\to \GrSpectra$ (where $\CCat{N}$ is viewed as
  a $1$-category) by declaring that:
  \begin{itemize}[leftmargin=*]
  \item For $v\in\Ob(\CCat{N})$ and
    $a_1,\dots,a_\alpha\in\rCrossinglessB{n}{\bbk}$,
    \[
      F(v,a_1,\dots,a_\alpha)=\oCKTSpecBimB{\Tan_v}{\bbk}{\bbk}(a_\alpha,a_1)\smas\CKTSpecCatB{n}{\bbk}(a_1,a_2)\smas\cdots\smas\CKTSpecCatB{n}{\bbk}(a_{\alpha-1},a_\alpha).
    \]
  \item $F(*,a_1,\dots,a_\alpha)=\{\pt\}$.
  \item $F$ sends a morphism $v\to w$ in $\CCat{N}$ to the
    map 
    from $\oCKTSpecBimB{\Tan_v}{\bbk}{\bbk}(a_\alpha,a_1)$ to
    $\oCKTSpecBimB{\Tan_w}{\bbk}{\bbk}(a_\alpha,a_1)$ from the functor
    $(\overline{K\circ \mCKTinvB{\Tan}{\bbk}{\bbk}})_+$,
    smashed with the identity map in the other factors.
  \item $F$ sends the morphism
    $(v,a_1,\dots,a_\alpha)\to
    (v,a_1,\dots,\widehat{a_i},\dots,a_\alpha)$ to the multiplication
    map
    \begin{align*}
      \CKTSpecCatB{n}{\bbk}(a_{i-1},a_i)\smas\CKTSpecCatB{n}{\bbk}(a_i,a_{i+1}) &\to \CKTSpecCatB{n}{\bbk}(a_{i-1},a_{i+1}) & &\text{if }1<i<\alpha\\
      \shortintertext{or maps from the functor $(\overline{K\circ \mCKTinvB{\Tan}{\bbk}{\bbk}})_+$}
      \oCKTSpecBimB{\Tan_v}{\bbk}{\bbk}(a_\alpha,a_1)\smas\CKTSpecCatB{n}{\bbk}(a_1,a_2) &\to \oCKTSpecBimB{\Tan_v}{\bbk}{\bbk}(a_\alpha,a_2) & &\text{if } i=1\\
      \CKTSpecCatB{\alpha}{\bbk}(a_{\alpha-1},a_\alpha)\smas\oCKTSpecBimB{\Tan_v}{\bbk}{\bbk}(a_\alpha,a_1) &\to \oCKTSpecBimB{\Tan_v}{\bbk}{\bbk}(a_{\alpha-1},a_1) & &\text{if } i=\alpha,  
    \end{align*}
    smashed with the identity map in the remaining factors.
  \end{itemize}
  The homotopy colimit of $F$, shifted down by $N_+$, is clearly equivalent to the topological
  Hochschild homology
  $\THH(\CKTSpecCatB{n}{\bbk};\CKTSpecBimB{\Tan}{\bbk}{\bbk})$.

  The advantage of the reformulation in terms of $F$ is that $F$
  factors through the divided cobordism category of the
  annulus. Specifically, there is a functor
  $G\co (\CCat{N}\times\Cat)'\to\CobD(\annulus)'$
  which sends an object
  $(v,a_1,\dots,a_\alpha)$ to the $1$-manifold
  \[
    a_{\alpha}\ou{\Tan_v}\Wmirror{a_1} \amalg a_1\Wmirror{a_2}\amalg\cdots\amalg a_{\alpha-1}\Wmirror{a_\alpha},
  \]
  embedded in $\annulus$ so that the middle of $\ou{\Tan_v}$ is on the
  line $\{1\}\times(-1,1)$ and the disjoint unions are in the cyclic
  order shown. The set $A$ is the union of:
  \begin{itemize}[leftmargin=*]
  \item a small closed neighborhood of the crossings labeled $0$ in
    $v$,
  \item a small closed neighborhood of each pox in $\Tan$,
    and
  \item the complement in each $a_i$ (respectively $\Wmirror{a_i}$) of
    a neighborhood of the boundary. 
  \end{itemize}
  Each morphism in $(\CCat{N}\times\Cat)$ is sent to a
  composition of saddle cobordisms (cf.~\cite[\S3.3--3.4]{LLS-kh-tangles}).

  Consider the composition $V_{\HKKa}\circ G\co (\CCat{N}\times\Cat)'\to
  \BurnsideCat$. Define a functor $L\co
  (\CCat{N}\times\Cat)'\to \BurnsideCat$ by declaring
  that for an object $(v,a_1,\dots,a_\alpha)$,
  $L(v,a_1,\dots,a_\alpha)=\emptyset$ if there is a type III circle in
  any of $a_{\alpha}\ou{\Tan_v}\Wmirror{a_1}$, $a_1\Wmirror{a_2}$,
  \dots, $a_{\alpha-1}\Wmirror{a_\alpha}$. Otherwise,
  after shifting quantum grading by $n-|v|-N_++2N_-$,
  $L(v,a_1,\dots,a_\alpha)$ is the set of elements $y\in
  V_{\HKKa}(G(v,a_1,\dots,a_\alpha))$ which label each type II circle
  by $1$. On morphisms, $L$ is obtained by restricting $V_{\HKKa}\circ
  G$.  It follows from Lemma~\ref{lem:mI-absorbing} that this defines
  a $2$-functor.

  Clearly, composing the Elmendorf-Mandell $K$-theory functor with
  $L$, rectifying, and adding a basepoint $*$ to $\CCat{N}$, gives a diagram
  equivalent to $F$. In particular, $\hocolim((\ol{K\circ L})_+)\simeq \hocolim
  F$, desuspended $N_+$ times, is the topological Hochschild
  homology.

  Let $\Cat^T$ be the result of adding a terminal object $T$ to
  $\Cat$. Extend $G$ to a functor $G^T\co (\CCat{N}\times\Cat^T)'\to
  \CobD(\annulus)'$ by declaring that $G^T(v,T)=\anclose{\Tan}$, and
  $G^T$ sends each morphism to the corresponding saddle
  cobordism. Define $L^T\co\Cat^T\to\BurnsideCat$ as follows. On
  $(\CCat{N}\times\Cat)'$, $L^T|_{\CCat{N}\times\Cat}=L$. After
  shifting quantum grading by $-|v|-N_++2N_-$, define $L^T(v,T)\subset
  V_{\HKKa}(G^T(v,T))$ to be those elements which
  \begin{itemize}[leftmargin=*]
  \item have annular filtration $0$,
  \item label every lower horizontal circle $X$, and
  \item label every upper horizontal circle $1$.
  \end{itemize}
  On morphisms, define $L^T$ to be the restriction of $V_{\HKKa}\circ
  G^T$.  It follows from the proof of Lemma~\ref{lem:CBA-descends}
  that this defines a $2$-functor. 

  Consider the functor
  $K\circ L^T\co(\CCat{N}\times\Cat^T)'\to\GrSpectra$, and let $\ol{K\circ L^T}$
  be its rectification. By definition, $\hocolim((\ol{K\circ
  L^T}|_{\CCat{N}\times\{T\}})_+)$, desuspended $N_+$ times, is
  $\AKhSpace{\anclose{\Tan}}^{n-2k}\{n-2k\}$. This leads to a cofibration sequence
  \[
  \Sigma^{N_+}\AKhSpace{\anclose{\Tan}}^{n-2k}\{n-2k\}\to\hocolim (\ol{K\circ L^T})_+\to \Sigma\hocolim (\ol{K\circ L})_+\simeq \Sigma\hocolim F.
  \]  
  Further, by construction the map
  $H_*(\Sigma\hocolim F)\to H_*(\Sigma^{N_++1}\AKhSpace{\anclose{\Tan}}^{n-2k}\{n-2k\})$
  induced by the Puppe construction is the map $\Xi_*$ from
  Theorem~\ref{thm:Xi-is-BPW}. Thus, the Puppe map is a weak
  equivalence
  $\hocolim F\simeq \Sigma^{N_+}\AKhSpace{\anclose{\Tan}}^{n-2k}\{n-2k\}$. Since
  $\hocolim F\simeq
  \Sigma^{N_+}\THH(\CKTSpecCatB{n}{\bbk};\CKTSpecBimB{\Tan}{\bbk}{\bbk})$, this proves
  the result.
\end{proof}

\begin{corollary}
  The action of the Hochschild cohomology of $\CKTalgB{n}{\bbk}$ on
  the annular Khovanov homology
  $\AKh(\anclose{\Tan};n-2k)\{n-2k\}\cong
  \HH_*(\CKTalgB{n}{\bbk};\CKTfuncB{\Tan}{\bbk}{\bbk})$ satisfies a Cartan
  formula with respect to the action by Steenrod operations. For
  example, with mod-2 coefficients,
  for $a\in \HH^*(\CKTalgB{n}{\bbk})$ and
  $\beta\in \HH_*(\CKTalgB{n}{\bbk};\CKTfuncB{\Tan}{\bbk}{\bbk})$,
  \[
    \Sq^n(a\cdot\beta)=\sum_{i+j=n}\Sq^i(a)\cdot\Sq^j(\beta).
  \]
\end{corollary}

\vspace{-0.3cm}
\bibliographystyle{myalpha}
\bibliography{newbibfile}

\providecommand{\bysame}{\leavevmode\hbox to3em{\hrulefill}\thinspace}
\providecommand{\MR}{\relax\ifhmode\unskip\space\fi MR }
\providecommand{\MRhref}[2]{%
  \href{http://www.ams.org/mathscinet-getitem?mr=#1}{#2}
}
\providecommand{\href}[2]{#2}
\begin{thebibliography}{AGW15}

\bibitem[AGW15]{AGW-kh-HH}
Denis Auroux, J.~Elisenda Grigsby, and Stephan~M. Wehrli, \emph{Sutured
  {K}hovanov homology, {H}ochschild homology, and the {O}zsv\'ath-{S}zab\'o
  spectral sequence}, Trans. Amer. Math. Soc. \textbf{367} (2015), no.~10,
  7103--7131. \MR{3378825}

\bibitem[APS04]{APS-kh-surfaces}
Marta~M. Asaeda, J\'ozef~H. Przytycki, and Adam~S. Sikora,
  \emph{Categorification of the {K}auffman bracket skein module of
  {$I$}-bundles over surfaces}, Algebr. Geom. Topol. \textbf{4} (2004),
  1177--1210. \MR{2113902}

\bibitem[Bal11]{Baldwin-hf-s-seq}
John~A. Baldwin, \emph{On the spectral sequence from {K}hovanov homology to
  {H}eegaard {F}loer homology}, Int. Math. Res. Not. IMRN (2011), no.~15,
  3426--3470. \MR{2822178 (2012g:57021)}

\bibitem[Bar05]{Bar-kh-tangle-cob}
Dror Bar-Natan, \emph{Khovanov's homology for tangles and cobordisms}, Geom.
  Topol. \textbf{9} (2005), 1443--1499. \MR{2174270 (2006g:57017)}

\bibitem[BM12]{BM-top-spectral}
Andrew~J. Blumberg and Michael~A. Mandell, \emph{Localization theorems in
  topological {H}ochschild homology and topological cyclic homology}, Geom.
  Topol. \textbf{16} (2012), no.~2, 1053--1120. \MR{2928988}

\bibitem[BPW19]{BPW-Kh-HH}
Anna Beliakova, Krzysztof~K. Putyra, and Stephan~M. Wehrli, \emph{Quantum link
  homology via trace functor {I}}, Invent. Math. \textbf{215} (2019), no.~2,
  383--492. \MR{3910068}

\bibitem[BS11]{BS-kh-tangle}
Jonathan Brundan and Catharina Stroppel, \emph{Highest weight categories
  arising from {K}hovanov's diagram algebra {I}: cellularity}, Mosc. Math. J.
  \textbf{11} (2011), no.~4, 685--722, 821--822. \MR{2918294}

\bibitem[CK14]{CK-kh-tangle}
Yanfeng Chen and Mikhail Khovanov, \emph{An invariant of tangle cobordisms via
  subquotients of arc rings}, Fund. Math. \textbf{225} (2014), no.~1, 23--44.
  \MR{3205563}

\bibitem[EM06]{EM-top-machine}
A.~D. Elmendorf and M.~A. Mandell, \emph{Rings, modules, and algebras in
  infinite loop space theory}, Adv. Math. \textbf{205} (2006), no.~1, 163--228.
  \MR{2254311 (2007g:19001)}

\bibitem[GW10]{GW-kh-sutured}
J.~Elisenda Grigsby and Stephan~M. Wehrli, \emph{Khovanov homology, sutured
  {F}loer homology and annular links}, Algebr. Geom. Topol. \textbf{10} (2010),
  no.~4, 2009--2039. \MR{2728482}

\bibitem[HKK16]{HKK-Kh-htpy}
Po~Hu, Daniel Kriz, and Igor Kriz, \emph{Field theories, stable homotopy theory
  and {K}hovanov homology}, Topology Proc. \textbf{48} (2016), 327--360.

\bibitem[HSS00]{HSS-top-symmetric}
Mark Hovey, Brooke Shipley, and Jeff Smith, \emph{Symmetric spectra}, J. Amer.
  Math. Soc. \textbf{13} (2000), no.~1, 149--208. \MR{1695653 (2000h:55016)}

\bibitem[Kho00]{Kho-kh-categorification}
Mikhail Khovanov, \emph{A categorification of the {J}ones polynomial}, Duke
  Math. J. \textbf{101} (2000), no.~3, 359--426. \MR{1740682 (2002j:57025)}

\bibitem[Kho02]{Kho-kh-tangles}
\bysame, \emph{A functor-valued invariant of tangles}, Algebr. Geom. Topol.
  \textbf{2} (2002), 665--741 (electronic). \MR{1928174 (2004d:57016)}

\bibitem[LLSa]{LLS-khovanov-product}
Tyler Lawson, Robert Lipshitz, and Sucharit Sarkar, \emph{{K}hovanov homotopy
  type, {B}urnside category, and products}, arXiv:1505.00213.

\bibitem[LLSb]{LLS-kh-tangles}
\bysame, \emph{Khovanov spectra for tangles}, arXiv:1706.02346.

\bibitem[LS14]{RS-khovanov}
Robert Lipshitz and Sucharit Sarkar, \emph{A {K}hovanov stable homotopy type},
  J. Amer. Math. Soc. \textbf{27} (2014), no.~4, 983--1042. \MR{3230817}

\bibitem[Rob13]{Roberts-kh-dcov}
Lawrence~P. Roberts, \emph{On knot {F}loer homology in double branched covers},
  Geom. Topol. \textbf{17} (2013), no.~1, 413--467. \MR{3035332}

\bibitem[Rob16a]{Roberts-kh-A-tangle}
\bysame, \emph{A type {$A$} structure in {K}hovanov homology}, Algebr. Geom.
  Topol. \textbf{16} (2016), no.~6, 3653--3719. \MR{3584271}

\bibitem[Rob16b]{Roberts-kh-tangle}
\bysame, \emph{A type {$D$} structure in {K}hovanov homology}, Adv. Math.
  \textbf{293} (2016), 81--145. \MR{3474320}

\bibitem[Str09]{St-kh-tangle}
Catharina Stroppel, \emph{Parabolic category {$\mathscr{O}$}, perverse sheaves
  on {G}rassmannians, {S}pringer fibres and {K}hovanov homology}, Compos. Math.
  \textbf{145} (2009), no.~4, 954--992. \MR{2521250}

\bibitem[SZ]{SZ-kh-localization}
Matthew Stoffregen and Melissa Zhang, \emph{Localization in {K}hovanov
  homology}, arXiv:1810.04769.

\end{thebibliography}
\vspace{1cm}
\end{document}